\documentclass[aop]{imsart}

\startlocaldefs

\usepackage{amsmath,amssymb,amsthm,tikz, nicefrac, mathrsfs,upgreek} 

\RequirePackage{color}
\RequirePackage[colorlinks, urlcolor=my-blue,linkcolor=my-red,citecolor=my-red]{hyperref}
\definecolor{my-blue}{rgb}{0.05,0.1,0.5}
\definecolor{my-red}{rgb}{0.5,0.0,0.0}

\usepackage{dsfont} 

\usepackage[utf8]{inputenc}

\usepackage{comment} 

\usetikzlibrary{shapes,arrows}

\usepackage[only,llbracket,rrbracket]{stmaryrd}

\numberwithin{equation}{section}
\newtheorem{theorem}{\color{my-blue}{\sc Theorem}}[section]
\newtheorem{lemma}[theorem]{\color{my-blue}{\sc Lemma}}
\newtheorem{proposition}[theorem]{\color{my-blue}{\sc Proposition}}
\newtheorem{corollary}[theorem]{\color{my-blue}{\sc Corollary}}

\newtheorem{definition}[theorem]{\color{my-blue}{\sc Definition}}
\newtheorem{remark}[theorem]{\color{my-blue}{\sc Remark}}

\newcommand{\TV}[1]{{\lVert #1 \rVert}_{\normalfont \text{TV}}}

\def\mix{\textup{mix}}

\def\E{\mathbb{E}}
\def\P{\mathbb{P}}
\def\N{\mathbb{N}}

\def\R{\mathbb{R}}
\def\Z{\mathbb{Z}}

\newcommand{\I}{\mathrm{i}}

\newcommand{\C}{\mathbb{C}}
\renewcommand{\Im}{\operatorname{Im}}
\newcommand{\Or}{{\mathcal O}}
\newcommand{\expon}{\bar{\kappa}}
\newcommand{\expont}{\kappa^{\prime}}

\def\Jz{\mathcal{J}^{\mathbb{Z}}}
\def\Jn{\mathcal{J}^{\mathbb{N}}}
\def\Jo{\mathcal{J}}

\def\e{\varepsilon}
\def\c{\textup{c}} 
\def\h{\textup{h}} 

\def \one{\mathbf 1}
\def \two{\mathbf 2}
\def \zero{\mathbf 0}
\def \onep{\mathbf 1'}
\def \zerop{\mathbf 0'}

\def \lbr {\llbracket}
\def \rbr {\rrbracket}
\def\modi{\textup{mod}}
\def\cou{\textup{couple}}
\def\Aup{\normalfont\textsf{A}}
\def\Bup{\normalfont\textsf{B}}
\def\Cup{\normalfont\textsf{C}}

\def\rev{\textup{r}}
\def\texit{\tau_{\textup{exit}}}

\def\Lup{\normalfont\textsf{L}}
\def\Rup{\normalfont\textsf{R}}

\renewcommand{\ln}{\log}

\DeclareMathOperator{\Var}{Var}

\allowdisplaybreaks

\endlocaldefs

\begin{document}

\begin{frontmatter}
\title{Mixing times for the open ASEP at the triple point}
\runtitle{Mixing times at the triple point}

\begin{aug}

\author[Patrik Ferrari]{\fnms{Patrik} \snm{Ferrari}\ead[label=e1]{ferrari@uni-bonn.de}}
\author[Dominik Schmid]{\fnms{Dominik} \snm{Schmid}\ead[label=e2]{d.schmid@uni-a.de}}

\address[Patrik Ferrari]{University of Bonn, Germany, \printead{e1}}
\address[Dominik Schmid]{University of Augsburg, Germany, \printead{e2}}

\end{aug}

\begin{abstract}
We consider mixing times for the open asymmetric simple exclusion process (ASEP) at the triple point. We show that the mixing time of the open ASEP on a segment of length $N$ for bias parameter $q$ is of order $N^{3/2+\kappa}$ if $1-q \asymp N^{-\kappa}$ for some $\kappa \in [0,\frac{1}{2})$, and the same result with poly-logarithmic corrections for $\kappa=\frac{1}{2}$. Our proof combines a fine analysis of the current of the open ASEP, moderate deviations of second class particles, the censoring inequality, and various couplings and multi-species extensions of the ASEP. Moreover, we establish a comparison between moderate deviations for the current of the open ASEP and the ASEP on $\Z$,
as well as bounds on mixing times for the open ASEP in the weakly high density phase, which are of independent interest.
\end{abstract}



\end{frontmatter}



\section{Introduction} \label{sec:Introduction}
\sloppy

The asymmetric simple exclusion process (ASEP) is a fundamental model in the study of interacting particle systems. Over recent decades, it has been analyzed from various  perspectives, including statistical mechanics, probability theory, and combinatorics; see  \cite{BE:Nonequilibrium,L:Book2,WBE:CombinatorialMappings} for relevant surveys. In recent years, ASEP has gained further attention as one of the best studied and rigorously established examples belonging to the Kardar–Parisi–Zhang (KPZ) universality class; see \cite{C:KPZReview,C:SurveyStationary} for an overview by Corwin.

In this work, we focus on the asymmetric simple exclusion process with open boundaries, also known as the open ASEP. This model can be intuitively described as follows: Consider a segment of length $N$, where each site is either occupied by a particle or left empty. Each site has two independent Poisson clocks with rates $1$ and $q$, where $q = q(N) \in [0,1)$. When the rate $1$ clock rings at an occupied site, the particle moves to the right if the neighboring site is empty. Similarly, when the rate $q$ clock rings at an occupied site, the particle moves to the left, provided the destination site is empty. At the left boundary, particles can enter with rate $\alpha$ or exit with rate $\gamma$, while at the right boundary, particles exit with rate $\beta $ or enter with rate $\delta$; see Figure~\ref{fig:ASEP}. Here, $\alpha, \beta, \gamma, \delta \geq 0$ are parameters which may also depend on $N$. Depending on the choice of these boundary parameters, the open ASEP exhibits three distinct phases, illustrated in Figure~\ref{fig:phase}. Our analysis focuses on a particular case called the triple point, where the three phases meet, and which is closely linked to constructing solutions to the open KPZ equation \cite{BL:StationaryKPZ,BKWW:MarkovKPZ,BWW:ASEPtoKPZ,C:SurveyStationary,CK:StationaryKPZ,CS:OpenASEPWeakly,H:EDRep}.

A topic of particular interest for exclusion processes is the analysis of mixing times. Mixing times provide a standard way for Markov chains of quantifying the speed at which a process converges to its stationary distribution; see \cite{LPW:markov-mixing} for an introduction. In the context of the open ASEP, questions about mixing times are closely linked to properties of the ASEP on $\Z$. It  is well-known that the fluctuations of the particle current, i.e., the number of particles passing through a given site in the ASEP, can be used to express the speed of perturbation in the system, usually described in terms of second class particles; see  \cite{BCS:CubeRoot}. Intuitively, mixing times reflect the time required for an  initially placed perturbation, respectively second class particle, to disappear. We will elaborate more on this connection after introducing the main model and presenting the main results.

\subsection{Model and results}

We define the \textbf{open ASEP} as a continuous-time Markov chain  $(\eta_t)_{t\geq 0}$ with state space  $\Omega_N=\{0,1\}^{N}$ for some $N \in \N$, and with generator
\begin{equation} \label{def:openASEP}
\begin{split}
\mathcal{L}f(\eta) &= \sum_{x =1}^{N-1} \big(\eta(x)(1-\eta(x+1)) + q \eta(x+1)(1-\eta(x)) \big) \left[ f(\eta^{x,x+1})-f(\eta) \right]  \\
 &+ \alpha (1-\eta(1)) \left[ f(\eta^{1})-f(\eta) \right] \hspace{2pt} + \beta \eta(N)\left[ f(\eta^{N})-f(\eta) \right]   \\
  &+ \gamma \eta(1) \left[ f(\eta^{1})-f(\eta) \right] \hspace{2pt} + \delta (1-\eta(N))\left[ f(\eta^{N})-f(\eta) \right]
  \end{split}
\end{equation}
 for all measurable functions $f\colon \Omega_{N} \rightarrow \R$. Here, we use the standard notations
\begin{equation*}
\eta^{x,y} (z) = \begin{cases}
 \eta (z) & \textrm{ for } z \neq x,y\\
 \eta(x) &  \textrm{ for } z = y\\
 \eta(y) &  \textrm{ for } z = x \,
 \end{cases} \qquad \text{ and } \qquad \eta^{a} (z) = \begin{cases}
 \eta (z) & \textrm{ for } z \neq a\\
 1-\eta(z) &  \textrm{ for } z = a \,
 \end{cases}
\end{equation*}
to denote swapping of values in a configuration $\eta \in \Omega_{N}$ at sites $x,y \in \lbr N \rbr := \{1,2,\dots,N\}$, and flipping at $a\in \lbr N \rbr$, respectively. We say that site $x$ is \textbf{occupied} if $\eta(x)=1$, and \textbf{vacant} otherwise. A visualization of this process is given in Figure \ref{fig:ASEP}.
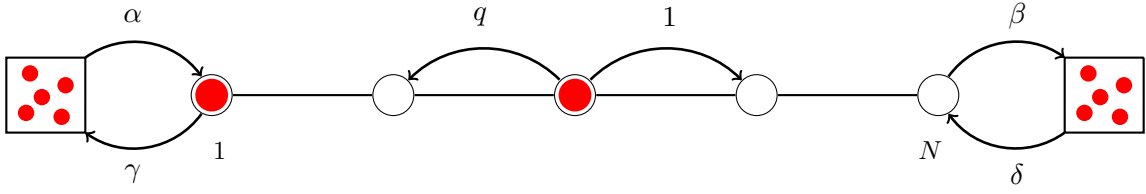
\begin{figure}
\centering
\begin{tikzpicture}[scale=1.05]

\def\spiral[#1](#2)(#3:#4:#5){
\pgfmathsetmacro{\domain}{pi*#3/180+#4*2*pi}
\draw [#1,
       shift={(#2)},
       domain=0:\domain,
       variable=\t,
       smooth,
       samples=int(\domain/0.08)] plot ({\t r}: {#5*\t/\domain})
}

\def\particles(#1)(#2){

  \draw[black,thick](-3.9+#1,0.55-0.075+#2) -- (-4.9+#1,0.55-0.075+#2) -- (-4.9+#1,-0.4-0.075+#2) -- (-3.9+#1,-0.4-0.075+#2) -- (-3.9+#1,0.55-0.075+#2);

  	\node[shape=circle,scale=0.6,fill=red] (Y1) at (-4.15+#1,0.2-0.075+#2) {};
  	\node[shape=circle,scale=0.6,fill=red] (Y2) at (-4.6+#1,0.35-0.075+#2) {};
  	\node[shape=circle,scale=0.6,fill=red] (Y3) at (-4.2+#1,-0.2-0.075+#2) {};
   	\node[shape=circle,scale=0.6,fill=red] (Y4) at (-4.45+#1,0.05-0.075+#2) {};
  	\node[shape=circle,scale=0.6,fill=red] (Y5) at (-4.65+#1,-0.15-0.075+#2) {}; }

  \def\annhil(#1)(#2){	  \spiral[black,thick](9.0+#1,0.09+#2)(0:3:0.42);
  \draw[black,thick](8.5+#1,0.55+#2) -- (9.5+#1,0.55+#2) -- (9.5+#1,-0.4+#2) -- (8.5+#1,-0.4+#2) -- (8.5+#1,0.55+#2); }

	\node[shape=circle,scale=1.5,draw] (B) at (2.3,0){} ;
	\node[shape=circle,scale=1.5,draw] (C) at (4.6,0) {};
	\node[shape=circle,scale=1.2,fill=red] (CB) at (2.3,0) {};
    \node[shape=circle,scale=1.5,draw] (A) at (0,0){} ;
 	\node[shape=circle,scale=1.5,draw] (D) at (6.9,0){} ;
 	 	\node[shape=circle,scale=1.5,draw] (Z) at (-2.3,0){} ;
	\node[shape=circle,scale=1.2,fill=red] (YZ) at (-2.3,0) {};
   \node[line width=0pt,shape=circle,scale=1.6] (B2) at (2.3,0){};
	\node[line width=0pt,shape=circle,scale=2.5] (D2) at (6.9,0){};
		\node[line width=0pt,shape=circle,scale=2.5] (Z2) at (-2.3,0){};
		
	\node[line width=0pt,shape=circle,scale=2.5] (X10) at (6.8,0){};
		\node[line width=0pt,shape=circle,scale=2.5] (X11) at (-2.2,0){};	
			\node[line width=0pt,shape=circle,scale=2.5] (X12) at (8.4,0){};
		\node[line width=0pt,shape=circle,scale=2.5] (X13) at (-3.8,0){};

		\draw[thick] (Z) to (A);	
	\draw[thick] (A) to (B);	
		\draw[thick] (B) to (C);	
  \draw[thick] (C) to (D);

\particles(0)(0);
\particles(6.9+4.9+1.6)(0);


\draw [->,line width=1pt]  (B2) to [bend right,in=135,out=45] (C);

   \draw [->,line width=1pt] (B2) to [bend right,in=-135,out=-45] (A);
    \node (text1) at (3.5,1){$1$} ;
	\node (text2) at (1.1,1){$q$} ;
	\node (text3) at (-2.3-1,1){$\alpha$};
	\node (text4) at (6.9+1,1){$\beta$};

	\node (text3) at (-2.3-1,-1){$\gamma$};
	\node (text4) at (6.9+1,-1){$\delta$};

    \node[scale=0.9] (text1) at (-2.2,-0.7){$1$} ;
    \node[scale=0.9] (text1) at (6.8,-0.7){$N$} ;
  	
  \draw [->,line width=1pt] (-3.9,0.475) to [bend right,in=135,out=45] (Z);
   \draw [->,line width=1pt] (Z) to [bend right,in=135,out=45] (-3.9,-0.475);
   \draw [->,line width=1pt] (6.9+1.6,-0.475) to [bend right,in=135,out=45] (D);	
   \draw [->,line width=1pt] (D) to [bend right,in=135,out=45] (6.9+1.6,0.475);

	\end{tikzpicture}	
	\caption{\label{fig:ASEP} Visualization of the open ASEP with respect to parameters $q,\alpha,\beta,\gamma,\delta$.}
	\end{figure}
		In order to study mixing times of the open ASEP, it will be convenient to consider specific functions $A,C$ of the boundary rates $\alpha,\beta,\gamma,\delta$ and the bias parameter $q$. Here and in the following, we usually drop the dependence on $N$ for ease of notation. When $\beta>0$, set
\begin{equation}\label{def:a}
A=A(\beta,\delta,q) := \frac{1}{2 \beta}\left( 1-q -\beta+ \delta + \sqrt{ (1-q -\beta+ \delta)^2 +4\beta\delta}\right)
\end{equation} and similarly, for $\alpha>0$,
\begin{equation}\label{def:c}
C=C(\alpha,\gamma,q) := \frac{1}{2 \alpha}\left( 1-q -\alpha+ \gamma + \sqrt{ (1-q -\alpha+ \gamma)^2 +4\alpha\gamma} \right) \ .
\end{equation}
We will assume that $\alpha,\beta>0$ in the following, hereby ensuring that the open ASEP has a unique stationary distribution $\mu_N$. We will see in Section \ref{sec:Prelim} that the parameters $A$ and $C$ play a crucial role as \textbf{effective reservoir densities}
\begin{equation}\label{def:EffectiveDensities}
\rho_{\Lup} := \frac{1}{1+C} \quad \text{and} \quad \rho_{\Rup} := \frac{A}{1+A} ,
\end{equation} which are closely related to basic properties of the stationary distribution $\mu_N$ of the open ASEP. {An important special case for the above parameters is \textbf{Liggett's condition} where the parameters have the simple form
\begin{equation}
    \gamma = (1-\alpha)q  , \quad \delta = (1-\beta)q , \quad  \rho_{\Lup}=\alpha , \quad  \rho_{\Rup}=1-\beta 
\end{equation}
}
Throughout this article, we will impose the assumption that
\begin{align}\label{def:TriplePointScaling}
q= \exp( - \psi N^{-\kappa} ) = 1 - \frac{\psi}{N^{\kappa}} + \mathcal{O}\left( N^{-2\kappa}\right)
\end{align}
for constants $\kappa \in [0,\frac{1}{2}]$ and $\psi>0$ -- see Section \ref{sec:Notation} for the respective asymptotic notation -- and
\begin{equation}\label{def:LiggettTypeCondition}
{\alpha + q^{-1} \gamma = 1 + \mathcal{O}(N^{-1/2}) \quad \text{and}   \quad  \beta + q^{-1} \delta = 1 + \mathcal{O}(N^{-1/2}) }.
\end{equation}
{Note that the special case when the error term $\mathcal{O}(N^{-1/2})$ vanishes is exactly Liggett's condition.}
We investigate the speed of convergence to the stationary distribution under the  \textbf{total variation distance}, i.e., for two probability measure $\nu,\nu^{\prime}$ on $\Omega_{N}$, we set
\begin{equation}\label{def:TVDistance}
\TV{ \nu - \nu^{\prime} } := \frac{1}{2}\sum_{x \in \Omega_{N}} |\nu(x)-\nu^{\prime}(x)| = \max_{A \subseteq \Omega_{N}} \left\{\nu(A)-\nu^{\prime}(A)\right\}.
\end{equation} We study the $\boldsymbol\varepsilon$\textbf{-mixing time} of $(\eta_t)_{t \geq 0}$, which is defined as
\begin{equation}\label{def:MixingTime}
t^{N}_{\text{\normalfont mix}}(\varepsilon) := \inf\Big\lbrace t\geq 0 \ \colon \max_{\eta^{\prime} \in \Omega_{N}} \TV{\P\left( \eta_t \in \cdot \ \right | \eta_0 = \eta^{\prime}) - \mu_{N}} < \varepsilon \Big\rbrace
\end{equation} for all $\varepsilon \in (0,1)$, and have the following result on the mixing time of the open ASEP when the effective densities are within order $N^{-1/2}$ of the \textbf{triple point}  $A=C=1$. {In the following, for functions $f,g \colon \N \rightarrow \R$, we write $f \lesssim g$ if $f(N)g(N)^{-1} \leq C_0$ as well as
$f \asymp g$ if $c'_0 \leq f(N)g(N)^{-1} \leq C'_0$ for constants $c_0,C_0',C_0>0$, and all $N$ large enough.}
\begin{theorem}\label{thm:MixingTimesTriple} Let $q$ satisfy \eqref{def:TriplePointScaling} for some $\kappa \in [0,\frac{1}{2})$ and $\psi>0$. Let $\alpha,\beta,\gamma,\delta$ satisfy condition  \eqref{def:LiggettTypeCondition}. Assume that there exist some constants $\tilde{A},\tilde{C} \in \R$ such that $A,C$ satisfy
\begin{equation}\label{eq:ScalingCondition}
A = \exp(-\tilde{A}N^{-1/2}) \quad \text{and} \quad C = \exp(-\tilde{C}N^{-1/2})
\end{equation}  for all $N$. Then for all $\varepsilon\in (0,1)$ fixed, we get that
\begin{equation}
t^{N}_{\text{\normalfont mix}}(\varepsilon) \asymp  N^{3/2+\kappa} .
\end{equation}
\end{theorem}

This confirms Conjecture 1.9 by Gantert et.\ al in \cite{GNS:MixingOpen} for the open ASEP at the triple point  {under Liggett's condition, i.e., when $\tilde{A}=\tilde{C}=0$ and $\kappa=0$, and where we have the simple parameter expression
\begin{equation}
    \alpha = \beta= \frac{1}{2}, \quad \gamma=\delta=\frac{q}{2} . 
\end{equation} }
To our best knowledge, this is the first time a sub-diffusive bound on the mixing time for a partially asymmetric system is established; see \cite{S:MixingTASEP} for the statement of Theorem \ref{thm:MixingTimesTriple} when $q=0$. We conjecture that Theorem \ref{thm:MixingTimesTriple} remains valid in the \textbf{maximal current phase} when we allow for general $A,C \leq 1$ {; see Remark~\ref{rem:MaxCurrent} on the possible extension of our arguments to the maximal current phase}.
When $\kappa=\frac{1}{2}$, we obtain the following bound on the mixing time.
\begin{theorem}\label{thm:MixingTimesKPZ} Let $q$ satisfy \eqref{def:TriplePointScaling} for $\kappa=\frac{1}{2}$ and some $\psi>0$. Let $\alpha,\beta,\gamma,\delta$ satisfy Liggett's condition. Assume that there exist some constants $\tilde{A},\tilde{C} \in \R$ such that \eqref{eq:ScalingCondition} holds for all $N\in \N$.
Then for all $\varepsilon\in (0,1)$ fixed, we get that
\begin{equation}
N^{2}\log^{-1}(N) \lesssim t^{N}_{\text{\normalfont mix}}(\varepsilon)  \lesssim  N^2\log^3(N)  .
\end{equation}
\end{theorem}

We conjecture that the correct order of the mixing time in Theorem \ref{thm:MixingTimesKPZ} is $N^2$. This is supported by results of Corwin and Shen, as well as Parekh, who establish in \cite{CS:OpenASEPWeakly} and \cite{P:KPZlimit} convergence under the scaling of Theorem \ref{thm:MixingTimesKPZ} to the open KPZ equation, and by Knizel and Matetski as well as Parekh, who establish a one force one solution principle for the open KPZ equation \cite{KM:StrongFeller,P:KPZergodic}.
On the way of establishing the main theorems, we will also prove  bounds on moderate deviations and the variance of the current of the open ASEP in Proposition~\ref{pro:CurrentNtoO} and  Corollary~\ref{cor:ModDevitationCurrentO}, as well as upper bounds on the mixing time in the weakly high density phase in Propositions~\ref{pro:MixingTimesWeaklyHighLow} and \ref{pro:MixingTimesWeaklyHighLowCritical}, which are of independent interest.

\subsection{Related work} \label{sec:RelatedWork}

Exclusion processes can be studied from a variety of different perspectives. Classical studies on exclusion processes focus on characterizing the invariant measure; see \cite{L:interacting-particle} for an introduction. While the set of invariant measures for ASEP on $\mathbb{Z}$ has a simple explicit description -- see Section 3 in Part VIII of \cite{L:interacting-particle} -- the stationary distribution of the open ASEP is still not fully understood.   In a classical result, Liggett establishes in \cite{L:ErgodicI} the phase diagram from Figure \ref{fig:phase}, showing that the overall density in the open ASEP depends on the effective densities at the boundaries. Since the 90s, an important tool to study the invariant measures of the open ASEP is the matrix product ansatz, introduced by Derrida, Evans, Hakim and Pasquier in \cite{DEHP:ASEPCombinatorics} when particles can move only in one direction; and which was later extended to the full set of parameters $\alpha,\beta,\gamma,\delta,q$; see \cite{BECE:ExactSolutionsPASEP,USW:PASEPcurrent} as well as \cite{BE:Nonequilibrium} for an introductory survey. Formally, the matrix product ansatz assigns to every configuration a weight, which can be represented as a product
of matrices and vectors. The matrices and vectors must satisfy certain relations, usually
called the DEHP algebra named after \cite{DEHP:ASEPCombinatorics}. Finding suitable weights in the matrix product ansatz is a question that gained lots of recent
attention. It also led to numerous combinatorial descriptions such as weighted Catalan paths and staircase tableaux; see for example \cite{BCEPR:CombinatoricsPASEP,CW:TableauxCombinatorics}.

A powerful matrix product ansatz representation was established by Uchiyama, Sasamoto and Wadati~in \cite{USW:PASEPcurrent} using Askey--Wilson polynomials. This led to a series of works by Bryc, Kuznetsov, Wang, Wesołowski, and many others, where the stationary distribution of the open ASEP in the fan region $AC \leq 1$ can be expressed using Askey–Wilson processes. As an application, this representation allows to study for example the limiting density and large deviations for the particle density of the open ASEP, as well as the convergence to a solution to the open KPZ equation, among various other properties \cite{BKWW:MarkovKPZ,BW:AskeyWilsonProcess,BW:QuadraticHarnesses,BW:Density,H:WASEP}. Very recently,  Wang, Wesołowski and Yang introduced signed Askey--Wilson measures in \cite{WWY:ASEPshock} to study properties of the open ASEP in the shock region $AC>1$. Another very recent line of research concerns the two--layer Gibbs representation  of the stationary distribution of the open ASEP established in   \cite{BCY:StationaryStrip,B:TwoLayerASEP,B:TASEPTwoLAyer}. Intuitively, the stationary distribution can be written as the top curve a coupled pair of suitably reweighted simple random walk trajectories.
Moreover, the approximation of the stationary distribution in total-variation distance by product measures was studied in \cite{NS:Stationary,Y:Approximation}. More generally, similar approaches involving different kinds of orthogonal polynomials or queuing interpretations of the stationary measures also allow to study the open ASEP with second class particles or to multi-species exclusion processes; see for example \cite{ALS:TwoSpeciesSemiPermeable, ALS:ClassesTwoSpecies,C:Koornwinder, FM:TASEPmulti,M:StationaryASEP,SY:OpenLight}. Finally, let us mention that for special choices of the boundary parameters in the shock {region} of the open ASEP, Schütz describes the stationary measure and the evolution of the open ASEP by a reversible exclusion process with finitely many particles, introducing the concept of reverse duality~\cite{S:ASEPReverse}.\\

Outside of the stationary distribution, the behavior of the open ASEP  is closely linked to properties of the current and second class particles for the asymmetric simple exclusion process on the integer lattice. In celebrated work \cite{BS:OrderCurrent}, Balázs and Seppäläinen establish that under a Bernoulli-$\frac{1}{2}$-product measure, the fluctuations of the current of the ASEP on $\Z$ at time $t$ are of order $t^{1/3}$. This was refined at the level of limiting distribution by Aggarwal in \cite{A:CurrentFluctuations}, the first proof in models without determinantal structures of the Baik-Rains limiting distribution~\cite{BR00}; see also \cite{F:ShockFluctuations,FF:CurrentFluctuations,FF:ShockFluctuations,PS:CurrentFluctuations} for a selection of classical results on the fluctuations of currents and second class particles. Let us mention that the results on the current fluctuations in \cite{BS:OrderCurrent} are established by proving order $t^{2/3}$ fluctuations for a second class particle in a Bernoulli-$\frac{1}{2}$-product measure, and using a scaling relation from~\cite{BS:FluctuationBounds}; see identity (2.1) therein. This  generalizes results from~\cite{PS01}, relating the two-point function to the probability density of second class particle, which in turn can be written in terms of height functions; see Proposition 4.1 of~\cite{PS01}. This was used to prove the asymptotic of the law of second class particle for TASEP, see~\cite{BFP12,BCS:CubeRoot,FS05a} for more details between current fluctuations and second class particles for a more general class of growth models.

In order to establish the fluctuations for second class particles,  \cite{BS:OrderCurrent} crucially relies on a multi-species asymmetric simple exclusion processes on the integers, using ideas from~\cite{FKS:MicroscopicStructure}. Similar instances for the use of multi-species ASEPs include the ASEP speed process relying on Rezakhanlou's coupling \cite{R:Coupling}; see also \cite{DH:SixSpeed} for recent extension to the six vertex model. We will in this paper require similar ideas for the study of multi-species ASEPs, relying instead of Rezakhanlou's coupling on the censoring inequality by Peres and Winkler from \cite{PW:Censoring}, which was used in the context of exclusion process for example in \cite{GNS:MixingOpen,L:CutoffSEP,S:MixingBallistic}.

Recently, Landon and Sosoe established in  \cite{LS:Tails} moderate deviations for the current and second class particles for the ASEP on $\Z$.
Their arguments rely on interpreting the ASEP as a limit of the stochastic six vertex model, as well as the so-called Rains–EJS identity, an exact formula developed by Emrah, Janjigian, and Seppäläinen \cite{EJS:Identity} for the moment generation function of  stationary exponential last passage percolation, recovering  unpublished work by Rains~\cite{R:Identity}. In a broader context, these scaling relations for the current and second class particle fluctuations  manifest the role of the ASEP as a central model in the KPZ universality class; see also \cite{ACG:ASEPspeed,ACH:ScalingASEP,QS:KPZfix} for recent results on the existence of the ASEP speed process and convergence to the directed landscape, as well as \cite{Y:kpzASEP} for KPZ scaling of a generalization of the ASEP with long range interactions. \\

For exclusion processes on a finite state space, mixing time are a key tool to quantify the speed of convergence to the stationary distribution; see \cite{LPW:markov-mixing} for a general introduction to mixing times.
Over the last decade, significant progress was achieved in the study of mixing times of the asymmetric exclusion processes on a closed segment. In \cite{LL:CutoffASEP}, Labbé and Lacoin verify that the mixing times is linear in the size of the segment and  the occurrence of the cutoff phenomenon, a sharp transition in the convergence to equilibrium, where the system goes from being far from equilibrium to being close to equilibrium within a small time window. These results were generalized in \cite{LL:CutoffWeakly} to the weakly asymmetric simple exclusion process where $q=q(N)\rightarrow 1$ as the system size growths. Recently, the results were further sharpened. Bufetov and Nejjar in \cite{BN:CutoffASEP} establish a Tracy-Widom limit profile using a delicate color position symmetry for the ASEP on $\Z$ from \cite{BB:ColorPosition}, where they interpret the ASEP as a random walk on Hecke algebra; see also \cite{Z:Metropolis} for the limit profile for a multi-species ASEP by Zhang. A similar result was established in \cite{HS:Limits} for the ASEP with one open boundary, using current fluctuations for the ASEP on $\N$ by He \cite{H:Boundary}.

All these sharp  results have in common that the respective exclusion processes are reversible, and that the stationary distribution is given by a variant or projection of Mallows measure. This is in contrast to the ASEP with two open boundaries, which is in general not reversible, and where the available results are much more sparse. For the high and low density phase, i.e., when $A>\max(1,C)$ or $C>\max(1,A)$ holds, Gantert et al.\ establish in \cite{GNS:MixingOpen} a mixing time of order $N$ for the open ASEP on a segment of length $N$. We will elaborate on their approach in more detail in Section \ref{sec:Overview} when we discuss our strategy in order to study mixing times of the open ASEP at the triple point. In the special case of the open TASEP, i.e., where $\gamma=\delta=q=0$, mixing times are much better understood. For the maximal current phase $AC \leq 1$, a mixing time of order $N^{3/2}$ and the absence of cutoff are shown in \cite{S:MixingTASEP,SS:TASEPcircle}. Similar results are also available for the TASEP on the circle where the mixing time is shown in \cite{SS:TASEPcircle} to be of order $N^{2}\min(k,N-k)^{-1/2}$ for a system with $k$ particles. In the high and low density phase, \cite{ES:HighLow} verifies that the cutoff phenomenon occurs while in the co-existence line $A=C>1$, the mixing time is of order $N^2$ and we see no cutoff. The reason for this drastically increasing precision in the results is an alternative representation of the open TASEP as a last passage percolation model. In~\cite{SS:TASEPcircle}, it is shown that the mixing time can be expressed by coalescence times of geodesics in exponential last passage percolation. In return, precise bounds on the coalescence of geodesics are due to
exact formulas for various quantities in last passage percolation,  which are themselves achieved by a connection via RSK correspondence and (Pfaffian) Schur processes to random matrix theory  \cite{BBCS:Halfspace,BGHK:BetaEnsembles,BSV:SlowBond,LR:BetaEnsembles,SI03}. Let us stress that a similarly powerful set of tools is currently not available for the asymmetric simple exclusion process with $q>0$,  thus requiring different ideas to establish mixing times.

\subsection{Overview of the proof}\label{sec:Overview}

As a reader’s guide, let us in the following give an outline of the
main ideas and the strategy for the proof of mixing times  at the triple point. \\

To this end, we recall the overall strategy from Gantert et al.~in \cite{GNS:MixingOpen} for mixing times for the open ASEP in the high density phase for constant boundary and bias parameters. Consider two open ASEPs started from the extremal configurations, where all sites are either fully occupied by particles or are left empty, respectively. Then under the basic coupling, the two open ASEPs can be interpreted as one open ASEP $(\xi_t)_{t \geq 0}$ with second class particles. We obtain an upper bound on the mixing by providing estimates on the time it takes for all initially $N$ many second class particles to exit the segment. In order to study this exit time, Gantert et al.\ couple the open ASEP with second class particles to a system of two stationary open ASEPs with different boundary parameters -- one of them with the same boundary parameters as $(\xi_t)_{t \geq 0}$ and one where $\alpha$ and $\beta$ are decreased. They show that if the currents of the two stationary systems until some time $T$ differ by at least order $N$, then this implies an upper bound on the mixing time of order $T$. \\

While this approach is sufficient for effective bounds on the mixing time for the high density phase and constant boundary parameters,  it indicates three main challenges for the triple point. First, we need effective bounds on the expected current of the open ASEP in the maximal current phase, and when the parameters $\alpha,\beta,\gamma,\delta,q$ are allowed to depend on the system size $N$. Second, as the most delicate part of the argument, we require bounds on the fluctuations of the current of the open ASEP in order to guarantee that the currents differ at the order of their expected differences with positive probability. Lastly, even with an optimal bound on the {expected} difference and fluctuations of the current, the arguments in \cite{GNS:MixingOpen} a priori only yield an upper bound of order $N^{2+\kappa}$ {(instead of $N^{3/2+\kappa}$)} on the mixing time of the open ASEP at  the triple point. Let us now describe how we address all three challenges.

In order to obtain effective bounds on the expected current of the open ASEP in the maximal current phase when the parameters $\alpha,\beta,\gamma,\delta,q$ are allowed to depend on the system size $N$, we rely on an exact representation of the stationary current by Uchiyama, Sasamoto,  Wadati from \cite{USW:PASEPcurrent} as a certain contour integral. This formula uses that the invariant measure of the open ASEP can be studied via the matrix product ansatz, and a particular solution related to Askey--Wilson polynomials.
We provide a sharp asymptotic analysis of this exact expression for the expected current using various identities of Gamma functions and $q$-calculus. In particular, we utilize recent fine asymptotic results on \mbox{$q$-Pochhammer} symbols by Corwin and Knizel from  \cite{CK:StationaryKPZ}.

In order to study the fluctuations of the current of the open ASEP, we introduce a framework to compare moderate deviations for the current of the open ASEP to moderate deviations of the current of the ASEP on $\Z$. Precise moderate deviations for the ASEP on $\Z$ were recently established by Landon and Sosoe \cite{LS:Tails}. In order to compare the currents, we use the basic coupling between the open ASEP and the ASEP on $\Z$. We establish moderate deviations for the speed of second class particles, which are inserted  at sites $1$ and $N$, using different hierarchies of particles and the censoring inequality to study the motion among them.
As the arguments on moderate deviations take most of the body of this article, and are of independent interest, we will provide a more detailed outline of the strategy of the proof in Section~\ref{sec:Strategy}.

In order to control the exit time of second class particles, we provide an iterative scheme, using a generalization of the partially ordered  multi-species exclusion process introduced by Gantert et al.\ in \cite{GNS:MixingOpen}. More precisely, by iterating over $n$ with  $2^{n} \leq \sqrt{N}$, after a time of order $N^{1+\kappa}2^{n}$, the law of the open ASEP is with probability at least \mbox{$1-\exp(\min(-2^{-n}\sqrt{N},N^{\expont})^{1/2})$} for some constant $\expont>0$ stochastically dominated from above by a Bernoulli-$(\frac{1}{2}+\frac{1}{2^n})$ product measure, and from below by a Bernoulli-$(\frac{1}{2}-\frac{1}{2^n})$ product measure.
This ensures that after time of order $N^{3/2+\kappa}$ there are, with positive probability, at most order $\sqrt{N}$ many second class particles in the open ASEP. This allows us to then apply the arguments of Gantert et al.\ from \cite{GNS:MixingOpen} to obtain an upper bound of order $N^{3/2+\kappa}$ on the mixing time of the open ASEP at the triple point.

\subsection{Outline of the paper} \label{sec:OutlinePaper}

This paper is structured as follows. In Section \ref{sec:Prelim}, we state preliminary results on the open ASEP and the asymmetric simple exclusion process on the integers and the half-line. This includes the basic coupling for multi-species exclusion processes, a  characterization of invariant measures for one-dimensional exclusion processes, the phase diagram for the open ASEP, the censoring inequality, and recent results on moderate deviations for the current of the ASEP on $\Z$. In Section \ref{sec:CurrentEstimates}, we study the stationary current of the open ASEP and provide sharp estimates in the maximal current phase. In Section~\ref{sec:ModDevCurrent}, we introduce an extended disagreement process, which allows us to compare the motion of second class particles in the ASEP on $\Z$, the ASEP on $\N$ and the open ASEP. We then use this process to convert moderate deviations  for the current of the ASEP on $\Z$ to moderate deviations for the current of the open ASEP. Section~\ref{sec:MixingHighLow} uses the previous results to achieve an upper bound on the mixing time of the open ASEP in the weakly high and weakly low density phase via an iterative scheme. The respective bounds at the triple point are established in Section~\ref{sec:MixingTriple}.
We conclude by providing the corresponding lower bounds on the mixing times in Section~\ref{sec:MixingTimesLowerBounds}.

\subsection{Notation}\label{sec:Notation}
We use standard asymptotic notation throughout this paper. For functions $f,g \colon \N \rightarrow \infty$, we will write $f=o(g)$ if $f(N)g(N)^{-1}\rightarrow 0$ and  $f \sim g$ if $f(N)g(N)^{-1}\rightarrow 1$ for $N \rightarrow \infty$. Similarly, we write  $f=\mathcal{O}(g)$ if $f(N)g(N)^{-1} \leq C_0$ as well as
$f=\Theta(g)$ if $c_0 \leq f(N)g(N)^{-1} \leq C_0$ for constants $c_0,C_0>0$ and all $N$ large enough. We will sometimes write $\asymp$ instead of $\Theta$, and  $\lesssim$ instead of $\mathcal{O}$, as well as $\ll$ instead of $o$. We allow the parameters $\alpha,\beta,\gamma,\delta,q$, and hence the corresponding functions $A,B,C,D$, to depend on the system size $N$, while the constants $c_0,C_0,c_1,C_1,\dots$ do not depend on $N$ and may change from line to line. A special case are the constants $\kappa,\expont,\psi$ as well as $\tilde{A},\tilde{B},\tilde{C},\tilde{D}$, which capture the scaling relations for the open ASEP, and hence do not depend on $N$.

\section{Preliminaries}\label{sec:Prelim}

In the following, we collect preliminary results on the asymmetric simple exclusion process, which we will use throughout this paper.

\subsection{The asymmetric simple exclusion process}\label{sec:ASEPs}

Recall the asymmetric simple exclusion process with open boundaries (open ASEP), which we defined in \eqref{def:openASEP}. We will now define two related exclusion processes, which can both be interpreted as a limit as $N \rightarrow \infty$ for the open ASEP; the asymmetric simple exclusion process on the half-line $\N$ and on the integers $\Z$. For parameters $\alpha,\gamma \geq 0$ and $q\in [0,1)$, we define the \textbf{ASEP on $\N$} as the Markov process  $(\eta^{\N}_t)_{t\geq 0}$ with state space $\{0,1\}^{\N}$ and generator
\begin{equation} \label{def:halflineASEP}
\begin{split}
\mathcal{L}_{\N}f(\eta) &= \sum_{x =1}^{\infty} \big(\eta(x)(1-\eta(x+1)) + q \eta(x+1)(1-\eta(x)) \big) \left[ f(\eta^{x,x+1})-f(\eta) \right]  \\
 &+ \alpha (1-\eta(1)) \left[ f(\eta^{1})-f(\eta) \right] \hspace{2pt} +  \gamma \eta(1) \left[ f(\eta^{1})-f(\eta) \right]
  \end{split}
\end{equation}
 with respect to all cylinder functions $f\colon \{0,1\}^{\N} \rightarrow \R$. {Note that as in \eqref{def:EffectiveDensities}, we can assign an effective density $\rho_{\mathsf{L}}$ to the ASEP on $\N$.} Similar, we let $(\eta^{\Z}_t)_{t\geq 0}$ denote the \textbf{ASEP on} $\boldsymbol{\Z}$ with state space $\{0,1\}^{\Z}$ and generator
\begin{equation} \label{def:fullASEP}
\mathcal{L}_{\Z}f(\eta) = \sum_{x \in \Z} \big(\eta(x)(1-\eta(x+1)) + q \eta(x+1)(1-\eta(x)) \big) \left[ f(\eta^{x,x+1})-f(\eta) \right] 
\end{equation}
for all cylinder functions $f\colon \{0,1\}^{\Z} \rightarrow \R$. We refer to Liggett \cite{L:interacting-particle} for an introduction to this model.
Moreover, on the integer lattice, we define the \textbf{multi-species ASEP} $(\zeta_t)_{t \geq 0}$. This is a Markov process taking values in $(\N \cup \{\infty\})^{\Z}$, i.e., we assign types in $\N \cup \{\infty\}$ to all vertices in $\Z$. Along each edge $\{x,x+1\}$ with $x\in \Z$, we place two independent rate $1$ and $q$ Poisson clocks. Whenever the rate $1$ clock rings, we sort the endpoints in decreasing order, whenever the rate $q$ clock rings, we sort the endpoints in increasing order.
Note that for any $k\in \N$, we obtain from $(\zeta_t)_{t \geq 0}$ an ASEP on $\Z$ by identifying types $1,2,\dots,k$ with particles and types $k+1,k+2,\dots$ (including $k=\infty$) with empty sites.

\subsection{The basic coupling}\label{sec:CanonicalCoupling}

Next, we introduce the basic coupling for simple exclusion processes. In contrast to \cite{BBHM:MixingBias} or \cite{GNS:MixingOpen}, we do neither assume that the underlying configurations are componentwise ordered, nor that the exclusion processes are defined with respect to the same underlying graphs. Let $G=(V,E)$ and $G^{\prime}=(V^{\prime},E^{\prime})$ be either the interval $ \lbr N \rbr$, the half-line $\N$, or the integer lattice $\Z$. Without loss of generality assume that $E \subseteq E^{\prime}$.

\begin{definition}[Basic coupling]\label{def:BasicCoupling}
Let $q \in [0,1)$. We define the basic coupling $\mathbf{P}$ between the asymmetric simple exclusion processes on $G$ and $G^{\prime}$ as follows.
We place independent rate $1+q$ Poisson clocks on all edges $e\in E$ and use the same clocks in both processes. Whenever the clock of an edge $e=\{ x,x+1\}$ rings, and the respective exclusion processes are in states $\eta$ and $\eta^{\prime}$, respectively, we sample an independent Uniform-$[0,1]$-distributed random variable $U$ and distinguish two cases:
\begin{itemize}
\item If $U \leq (1+q)^{-1}$, and $\eta(x)=1-\eta(x+1)=1$ holds, we move the particle at site $x$ to site $x+1$ in configuration $\eta$.
\item If $U > (1+q)^{-1}$, and $\eta(x)=1-\eta(x+1)=0$ holds, we move the particle at site $x+1$ to site $x$ in configuration $\eta$.
\end{itemize}
The configuration $\eta^{\prime}$ is updated in the same way, using the same random variable $U$. {For all edges in $E' \setminus E$, we use independent rate $1+q$ Poisson clocks, and only update the asymmetric simple exclusion process on $G'$ according to the above rules.}
In addition, when $G$ {or $G'$} is either the interval or the half-line, we place a rate $\alpha$ Poisson clock (respectively a rate $\gamma$ Poisson clock) on site $1$. Whenever the clock rings, we place a particle (respectively an empty site) at site $1$, irrespective of the current value of $\eta(1)$. Similarly, when $G$ {or $G'$} is the interval, we place a rate $\beta$ Poisson clock (respectively a rate $\delta$ Poisson clock) on site $N$. Whenever the clock rings, we place an empty site (respectively a particle) at site $N$, irrespective of the current value of $\eta(N)$. As in the cases before, we use the same Poisson clocks to update $\eta^{\prime}$.
\end{definition}
The basic coupling allows us to naturally define a process $(\xi_{t})_{t \geq 0}$ taking values in $(\{0,1\}\times \{0,1\})^{\Z}$, which has the laws of $(\eta_t)_{t \geq 0}$  on $V$ and $(\eta^{\prime}_t)_{t \geq 0}$ on $V^{\prime}$ as marginals. We say that $(\xi_{t})_{t \geq 0}$ is occupied by a \textbf{first class particle} at site $x$ if $\xi_t(x)=(1,1)$. Similar, we say that $x$ is occupied by a \textbf{second class particle of type \Aup}  if $\xi_t(x)=(1,0)$, by a \textbf{second class particle of type \Bup} if $\xi_t(x)=(0,1)$, and by an \textbf{empty site} if $\xi_t(x)=(0,0)$. We refer to $(\xi_{t})_{t \geq 0}$ as the \textbf{disagreement process} between $(\eta_t)_{t \geq 0}$ and $(\eta^{\prime}_t)_{t \geq 0}$. \\

Note that the basic coupling respects the partial order, where first class particles have a higher priority than second class particles, and second class particles have a higher priority than empty sites. On the other hand, whenever two adjacent second class particles of different types are updated, they are replaced by a pair of first class particles and empty sites. Hence, $(\xi_t)_{t \geq 0}$ can also be interpreted as a continuous-time Markov chain on the state space $$\Omega_N^{2} := \{ (0,0),(1,0),(0,1),(1,1)\}^{N} . $$

\begin{remark}\label{rem:Disagreement}
Note that whenever the disagreement process contains only second class particles of one type (either type {\Aup} or type {\Bup}), we can identify the disagreement process as a multi-species exclusion process of three types $1,2,\infty$.
\end{remark}

\subsection{The stationary distribution of asymmetric simple exclusion processes}

In the following, we collect results on the stationary distribution of the open ASEP as well as the ASEPs on $\N$ and $\Z$. Note that whenever $\alpha,\beta>0$, the open ASEP has a unique stationary distribution $\mu_N$, which was intensively studied over the past decades. The next  result is due to Liggett \cite{L:ErgodicI} when
\begin{equation}\label{def:LiggettsCondition}
\gamma= q(1-\alpha)\quad \text{ and }   \quad  \delta = q(1-\beta) ;
\end{equation}
 see also Proposition A.2 and Remark A.3 in \cite{NS:Stationary} for general parameters $\alpha,\beta,\gamma,\delta,q$. We write $\textup{Ber}_n(\rho)$ for the Bernoulli-$\rho$-product measure on $\{0,1\}^{n}$ for some $n\in \N$ and $\rho \in [0,1]$.

\begin{theorem}\label{pro:Finite} Let $q\in [0,1)$. Assume that the parameters $\alpha,\beta,\gamma,\delta \geq 0$ and $q$ do not depend on $N$. Let $M>0$ be finite and consider intervals $I=I_N=\lbr a_N, a_N+M-1\rbr$ with
\begin{equation}
\lim_{N \rightarrow \infty} \min(a_N,N-a_N) = \infty .
\end{equation}
Let $\mu_N^{I}$ denote the projection of $\mu_N$ onto the sites $I$. Then
\begin{equation}
\mu_{N}^{I} \rightarrow \begin{cases}  \textup{Ber}_M(\rho_{\Lup} ) & \text{ if } C > \max(A,1) \\
\textup{Ber}_M(\rho_{\Rup}) &  \text{ if } A > \max (C,1) \\
\textup{Ber}_M\big(\frac{1}{2}\big) & \text{ if } \max(A,C) \leq 1
\end{cases}
\end{equation} in the sense of weak convergence, where we recall $\rho_{\Lup}$ and $\rho_{\Rup}$ from \eqref{def:EffectiveDensities}.
\end{theorem}

Theorem \ref{pro:Finite} motivates the following phase diagram for the open ASEP.
We say that the open ASEP is in the \textbf{maximal current phase} when $AC \leq 1$, it is in the \textbf{high density phase} when $A>\max(1,C)$, and in the \textbf{low density phase} when  $C>\max(1,A)$. We refer to the remaining case $A=C>1$ as the \textbf{coexistence line}. A visualization of the different phases is given in Figure \ref{fig:phase}.

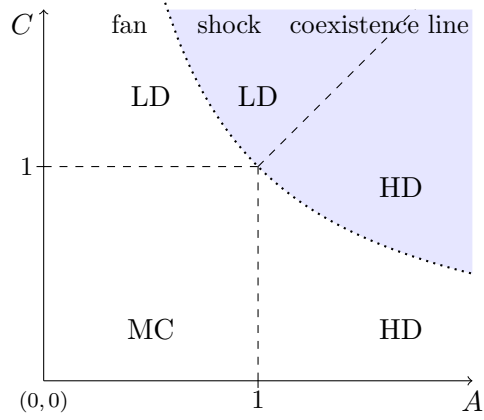
\begin{figure}
    \centering
    \begin{tikzpicture}[scale=0.95]

 \begin{scope}
   \clip (6.7,5) rectangle (11,10.2); 
   \fill[blue, opacity=0.1] plot[domain=6.7:11, smooth, variable=\x] ({\x},{1/((\x-7)*1/3+2/3)*3+5}) -- (11,10.2) -- (6.7,10.2) -- cycle;
 \end{scope}

 \draw[scale = 1,domain=6.7:11,smooth,variable=\x,dotted,thick] plot ({\x},{1/((\x-7)*1/3+2/3)*3+5});
 \draw[->] (5,5) to (5,10.2);
 \draw[->] (5.,5) to (11,5);
   \draw[dashed] (5,8) to (8,8);
  \draw[dashed] (8,8) to (8,5);
   \draw[dashed] (8,8) to (10.2,10.2);
   \node [left] at (5,8) {\small $1$};
   \node[below] at (8,5) {\small $1$};
     \node [below] at (11,5) {$A$};
   \node [left] at (5,10) {$C$};
 \draw[dashed] (8,4.9) to (8,5.1);
  \draw[dashed] (4.9,8) to (5.1,8);
 \node [below] at (5,5) {\scriptsize{$(0,0)$}};
    \node [above] at (6.5,8.7) {LD};\node [above] at (8,8.7) {LD};
    \node at (9.7,10) {\small coexistence line };
    \node [below] at (10,6) {HD};  \node [below] at (10,8) {HD};
 \node [below] at (6.5,6) {MC};
 \node at (6.2,10) {\small{fan}};
 \node at (7.6,10) {\small{shock}};

 \end{tikzpicture}
    \caption{Phase diagrams for the open ASEP stationary measures. The terms LD, HD, and MC, respectively, correspond to the low density, high density and maximal current phase. Moreover, we distinguish between the fan and the shock region for the open ASEP.}
    \label{fig:phase}
\end{figure}

Let us mention that the above results on the stationary distribution can be strengthened; see \cite{BW:Density,WWY:ASEPshock} for scaling limits of the particle density under $\mu_N$, and \cite{NS:Stationary,Y:Approximation} for an approximation of the stationary measure as $M=M(N) \rightarrow \infty$ with $N \rightarrow \infty$ by Bernoulli product measures.
Let us stress that we will in the following allow  $A=A_N$ and $C=C_N$ to depend on $N$. Whenever $A_N \rightarrow 1$ while $A_N > \max(C_N,1)$, we say that the open ASEP belongs to the \textbf{weakly high density phase}.
Similarly, when $C_N \rightarrow 1$ while $C_N > \max(A_N,1)$, we refer to the \textbf{weakly low  density phase}. The next lemma states that in the special case $A C =1$, the stationary measure of the open ASEP has a simple product form.

\begin{lemma}\label{lem:Bernoulli}
Recall $\rho_{\Lup},\rho_{\Rup}$ from \eqref{def:EffectiveDensities}. Whenever $AC=1$, we have that
\begin{equation}
\mu_N = \textup{Ber}_N(\rho_{\Lup})=\textup{Ber}_N(\rho_{\Rup}) .
\end{equation}  Moreover, for all $\alpha,\beta,\gamma,\delta>0$ and $q=q_N \in (0,1)$, we have that
\begin{equation}
 \textup{Ber}_N\big(\max\big(\rho_{\Lup},\rho_{\Rup}\big) \big) \succeq  \mu_N \succeq \textup{Ber}_N\big(\min\big(\rho_{\Lup},\rho_{\Rup}\big) \big) ,
\end{equation} where we denote by $\succeq$ stochastic domination between probability measures on $\{0,1\}^{N}$.
\end{lemma}
\begin{proof}
The first statement is a simple computation; see Proposition 2 in \cite{BCEPR:CombinatoricsPASEP}. The second statement can be found for example as Lemma 2.10 in \cite{GNS:MixingOpen}.
\end{proof}
The condition $AC=1$ provides another partition of the parameter space for the open ASEP. We refer to $AC < 1$ as the \textbf{fan region} and to $AC > 1$ as the \textbf{shock region} of the open ASEP.
We now investigate the stationary measure of the half-line ASEP as well as the ASEP on $\Z$. The next result is standard and follows from a straightforward computation; see Theorem 2.1 in Section VIII of \cite{L:interacting-particle}.

\begin{lemma}
For all $\rho \in [0,1]$ and all $q\in [0,1]$, the Bernoulli-$\rho$-product measures on $\{0,1\}^{\Z}$ are invariant measure for the ASEP on $\Z$. {Recall $C=C(\alpha,\gamma,q)$ from \eqref{def:c}. Then the Bernoulli-$\rho_{\Lup}$-product measure with $\rho_{\Lup}=(C+1)^{-1}$} is an invariant measure for the ASEP on $\N$ with respect to parameters $\alpha,\gamma>0$ and $q\in [0,1)$.
\end{lemma}

We refer to \cite{Y:Approximation} for a more detailed discussion by Yang on the existence and structure of further invariant measures of the half-line ASEP when $C \geq 1
$. For the ASEP on $\Z$, we consider the family of invariant measures called \textbf{blocking measures} $\nu^{(n)}$ on
\begin{equation}\label{def:BlockingStateSpace}
\mathfrak{A}_n := \Big\{ \eta \in \{ 0,1 \}^{\Z} \, \colon \, \sum_{i>n} (1-\eta(i)) =  \sum_{i \leq n} \eta(i) < \infty \Big\} .
\end{equation}
{Observe that the sets $\mathfrak{A}_n$ are countable and stable under the ASEP $(\eta^{\Z}_t)_{t \geq 0}$ on $\Z$, i.e., if $\eta^{\mathbb{Z}}_0 \in \mathfrak{A}_n$ for some $n\in \Z$, then almost surely $\eta^{\mathbb{Z}}_t \in \mathfrak{A}_n$ for all $t\geq 0$. Set
\begin{equation}
    \mathfrak{A} := \bigcup_{n\in \Z} \mathfrak{A}_n = \Big\{ \eta \in \{ 0,1 \}^{\Z} \, \colon \, \sum_{i>0} (1-\eta(i)) <\infty \textrm{ and }  \sum_{i \leq 0} \eta(i) < \infty \Big\} .
\end{equation}
}
For $q>0$, we denote by $\tilde{\nu}$ the Bernoulli product measure on $\{0,1\}^{\Z}$ with marginals
\begin{equation}
\tilde{\nu}(\eta(x)=1) = \frac{q^{-x}}{1+q^{-x}}
\end{equation} for all $x \in \Z$. {Observe that $\tilde{\nu}$  has its support on the $\mathfrak{A}$, and thus is an atomic measure. Therefore, $\tilde{\nu}(\mathfrak{A}_n)>0$ for all $n\in \N, q>0$, and we set 
$\nu^{(n)} := \tilde{\nu}(\, \cdot \, | \mathfrak{A}_n)$ for all $n\in \Z$.} We make the following observation about  blocking measures, which we will frequently use throughout this article.

\begin{lemma}\label{lem:BlockingMeasures}
Let $q\in (\varepsilon,1)$ for some $\varepsilon>0$. For all $n\in \Z$, the ASEP on $\Z$ restricted to $\mathfrak{A}_n$ is reversible with respect to $\nu^{(n)}$. Moreover, there exist some  constants $c_0,C_1,C_0>0$, depending only on $\varepsilon>0$, such that for all $x>C_1(1-q)^{-1}$,
\begin{equation}\label{eq:BlockingMax}
\begin{split}
\nu^{(n)}\Big( \exists y<n-\frac{x}{1-q} \colon \eta(y) = 1   \Big) &\leq C_0 \exp(-c_0 x) \\
\nu^{(n)}\Big( \exists y^{\prime} > n+\frac{x}{1-q} \colon  \eta(y^{\prime}) = 0  \Big) &\leq C_0 \exp(-c_0 x) .
\end{split}
\end{equation}
\end{lemma}
\begin{proof}
Reversibility of the measure can be checked directly by a simple computation, see Theorem 2.1 in Section VIII of \cite{L:interacting-particle}.
For the second statement, assume without loss of generality that $n=0$. Note that \eqref{eq:BlockingMax} holds under the measure $\tilde{\nu}$ as
\begin{equation*}
\begin{split}
\tilde{\nu}\left( \exists z \in \Big(-\infty,\frac{-y}{1-q}\Big] \, \colon  \eta(z)=1  \right) &\leq \frac{1}{1-q}\exp(-c_2y) \\
\tilde{\nu}\left( \exists z^{\prime} \in \Big[\frac{y}{1-q},\infty\Big) \, \colon \eta(z^{\prime})=0  \right) &\leq \frac{1}{1-q}\exp(-c_2y)
\end{split}
\end{equation*} for some constant $c_2>0$ and all $q\in (\varepsilon,1]$, and all $y\geq 1$. Since
\begin{equation*}
\tilde{\nu}(\mathfrak{A}_0) \geq c_3(1-q)
\end{equation*} for some constant $c_3>0$ and all $q\in (\varepsilon,1)$, this allows us to conclude.
\end{proof}

\subsection{The censoring inequality}

The censoring inequality by Peres and Winkler in \cite{PW:Censoring} intuitively states {that for monotone spin systems} leaving out transitions only increases the distance from equilibrium. This was applied by Lacoin in \cite{L:CutoffSEP} to the symmetric simple exclusion processes on a closed segment. We will in the following apply the censoring inequality with respect to the ASEP $(\eta^{\Z}_t)_{t \geq 0}$ on the integers, restricted to $\mathfrak{A}_n$ from \eqref{def:BlockingStateSpace} for some $n\in \N$.

\begin{definition}[Censoring]\label{def:Censoring}
Let $(\eta^{\Z}_t)_{t \geq 0}$ be an ASEP on $\mathfrak{A}_n$ for some $n\in \Z$ and let $ \mathcal{P}\left( E \right)$ denotes the power set of $E(\Z)$. We call a random càdlàg function
\begin{equation}
\mathcal{C} \colon \mathbb{R}_0^+ \rightarrow \mathcal{P}\left( E \right)
\end{equation} a \textbf{censoring scheme} for the exclusion process $(\eta^{\Z}_t)_{t \geq 0}$ when $\mathcal{C}$ is independent of the process $(\eta^{\Z}_t)_{t \geq 0}$. We define the process $(\eta^{\mathcal{C}}_t)_{t\geq 0}$ as the \textbf{censored exclusion process}, where an update along an edge $e$ at time $t$ is performed if and only if $e \notin \mathcal{C}(t)$.
\end{definition}
Define the partial order $\succeq_{\h}$ on $\mathfrak{A}$ by the relation
\begin{equation}\label{eq:PartialHeight}
\eta \succeq_{\h} \eta^{\prime} \quad \Leftrightarrow {\quad \sum_{y=-\infty}^{x} \eta(y) \geq  \sum_{y=-\infty}^{x} \eta^{\prime}(y) \,  \forall x \in \Z . }
\end{equation}  Similarly, we write $\nu \succeq_{\h} \nu^{\prime}$ for laws $\nu,\nu^{\prime}$ if there exists a coupling {$\mathbf{P}$ such that for $\eta \sim \nu$ and $\eta' \sim \nu'$, 
\begin{equation}
   \mathbf{P}( \eta \succeq_{\h} \eta' ) = 1 . 
\end{equation}
}
The following result is Proposition~2.12 and Remark~2.13 in \cite{GNS:MixingOpen}, {which can be seen as a special case of Theorem 1.1 in \cite{PW:Censoring}.}

\begin{lemma}\label{lem:CensoringASEP}  Let $\mathcal{C}$ be a censoring scheme. Let $P_{\eta}(\eta_t \in \cdot)$ and $P_{\eta}(\eta^{\mathcal{C}}_t \in \cdot)$ denote the law of $(\eta^{\Z}_t)_{t\geq 0}$ and its censored dynamics $(\eta^{\mathcal{C}}_t)_{t\geq 0}$ at time $t\geq 0$, respectively, starting from some $\eta \in \mathfrak{A}_n$. Then for all $n \in \Z$, we have that
\begin{equation}\label{def:GroundStates}
P_{\vartheta_n}\big(\eta^{\mathcal{C}}_t \in \cdot \big) \preceq_\h P_{\vartheta_n}(\eta_t \in \cdot) 
\end{equation} for all $t\geq 0$, 
where we define {$\vartheta_n \in \{0,1\}^{\Z}$} by $\vartheta_n(x):= \mathds{1}_{\{x > n \}}$ for all $x\in \Z$.
\end{lemma}

We record the following consequence of Lemma \ref{lem:BlockingMeasures} and Lemma \ref{lem:CensoringASEP}.

\begin{lemma}\label{lem:BlockingMeasureMaximum}
Let $q$ satisfy \eqref{def:TriplePointScaling} for some $\kappa \in [0,1]$ and $\psi>0$. Let $(\eta^{\mathcal{C}}_t)_{t \geq 0}$ denote a censored exclusion process on $\mathfrak{A}_n$ for some $n\in \Z$. Let $(L_t)_{t \geq 0}$ and $(R_t)_{t \geq 0}$ denote the position of the leftmost particle and rightmost empty site in $(\eta^{\mathcal{C}}_t)_{t \geq 0}$, respectively. Then there exist constants $c_0,C_0,C_1>0$ such that for any censoring scheme $\mathcal{C}$, and all  $x \geq C_1 \log(N)$,
\begin{equation}
 \P \left(  R_{t},L_{t} \in \Big[ n-xN^{\kappa}, n+ xN^{\kappa}\Big]  \, \forall 0 \leq t \leq N^3  \right) \geq 1- C_0\exp(-c_0 x) .
\end{equation}
\end{lemma}
\begin{proof}
Note that by Lemma \ref{lem:CensoringASEP} with $t\rightarrow \infty$, the bounds from Lemma \ref{lem:BlockingMeasures} continue to hold for the censored exclusion process $(\eta^{\mathcal{C}}_t)_{t \geq 0}$ on $\mathfrak{A}_n$ started from $\nu^{(n)}$. Thus, we get that for some constants $c_2,C_2,C_3>0$
\begin{equation}\label{eq:CensorBlock1}
 \P \left(  R_{m},L_{m} \in \Big[ n-x N^{\kappa}, n+ xN^{\kappa}\Big] \, \forall m \in \lbr N^3 \rbr \right) \geq 1- C_2\exp(-c_2 x)
\end{equation} for all $x \geq C_3 \log(N)$ and $N$ large enough. For all $m\in \N$ and all $\tilde{\eta} \in \mathfrak{A}_0$, we see that for some constant $c_3>0$ and all $x$ large enough
\begin{equation}\label{eq:CensorBlock2}
 \P \Big(  \sup_{t \in [m,m+1]}(| R_t - R_{m} |+ |L_{t}-L_m| ) \geq x \, \Big| \, \eta_m =\tilde{\eta} \Big)  \leq \exp(-c_3 x)
\end{equation} as particles can move at most at speed $2$. This
allows us to conclude by \eqref{eq:CensorBlock1} and a union bound on $m \in  \lbr N^{3} \rbr $ in \eqref{eq:CensorBlock2}.
\end{proof}

We will primarily apply Lemma \ref{lem:BlockingMeasureMaximum} with respect to censoring schemes, which result from a projection of a multi-species exclusion process, e.g., where we can interpret the dynamics of second class particles with respect to empty sites as an  exclusion process with censoring by erasing all first class particles (and the corresponding sites), merging the respective edges, and preventing all updates along such newly created edges; see also Lemma~\ref{lem:MaxSecondClassOrder}.

\begin{remark}\label{rem:ASEPRestricted}
Note that Lemma \ref{lem:CensoringASEP} and Lemma \ref{lem:BlockingMeasureMaximum} continue to hold for a stationary  ASEP on a closed interval $I$ by censoring all moves outside of $I$ at all times; see also Proposition 5.1 in \cite{S:MixingBallistic} for a precise statement.
\end{remark}

\subsection{Current fluctuations for the ASEP on $\Z$}\label{sec:CurrentFluctuationsIntegers}

A key observable for the simple exclusion process is the number of  particles passing through a given site over time. For the ASEP $(\eta^{\Z}_t)$ on the integers, we define the \textbf{current}
$(\Jz_t(x))_{t \geq 0}$ through a site $x$ as the net number of particles which have passed through site $x$ until time $t$, i.e., set $\Jz_0(x)=0$ and increase (respectively decrease) the current by one  every time a particle jumps from $x-1$ to $x$ (respectively from $x$ to $x-1$). 
Note that $\Jz_t(x)$ is almost surely finite for all $t\geq 0$ and $x\in \Z$. Similarly, we denote by  $(\Jn_t(x))_{t \geq 0}$ the current of the ASEP on $\N$ through $x$ for all $x\in \N$, and by {$(\Jo^N_t(x))_{t \geq 0}$} the current of the open ASEP for all $x\in \lbr N \rbr$.  Note that the ergodic theorem Markov chain  ensures that
\begin{equation}\label{eq:AsymptoticCurrent}
J_N = \lim_{t \rightarrow \infty} \frac{\Jo^N_t(1)}{t} \qquad \text{almost surely,}
\end{equation} where $J_N$ is the \textbf{stationary current} of the open ASEP given by
\begin{equation}
J_N := \mu_N( \eta(1)=1 , \eta(2)=0 ) - q \mu_N( \eta(1)=0 , \eta(2)=1 ) .
\end{equation}
Similarly, suppose that the ASEP on $\N$ or $\Z$ is stationary, started from a Bernoulli-$\rho$-product measure with some $\rho \in [0,1]$. {Here, we assume for the ASEP on $\N$ that the effective density satisfies $\rho_{\mathsf{L}}=\rho$. We define by  
\begin{equation*}
J^{(\rho)} = (1-q)\rho(1-\rho)
\end{equation*} the respective stationary currents. In particular, note that
\begin{equation*}
\E[\Jz_t(0)] = J^{(\rho)} t \quad \text{and} \quad \E[\Jn_t(1)] = J^{(\rho)} t
\end{equation*} }for all $t \geq 0$. {Let us now study the stationary current of the open ASEP in more detail.} A simple observation from  Lemma~\ref{lem:Bernoulli} is that whenever $AC=1$, the stationary current of the respective open ASEP satisfies
\begin{equation}\label{eq:StatCurrent}
J_N = (1-q) \frac{A}{(1+A)^2} .
\end{equation}
We make the following observation on the stationary current of the open ASEP in the fan region of the (weakly) high and low density phase.

\begin{lemma}\label{lem:CurrentHighLow} Let $(\eta_t)_{t \geq 0}$ be an open ASEP with parameters $\alpha,\beta,\gamma,\delta \geq 0$  and $q\in [0,1)$ in the fan region of the high density phase, i.e., where $AC \leq 1$ and $A > \max(1,C)$, allowing the parameters also to depend on $N$. Then the stationary current satisfies for all $N \in \N$
\begin{equation}\label{eq:CurrentHighLow}
J_N \geq (1-q) \frac{A}{(1+A)^{2}} .
\end{equation}
Similarly, in the fan region of the low density phase, where $AC \leq 1$ and $C > \max(1,A)$, the stationary current satisfies
\begin{equation}\label{eq:CurrentLow}
J_N \geq (1-q) \frac{C}{(1+C)^{2}} .
\end{equation}
\end{lemma}
\begin{proof} We will consider only the high density phase as the result for the low density phase follows from the standard particle--hole symmetry. 
{Note that the stationary current $J_N$ is monotone increasing in the parameters $\alpha,\beta$. Recall $C=C(\alpha,\gamma,q)$ from \eqref{def:c}, and that $C$ is monotone decreasing in $\alpha$. Hence, since $C<1$ by our assumptions $AC\leq 1$ and $A>\max(1,C)$, we can find some $\alpha^{\prime} \leq \alpha$ such that $C^{\prime}=C^{\prime}(\alpha^{\prime},\gamma,q)$ satisfies $AC^{\prime}=1$, and so that the stationary current $J'_N$ of an open ASEP with respect to parameters $\alpha^{\prime},\beta,\gamma,\delta \geq 0$ satisfies
\begin{equation}
  J_N \geq   J'_N  = (1-q) \frac{A}{(1+A)^{2}} , 
\end{equation}
where the equality follows from observation \eqref{eq:StatCurrent}. }
\end{proof}

We remark that for constant parameters $\alpha,\beta,\gamma,\delta,q$, the lower bound in \eqref{eq:CurrentHighLow} is asymptotically sharp, {i.e., 
\begin{equation}
    \lim_{N \rightarrow \infty} J_N = (1-q)\frac{A}{(1+A)^2} ; 
\end{equation}
  see \cite{DEHP:ASEPCombinatorics,L:ErgodicI}.} We expect that a similar result can be achieved in the weakly high and low density phase using Askey--Wilson signed measures introduced in \cite{WWY:ASEPshock}. However, since we only require a lower bound on the stationary current in the (weakly) high and (weakly) low density phase, we leave this task to future work.

Next, we revisit the current $(\mathcal{J}^{\Z}_{t})_{t \geq 0}$ of the ASEP on $\Z$.  Recall the following recent moderate deviation result on the current of the ASEP on $\Z$ by Landon and Sosoe \cite{LS:Tails}, {which we will use in Section~\ref{sec:ModDevCurrent}. 
\begin{theorem}[c.f.~Theorem 2.4 in \cite{LS:Tails}]\label{thm:CurrentASEPOriginal} 
Let $\mathfrak{a} \in(0,1)$ and consider the current of the ASEP on $\Z$ started from a Bernoulli-$\rho$-product measure, where $\rho \in (\mathfrak{a},1-\mathfrak{a})$. Let $q\in (0,1)$. Then there exist constants $c_0,C_0>0$, which depend only on $\mathfrak{a}$, such that for all $T>0$ 
and $1 \leq w \leq T^{2/3}(1-q)^{2/3}$, 
\begin{equation*}
\P\left( \left| \frac{\Jz_T(0) - T\rho(1-\rho)(1-q)}{(1-q)^{1/3}T^{1/3}} \right| \geq  w \right) \leq C_0 \exp\left(-c_0 \min\left(w^{3/2}, \frac{w^2(1-q)^{2/3}T^{2/3}}{1+|(1-q)T(1-2\rho)|}\right) \right) . 
\end{equation*} 
\end{theorem}
\begin{proof}
This follows from Theorem 2.4 in \cite{LS:Tails} with $L-R=1-q$, $b=\rho$, $u=w$ and $x-x_0=(1-q)T(1-2\rho)$ in their notation.
\end{proof}

We will use Theorem~\ref{thm:CurrentASEPOriginal} in a slightly different parametrization (see Theorem~\ref{thm:CurrentASEP}) in order to compare the ASEP on $\Z$ to an ASEP on $\N$.
}

\begin{remark}\label{rem:ParametersCurrent}
Let us emphasize that we require $\rho \in (\mathfrak{a},1-\mathfrak{a})$ for some constant $\mathfrak{a}>0$ uniformly in $N$, so that the constants $C_0,c_0>0$ in Theorem~\ref{thm:CurrentASEPOriginal} only depend on $\mathfrak{a}$. This assumption is stated in Theorem~2.3 of \cite{LS:Tails} for the stochastic six vertex model and carries over to the ASEP in Theorem 2.4 of \cite{LS:Tails}, when applying the convergence result by Aggarwal from  \cite{A:Convergence}; see Section~7 in \cite{LS:Tails} for more details.
\end{remark}

\section{Current estimates in the maximal current phase}\label{sec:CurrentEstimates}

In this section, we consider the stationary current $J_N$ for the open ASEP  when the boundary rates $\alpha,\beta,\gamma,\delta$ and the bias $q$ depend on the system size $N$. We establish second order asymptotics on the stationary current in the maximal current phase relying on a representation by Uchiyama, Sasamoto and Wadati \cite{USW:PASEPcurrent}.
Recall the functions $A=A_N$ and $C=C_N$ from \eqref{def:a} and \eqref{def:c}, and define the constants $B,D \in [-1,0]$, when $\alpha>0$ as
\begin{equation}\label{def:d}
D=D(\alpha,\gamma,q) := \frac{1}{2 \alpha}\left( 1-q -\alpha+ \gamma - \sqrt{ (1-q -\alpha+ \gamma)^2 +4\alpha\gamma}\right),
\end{equation}
and similarly, when $\beta>0$ as
\begin{equation}\label{def:b}
B=B(\beta,\delta,q) := \frac{1}{2 \beta}\left( 1-q -\beta+ \delta - \sqrt{ (1-q -\beta+ \delta)^2 +4\beta\delta} \right)  .
\end{equation}
The expected current of the system of size $N$ is given by (see (2.11) and (6.1) of~\cite{USW:PASEPcurrent})
\begin{equation}\label{eq:CurrentInt}
J_N=\frac{Z_{N-1}}{Z_N} \quad \text{with} \quad
Z_N=\frac{1}{(1-q)^N}\frac{1}{4\pi\I} \oint_\Gamma \frac{dz}{z} e^{N f(z)} g(z;q) , 
\end{equation}
where we set
\begin{equation}\label{eq:fgDefinition}
\begin{aligned}
f(z)&=\ln(2+z+z^{-1}),\\
g(z;q)&=\frac{(z^2,z^{-2};q)_\infty}{(A z,A z^{-1};q)_\infty(B z,B z^{-1};q)_\infty(C z,C z^{-1};q)_\infty(D z,D z^{-1};q)_\infty}.
\end{aligned}
\end{equation}
The (anticlockwise) integration contour $\Gamma$ in \eqref{eq:CurrentInt} encloses the poles $\{a,a q, a q^2,\ldots\}$, but not the poles $\{a^{-1},a^{-1}q^{-1},a^{-1}q^{-2},\ldots\}$, where $a\in\{A,B,C,D\}$; see also \cite{BW:AskeyWilsonProcess,BW:QuadraticHarnesses} for similar results  using Askey--Wilson processes. In particular, when $A,B,C,D<1$, we can take the contour $\Gamma$ to be the circle of radius $1$. In \eqref{eq:fgDefinition}, we use the notation $(x,y;q)_\infty=(x;q)_\infty (y;q)_\infty$, where
\begin{equation}
(x;q)_\infty:=\prod_{k\geq 0}(1-x q^k)
\end{equation} is the $\mathbf{q}$\textbf{-Pochhammer symbol}.
From \eqref{eq:CurrentInt}, we get that
\begin{equation}\label{eqCurrent}
J_N=\frac{1-q}{4}\left[1+\frac{\oint \frac{dz}{4\pi\I z} e^{N f(z)} g(z;q)\left(\frac{4}{2+z+z^{-1}}-1\right)}{\oint \frac{dz}{4\pi\I z} e^{N f(z)} g(z;q)}\right].
\end{equation}
We will in the following, under assumption \eqref{def:TriplePointScaling} for $q$ with some $\kappa \in [0,\frac{1}{2}]$ and $\psi>0$, assume that the parameters $A,B,C,D$ satisfy
\begin{equation}\label{eq:ABCDscaling}
\begin{aligned}
&C = 1- \Theta(N^{-1/2}),\quad B = -1+\Theta(N^{-\kappa}),\\
&A = 1- \Theta(N^{-1/2}),\quad D= -1+\Theta(N^{-\kappa}),
\end{aligned}
\end{equation}
{where we recall the asymptotic notation from Section~\ref{sec:Notation}.}
\begin{figure}[t]
\begin{center}
\includegraphics[height=4cm]{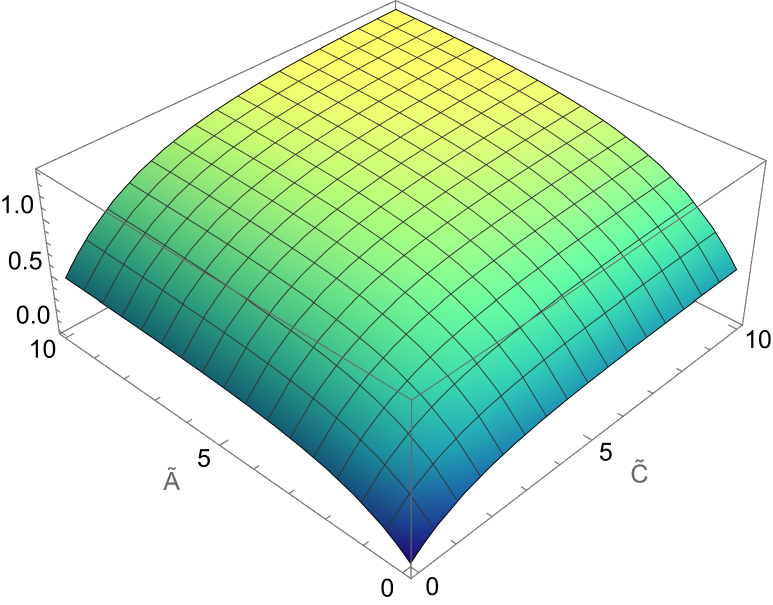}
\caption{Plot of the function $F(\tilde A,\tilde C)$ used in the proof of  Proposition \ref{pro:CurrentMaxCurrent1}.}
\label{FigMonotonicity}
\end{center}
\end{figure}

\subsection{Improved current estimates in the maximal current phase}

In the next two propositions, we get the asymptotics of the current for large $N$ under assumption \eqref{def:TriplePointScaling} for $q$ with some $\kappa \in [0,\frac{1}{2}]$ and $\psi>0$. We start with the case where $0\leq \kappa<\frac{1}{2}$.

\begin{proposition}
\label{pro:CurrentMaxCurrent1}
Consider the scaling $q=e^{-\psi N^{-\kappa}}$ for $\kappa\in [0,\frac{1}{2})$ and some $\psi>0$. Moreover, we assert that $A=e^{-\tilde A N^{-1/2}}$, $C=e^{-\tilde C N^{-1/2}}$, with some constants $\tilde A,\tilde C > 0$. We also set $B=-q e^{-\tilde B N^{-1/2}}$ and $D=-q e^{-\tilde D N^{-1/2}}$ with some $\tilde B,\tilde D \in \R$. Then, as $N\to\infty$,
\begin{equation}
J_N=\frac{1-q}{4} \left[1+\frac{1}{N}F(\tilde A,\tilde C)+o(N^{-1})\right],
\end{equation}
where we set
\begin{equation}\label{def:FunctionF}
F(\tilde A,\tilde C):=\frac{1}{4}\frac{\int_\R dx e^{-x^2/4} \frac{x^4}{(x^2+\tilde A^2)(x^2+\tilde C^2)}}{\int_\R dx e^{-x^2/4} \frac{x^2}{(x^2+\tilde A^2)(x^2+\tilde C^2)}}.
\end{equation}
\end{proposition}
A visualization of the function $F$ is given in  Figure~\ref{FigMonotonicity}.
In the proof of Proposition \ref{pro:CurrentMaxCurrent1}, we will provide two separate arguments to cover the (overlapping) regimes $\kappa \in [0,\frac{1}{10})$ and $\kappa \in (0,\frac{1}{2})$ separately.
Finally, we also cover the case where $\kappa=1/2$ as follows.

\begin{proposition}\label{pro:CurrentMaxCurrent3}
Let $q=e^{-\psi N^{-1/2}}$, $A=e^{-\tilde A N^{-1/2}}$, $C=e^{-\tilde C N^{-1/2}}$, $B=-q e^{-\tilde B N^{-1/2}}$, $D=-q e^{-\tilde D N^{-1/2}}$ for constants $\tilde{A},\tilde{C}> 0$ and $\tilde B,\tilde D > - \psi$. Then, as $N\to\infty$,
\begin{equation}\label{eq4.5}
J_N=\frac{1-q}{4}\left[1+\frac{1}{4N}\frac{\int_\R dx e^{-x^2/4} x^2 H(\tilde A,\tilde C;x)}{\int_\R dx e^{-x^2/4} H(\tilde A,\tilde C;x)} +\Or(N^{-9/8})\right],
\end{equation}
where we set
\begin{equation}\label{def:FunctionH}
H(\tilde A,\tilde C;x) = \frac{\Gamma((\tilde A+\I x)/\psi)\Gamma((\tilde A-\I x)/\psi)\Gamma((\tilde C+\I x)/\psi)\Gamma((\tilde C-\I x)/\psi)}{\Gamma(2\I x/\psi)\Gamma(-2\I x/\psi)}.
\end{equation}
\end{proposition}
{We will discuss strict monotonicity of the function $F$, and its consequences on the current $J_N$ in Section \ref{sec:Strict}.}

\subsection{Identities and expansions of Gamma functions}
In this short section, we collect some basic properties of Gamma and reciprocal Gamma functions. Below a complex number $z$ is decomposed as $z=x+\I y$, $x,y\in\R$. Note that $\Gamma(z)$ is an analytic function on $\C$ except for the points $z=0,-1,-2,-3,-4,\ldots$ where it has poles with residues $0,-1,1/2,-1/3!,1/4!,\ldots$. The reciprocal of the Gamma function, $1/\Gamma(z)$, is an entire function with simple zeroes at the points $\ldots,-3,-2,-1,0$.
The first identity is (see \textbf{6.1.29} in \cite{DAR:Pocketbook})
\begin{equation}\label{eqGammaIdentity}
\frac{1}{\Gamma(\I y)\Gamma(-\I y)}=\frac{y \sinh(\pi y)}{\pi}.
\end{equation}
From equation \textbf{6.1.15} of \cite{DAR:Pocketbook}, $\Gamma(z+1)=z \Gamma(z)$, we have
\begin{equation}\label{eqGamma3}
\Gamma(z)=\frac1z \Gamma(z+1)
\end{equation}
and, since $\Gamma(1)=1$ and $\Gamma$ is analytic and bounded in a neighborhood of $z=1$, for $|z|\to 0$ we get
\begin{equation}\label{eqGamma4}
\Gamma(z)=\frac1z (1+\Or(z)).
\end{equation}
As a consequence, as $\e\to 0$, it holds that
\begin{equation}\label{eqGamma1}
\Gamma((x+\I y) \e)\Gamma((x-\I y) \e)=\frac{1}{(x^2+y^2) \e^2} (1+\Or(\e)).
\end{equation}
Inequality \textbf{6.1.26} of \cite{DAR:Pocketbook} states that
\begin{equation}\label{eqGamma2}
|\Gamma(x+\I y)|\leq |\Gamma(x)|.
\end{equation}
Finally, for any $x>0,y\in\R$,
\begin{equation}\label{eqGamma7}
\Gamma(x+\I y)\Gamma(x-\I y)>0.
\end{equation}
This follows from
\begin{equation*}
\Gamma(x+\I y)\Gamma(x-\I y)=|\Gamma(x+\I y)|^2 e^{\I [\arg(x+\I y)+\arg(x-\I y)]}
\end{equation*}
as using formula \textbf{6.1.27} of \cite{DAR:Pocketbook}, we get that $\arg(x+\I y)+\arg(x-\I y)=0$.

\begin{remark}
Note that for $\tilde A,\tilde C>0$ and $x\in\R$, we see that $H$ from \eqref{def:FunctionH} satisfies $H(\tilde A,\tilde C;x)>0$ by \eqref{eqGammaIdentity} and \eqref{eqGamma7}. This means that the integrals in the numerator and denominator are of positive functions, i.e., there are no cancellations.
\end{remark}

\subsection{Asymptotics for $q$-Pochhammer expressions}
In the following, we parameterize the integration path as $z=e^{\I\phi}$. We collect some lemmas to control the contribution for $\phi$ away from $0$ as well as the expansions for $\phi$ close to $0$.

\begin{lemma}\label{lemF}
The critical point of $f(z)$ is $z=1$. Moreover, 
\begin{equation}\label{eq2.2b}
f(z)=\ln(4)+\frac14(z-1)^2+\Or((z-1)^3)
\end{equation}
and for all $\phi\in [-\pi,\pi)$,
\begin{equation}\label{eq2.1}
f(z=e^{\I\phi})\leq \ln(4)-\frac14 \phi^2.
\end{equation}
\end{lemma}
\begin{proof}
We have
\begin{equation*}
\frac{d}{dz} f(z)=\frac{z-1}{z(z+1)}
\end{equation*}
which is equal to $0$ for $z=1$. Moreover,
\begin{equation*}
\frac{d^2}{dz^2} f(z)\big|_{z=1}=\frac12
\end{equation*}
so \eqref{eq2.1} is obtained using the explicit parametrization $f(e^{\I\phi})=\ln[2(1+\cos(\phi))]$. Indeed, setting $g(\phi)=\ln(4)-\phi^2/4-\ln[2(1+\cos(\phi))]$, we have $g(0)=0$ and, for $\phi\in [0,\pi)$,
\begin{equation*}
\frac{d g(\phi)}{d\phi}= \tan(\phi/2)-\phi/2\geq 0,
\end{equation*}
which implies $g(\phi)\geq 0$, and thus \eqref{eq2.1}.
\end{proof}

The next estimates are useful in the case that $1-A$ goes to $1$ much faster than $1-q$ to $0$ as we let $N\to\infty$.

\begin{lemma}\label{lemQpochhammerKappa0Complex}
{For all $q\in (0,1)$ and $w\in\C$ with $|(1-w)/(1-q)|\leq \frac{1}{2}$, we get 
\begin{equation}
(q w;q)_\infty = (q;q)_\infty e^{\Or(|1-w|/(1-q)^2)},
\end{equation} 
i.e., there exists a constant $C>0$ such that 
\begin{equation*}
 (q;q)_\infty \exp(-C(|1-w|/(1-q)^2)\leq (q w;q)_\infty  \leq (q;q)_\infty \exp(C(|1-w|/(1-q)^2) . 
\end{equation*}
}
\end{lemma}
\begin{proof}
We start with
\begin{equation}\label{eq3.12}
\begin{split}
    \ln((q w;q)_\infty)-\ln((q;q)_\infty)&=\sum_{n\geq 0} \ln\left(\frac{1-w q^{n+1}}{1-q^{n+1}}\right) \\
    &=\sum_{n\geq 0} \ln\left(1+\frac{(1-w) q^{n+1}}{1-q^{n+1}}\right).
\end{split}
\end{equation}

Notice that for all $n\geq 0$ and $q\in [0,1)$, we have that
\begin{equation}\label{eqLog}
0\leq \frac{1}{1-q^{n+1}}\leq \frac{1}{1-q},\quad 0\leq q^{n+1}\leq q.
\end{equation}
Denoting $z=1-w$, our assumptions give $|z/(1-q)|\leq \frac{1}{2}$. Then by \eqref{eqLog} for all $n$ we also get $|z q^{n+1}/(1-q^{n+1})|\leq \frac{1}{2}$. Therefore, the series of the logarithm in the summands in \eqref{eq3.12} is convergent for all $n$. This implies
\begin{equation}\label{eq3.14}
\eqref{eq3.12}=\sum_{n\geq 0} \sum_{\ell\geq 1} \frac{(-1)^\ell z^\ell q^{\ell (n+1)}}{(1-q^{n+1})^\ell \ell}.
\end{equation}
Using \eqref{eqLog}, we have the bound
\begin{equation}\label{eq3.16}
\begin{aligned}
|\eqref{eq3.14}|&\leq \sum_{\ell\geq 1}\frac{|z|^\ell}{(1-q)^\ell \ell}\sum_{n\geq 1} q^{\ell n} = \sum_{\ell\geq 1}\frac{|z|^\ell}{(1-q)^\ell \ell} \frac{q^\ell}{1-q^\ell}\\
&\leq \frac{1}{1-q}\sum_{\ell\geq 1}\frac{|z|^\ell q^\ell}{(1-q)^\ell}=\frac{1}{1-q}\frac{|z|q}{1-q}\frac{1}{1-\frac{|z|q}{1-q}}\leq 2 \frac{|z| q}{(1-q)^2},
\end{aligned}
\end{equation}
as our assumptions imply $\frac{|z| q}{1-q}\leq \frac{1}{2}$.
\end{proof}

\begin{lemma}\label{lemQpochhammerKappa0Bis}
{For all $q\in (0,1)$ and $w\in\C$ with $|(1-w)/(1-q)|\leq \frac{1}{2}$, we get
\begin{equation}
    (-qw;q)_\infty = (-q;q)_\infty e^{\mathcal{O}(|1-w|/(1-q)^2)}.
\end{equation}
}
\end{lemma}
\begin{proof}
We have
\begin{equation}\label{eq3.12B}
\begin{split}
\ln((-q w;q)_\infty)-\ln((-q;q)_\infty)&=\sum_{n\geq 0} \ln\left(\frac{1+w q^{n+1}}{1+q^{n+1}}\right) \\
&=\sum_{n\geq 0} \ln\left(1-\frac{(1-w) q^{n+1}}{1+q^{n+1}}\right).
\end{split}
\end{equation}
Notice that for all $n\geq 0$ and $q\in [0,1)$, it holds that
\begin{equation}\label{eqLogB}
\frac{1}{1+q}\leq \frac{1}{1+q^{n+1}}\leq 1,\quad 0\leq q^{n+1}\leq q.
\end{equation}
Denoting $z=1-w$, we see that $|z/(1-q)|\leq \frac{1}{2}$. Then by \eqref{eqLogB} for all $n$, we get $|z q^{n+1}/(1+q^{n+1})|\leq \frac{1}{2}$. Therefore, the series of the logarithm is convergent for all $n$, so
\begin{equation}\label{eq3.14B}
|\eqref{eq3.12B}|\leq \sum_{n\geq 0} \sum_{\ell\geq 1} \frac{|z|^\ell q^{\ell (n+1)}}{(1+q^{n+1})^\ell \ell}
\leq  \sum_{\ell\geq 1} |z|^\ell \sum_{n\geq 0} q^{\ell (n+1)} = \sum_{\ell\geq 1}\frac{|z|^\ell q^\ell}{1-q^\ell}\leq \frac{2 |z| q}{(1-q)^2}
\end{equation}
since our assumptions imply $|z|q\leq \frac{1}{2}$.
\end{proof}

The following estimate will be used when $0<\kappa\leq \frac{1}{2}$. It is a special case of Proposition~2.3 in~\cite{CK:StationaryKPZ} (with $m=1$ and \emph{their} $\e=\frac{1}{4}$, $b=\frac{1}{2}$, $\alpha=2$).
\begin{lemma}[Special case of Proposition~2.3 of~\cite{CK:StationaryKPZ}]\label{lemCK}
Let $q=e^{-\e}$. Define the functions
\begin{equation}
\begin{aligned}
{\mathcal A^+}(\e,w)&=-\frac{\pi^2}{6\epsilon}-\Big(w-\frac{1}{2}\Big) \ln(\e)+\tfrac12\ln(2\pi)-\ln(\Gamma(w)),\\
{\mathcal A^-}(\e,w)&=\frac{\pi^2}{12 \epsilon}-\Big(w-\frac{1}{2}\Big)\ln(2).
\end{aligned}
\end{equation}
For all $w\in\C$ with $|\Im(w)|\leq 2/\e$,
\begin{align}
\ln((q^w;q)_\infty)&={\mathcal A^+}(\e,w)+E^+(\e,w),\label{eqCKplus}\\
\ln((-q^w;q)_\infty)&={\mathcal A^-}(\e,w)+E^-(\e,w),\label{eqCKminus}
\end{align}
with $E^\pm(\e,w)=\Or(\e (1+|w|)^2+\e^{1/2}(1+|w|)^{2+1/4})$.
\end{lemma}

\begin{lemma}\label{lemz2}
Let $q=e^{-\e}$ with $0<\e\leq 1$ and $z=e^{\I\phi}$. Then
\begin{equation}
|(z^2,z^{-2};q)_\infty|\leq 4 \sin(\phi)^2 (-q;q)_\infty^2 \leq c_0 e^{\pi^2/(6\e)}
\end{equation}
for some constant $c_0$. Furthermore, for $|\phi|\leq 1$,
\begin{equation}
(z^2,z^{-2};q)_\infty = \frac{e^{-\frac{\pi^2}{3\e} +\ln(\e)+\ln(2\pi)}}{\Gamma(-2\I\phi/\e)\Gamma(2\I\phi/\e)} e^{\Or(\phi^2\e^{-7/4};\sqrt{\e})}.
\end{equation}
{Here, we write $\Or(\phi^2\e^{-7/4};\sqrt{\e})$ as a short-hand when $\Or(\phi^2\e^{-7/4})$ as well as $\Or(\sqrt{\e})$ hold.}
\end{lemma}
\begin{proof}
Taking out the first terms in the $q$-Pochhammer products, we get
\begin{equation*}
\begin{aligned}
(z^2,z^{-2};q)_\infty &= -(z^{-1}-z)^2 (q z^2,q z^{-2};q)_\infty\\
&=4\sin^2(\phi)\prod_{k\geq 0} (1-e^{2\I\phi}q^{k+1})(1-e^{-2\I\phi}q^{k+1}).
\end{aligned}
\end{equation*}
The last product is bounded by $\left((-q,q)_\infty\right)^2$. By Lemma~\ref{lemCK} (Equation~\eqref{eqCKminus} with $w=1$), we get that
\begin{equation*}
\left((-q,q)_\infty\right)^2\leq C_0 e^{\pi^2/(6\e)}
\end{equation*}
for some constant $C_0>0$.
Next, notice that $z^2=e^{2\I\phi}=q^w=e^{-\e w}$ for $w=-2\I \phi/\e$. As $|\phi|\leq 1$ gives $|\Im(w)|\leq 2 /\e$, we can use the expansion of Lemma~\ref{lemCK} to see that
\begin{equation*}
(z^2;q)_\infty = \frac{e^{-\frac{\pi^2}{6\e} +(2\I\phi/\e+1/2)\ln(\e)+\tfrac12 \ln(2\pi)}}{\Gamma(-2\I\phi/\e)} e^{\Or(\phi^2\e^{-7/4};\sqrt{\e})},
\end{equation*}
where $\Or(\phi^2\e^{-7/4})$ comes from the case where $|w|\gg 1$ and $\sqrt{\e}$ from the case where $|w|$ remains bounded. Similarly, by replacing $\phi$ with $-\phi$, we get the expansion for $(z^{-2};q)_\infty$. Multiplying them together we get the claimed result.
\end{proof}
The next result gives us the asymptotics for  $A$ close to the critical point.
\begin{lemma}\label{lemAN}
Let $q=e^{-\e}$ with $0<\e\leq 1$, $A=e^{-\tilde A N^{-1/2}}$ for some $\tilde A\geq 0$ and $z=e^{\I\phi}$. Then
\begin{equation}
|(A z,A z^{-1};q)_\infty|\geq c_0 \tilde A^2 N^{-1}  e^{-\tilde c /\e}
\end{equation}
for some constants $c_0,\tilde c>0$. Furthermore, for $|\phi|\leq 1$, we see that
\begin{equation}
(A z,A z^{-1};q)_\infty = \frac{e^{-\frac{\pi^2}{3\e}+\ln(\e)+\ln(2\pi) -\frac{2\tilde A\ln(\e)}{N^{1/2}\e}}}{\Gamma\left(\frac{\tilde A}{N^{1/2}\e}-\I\frac{\phi}{\e}\right)\Gamma\left(\frac{\tilde A}{N^{1/2}\e}+\I\frac{\phi}{\e}\right)} e^{\Or(\phi^2\e^{-7/4};\sqrt{\e})} .
\end{equation}
\end{lemma}
\begin{proof}
We have that
\begin{equation*}
\begin{aligned}
|(A z,A z^{-1};q)_\infty| &= \bigg|\prod_{k\geq 0}(1-A z q^{k})(1-A z^{-1} q^{k})\bigg| \\
&\geq \prod_{k\geq 0} (1-A q^k)^2=\left((A,q)_\infty\right)^2 \geq (1-A)^2 \left((q;q)_\infty\right)^2 ,
\end{aligned}
\end{equation*}
where we used $|A|\leq 1$ for the first inequality. Note that $(1-A)^2\simeq \tilde A^2 N^{-1}$ and $(q;q)_\infty\simeq e^{-\pi^2/(6\e)}$.
For the expansion around $\phi=0$, we can use Lemma~\ref{lemCK}. Defining $w$ by $A z = q^w$, we have that $w=\tilde A/(N^{1/2}\e)-\I \phi/\e$. This leads to
\begin{equation*}
(A z;q)_\infty=\frac{e^{-\frac{\pi^2}{6\e} -\left(\frac{\tilde A}{N^{1/2}\e}-\frac{\I\phi}{\e}-\frac12\right) \ln(\e)+\tfrac12\ln(2\pi)}}{\Gamma\left(\frac{\tilde A}{N^{1/2}\e}-\I\frac{\phi}{\e}\right)} e^{\Or(\phi^2\e^{-7/4};\sqrt{\e})} .
\end{equation*}
{Substituting $\phi$ by $-\phi$}, we get the expression for $(A z^{-1};q)_\infty$. Multiplying them leads to
\begin{equation*}
(A z,A z^{-1};q)_\infty=\frac{e^{-\frac{\pi^2}{3\e}+\ln(\e)+\ln(2\pi) -\frac{2\tilde A\ln(\e)}{N^{1/2}\e}}}{\Gamma\left(\frac{\tilde A}{N^{1/2}\e}-\I\frac{\phi}{\e}\right)\Gamma\left(\frac{\tilde A}{N^{1/2}\e}+\I\frac{\phi}{\e}\right)} e^{\Or(\phi^2\e^{-7/4};\sqrt{\e})},
\end{equation*}
which is the claimed result.
\end{proof}

The next result tells us that in the asymptotics, the terms with $B$ (and $D$) can be replaced by constants close to the critical point.
\begin{lemma}\label{lemBN}
Let $q=e^{-\e}$ with $0<\e\leq 1$, $B=-q e^{-\tilde B N^{-1/2}}$ for some $\tilde B >-\e N^{1/2}$ and $z=e^{\I\phi}$. Then
\begin{equation}
|(B z,B z^{-1};q)_\infty|\geq c_0 (\tilde B N^{-1/2}+\e)^2 e^{-\tilde c/\e}
\end{equation}
for some constants $c_0,\tilde c>0$. Furthermore, for $|\phi|\leq 1$,
\begin{equation}
(B z,B z^{-1};q)_\infty = e^{\frac{\pi^2}{6\e}-2\ln(2)-2\frac{\tilde B}{N^{1/2}\e}\ln(2)} e^{\Or(\phi^2\e^{-7/4};\sqrt{\e})} .
\end{equation}
\end{lemma}
\begin{proof}
The first bound is as in Lemma~\ref{lemAN}, replacing $\phi$ with $\pi-\phi$.
The expansion for small $\phi$ is obtained using the second formula of Lemma~\ref{lemCK}. Setting $w$ through $B z = -q^w$ we get $w=-\I \phi/\e+1+\tilde B/(N^{1/2}\e)$, and hence
\begin{equation*}
(B z;q)_\infty = e^{\frac{\pi^2}{12\e}-\left(1-\frac{\I\phi}{\e}+\frac{\tilde B}{N^{1/2}\e}\right)\ln(2)} e^{\Or(\phi^2\e^{-7/4};\sqrt{\e})} ,
\end{equation*}
and similarly for $(B z^{-1};q)_\infty$. Multiplying the two terms, at first order the dependence on $\phi$ vanishes, and we get the claimed result.
\end{proof}

\begin{remark}\label{rem:ExtraFactor}
We will use the following estimate to bound an extra factor in the numerator in \eqref{def:FunctionF}. For $z=e^{\I\phi}$, we have that
\begin{equation}\label{eq:RHSTan}
\frac{4}{2+z+z^{-1}}-1 = (\tan(\phi/2))^2.
\end{equation}
As $\phi\to \pi$, the right-hand side in \eqref{eq:RHSTan} diverges, but will be compensated by the factor of $e^{N f(z)}$, which converges to $0$ much faster. Furthermore, for $\phi$ small, we have that
\begin{equation}
\frac{4}{2+z+z^{-1}}-1 = \frac14 \phi^2 (1+\Or(\phi^2)).
\end{equation}
\end{remark}

\subsection{Proof of the improved current estimates}

We have now all tools to show the improved estimates on the stationary current in the maximal current phase.

\begin{proof}[Proof of Proposition \ref{pro:CurrentMaxCurrent1} for $\kappa\in [0,\frac{1}{10})$ ]
Consider first the integral in the denominator of \eqref{eqCurrent}. Notice that for any $a\in [-1,1)$ and $q\in [0,1)$, with $z=e^{\I\phi}$,
\begin{equation}\label{eq1.35}
0<(a;q)_\infty^2\leq |(a z,a z^{-1};q)_\infty| \leq (-1;q)_\infty^2<\infty.
\end{equation}
Also, we can write
\begin{equation*}
(z^2,z^{-2};q)=(1-z^2)(1-z^{-2})(qz^2,qz^{-2};q)_\infty
\end{equation*}
as well as
\begin{equation*}
(A z,A z^{-1};q)=(1-A z)(1-A z^{-1})(q A z,q A z^{-1};q)_\infty ,
\end{equation*}
and similarly for $B$, $C$ and $D$. For the choice of $A,B,C,D$, the rough bounds from \eqref{eq1.35} implies that $|g(z,q)|\leq c_1 e^{c_2 N^{1/10}}$ for some constants $c_1,c_2>0$ (if $\kappa=0$ it is even only polynomial in $N$, while for $\kappa>0$ it grows at most as $e^{c_2 N^\kappa}$). This term is dominated by $e^{-N \phi^2/4}$ for $|\phi|>N^{-2/5}$. Therefore, the contribution to $\oint \frac{dz}{4\pi\I z} e^{N f(z)} g(z;q)$ for $|\phi|>N^{-2/5}$ is of order $\Or(e^{-N^{1/5}})$ smaller than the leading term.

Next consider $|\phi|\leq N^{-2/5}$. By using Taylor expansion for the simple factors and Lemma~\ref{lemQpochhammerKappa0Complex} applied to the terms with the $q$-Pochhammer terms we get
\begin{equation*}
\begin{aligned}
(z^2,z^{-2};q)&=(1-z^2)(1-z^{-2})(qz^2,qz^{-2};q)_\infty\\
&=4\phi^2 (1+\Or(N^{-4/5}))((q;q)_\infty)^2 e^{\Or(N^{-(2/5-2\kappa)})},
\end{aligned}
\end{equation*}
as well as
\begin{equation*}
\begin{aligned}
(A z,A z^{-1};q)&=(1-A z)(1-A z^{-1})(q A z,q A z^{-1};q)_\infty\\
&=(\tilde A^2/N+\phi^2)(1+\Or(N^{-2/5})) ((q;q)_\infty)^2 e^{\Or(N^{-(2/5-2\kappa)})} ,
\end{aligned}
\end{equation*}
and similarly for $C$. As a consequence, if we denote $\phi=x N^{-1/2}$, we have that
\begin{equation*}
\frac{1}{(A z,A z^{-1};q)}=\frac{N}{\tilde A^2+x^2}\frac{1}{((q;q)_\infty)^2} (1+\Or(N^{-(2/5-2\kappa)})).
\end{equation*}
Using Lemma~\ref{lemQpochhammerKappa0Bis}, we get that
\begin{equation*}
\begin{aligned}
(B z,B z^{-1};q)&=(1-B z)(1-B z^{-1})(q A z,q A z^{-1};q)_\infty\\
&=((-q;q)_\infty)^2  e^{\Or(N^{-(2/5-2\kappa)})}(1+\Or(N^{-2/5})) .
\end{aligned}
\end{equation*}
Finally, we have
\begin{equation*}
e^{N f(z)}=e^{N \ln(4)} e^{-N \phi^2/4} e^{\Or(N \phi^3)}=e^{N \ln(4)} e^{-N \phi^2/4} e^{\Or(N^{-1/5})}.
\end{equation*}
Collecting all the terms we get, with $\phi=x N^{-1/2}$, that
\begin{equation*}
g(z;q)\sim \frac{4 Nx^2}{(\tilde A^2+x^2)(\tilde C^2+x^2) ((q;q)_\infty)^2 ((-q;q)_\infty)^4} e^{N \ln(4)-x^2/4}
\end{equation*}
up to errors of order $\Or(N^{-(2/5-2\kappa)}; N^{-1/5})=\Or(N^{-1/5})$.
Removing the error terms for the integral over $|x|\leq N^{1/10}$ leads to an error term of at most $\Or(N^{-1/5})$ with respect to the leading term. After this step, extending the integral from $|x|\leq N^{1/10}$ for $x\in\R$ can be made up to an error of order $\Or(e^{-N^{1/5}})$.
Summing up, the denominator is given by
\begin{equation*}
\Xi(N,q) (1+\Or(N^{-1/5})) \int_\R dx e^{-x^2/4} \frac{x^2}{(x^2+\tilde A^2)(x^2+\tilde C^2)}
\end{equation*}
with $\Xi(N,q)=e^{N\ln(4)}\frac{4N}{((q;q)_\infty)^2 ((-q;q)_\infty)^4}$. Similarly, one deals with the numerator, which is given by
\begin{equation*}
\Xi(N,q) (1+\Or(N^{-1/5})) \frac1{N}\int_\R dx e^{-x^2/4} \frac{x^4}{(x^2+\tilde A^2)(x^2+\tilde C^2)}.
\end{equation*}
Taking the ratio, we get the claimed result.
\end{proof}

Next, we consider the case where $\kappa \in (0,\frac{1}{2})$.

\begin{proof}[Proof of Proposition \ref{pro:CurrentMaxCurrent1} for $\kappa\in (0,\frac{1}{2})$ ]
Let us parameterize $z=e^{\I\phi}$ with $\phi\in[-\pi,\pi)$. We decompose the integration into the following subsets:
\begin{equation}\label{def:Set012}
\Gamma_0=\{e^{\I\phi},-\delta_1\leq \phi\leq \delta_1\},\quad  \Gamma_1=\{e^{\I\phi},|\phi|\in(\delta_1,1]\}, \quad\Gamma_2=\{e^{\I\phi},|\phi|\in(1,\pi/2]\}.
\end{equation}
Here, we choose $\delta_1=N^{-(2\kappa+1)/4}$. We also define a small parameter $\e$ by the relation $q=e^{-\e}$, i.e.,
\begin{equation*}
\e=\psi N^{-\kappa}\gg N^{-1/2}.
\end{equation*}
We first consider the integral in the denominator. The analysis for the numerator is almost identical, except for the extra term $\frac{4}{2+z+z^{-1}}-1$; see Remark \ref{rem:ExtraFactor}. We bound the contributions from $\Gamma_2$ and $\Gamma_1$, and then determine the asymptotics from the contribution by~$\Gamma_0$.

\subsubsection{(a) \textit{Bound for $z\in\Gamma_2$}} For $z\in\Gamma_2$, by Lemma~\ref{lemF}, we have that
\begin{equation*}
|e^{N f(z)}|\leq e^{N\ln(4)} e^{-N/4}.
\end{equation*}
By Lemma~\ref{lemz2}, we get that
\begin{equation*}
|(z^2,z^{-2};q)_\infty|\leq c_0 e^{\pi^2N^\kappa/(6\psi)}\leq c_0 e^{\pi^2 N^{1/2}/6},
\end{equation*}
where in the second inequality, we used that $1/\e \gg N^{-1/2}$. Similarly, by Lemma~\ref{lemAN},
\begin{equation*}
\frac{1}{|(A z,A z^{-1};q)_\infty (C z,C z^{-1};q)_\infty|}\leq c_1 \frac{N^2}{\tilde A^2 \tilde C^2} e^{2\tilde c N^{1/2}}
\end{equation*}
for some constants $c_1,\tilde c>0$.
Finally, Lemma~\ref{lemBN} gives that
\begin{equation*}
\frac{1}{|(B z,B z^{-1};q)_\infty (D z,D z^{-1};q)_\infty |}\leq c_1 N^{2} e^{2\tilde c N^{1/2}}
\end{equation*}
for some constant $c_1>0$.
Therefore,
\begin{equation}\label{eq1.16}
e^{-N \ln(4)}\left|\int_{\Gamma_2} \frac{dz}{4\pi\I z} e^{N f(z)} g(z;q)\right|\leq  c_2 e^{c_3 N^{1/2}}e^{-N/4}
\end{equation}
for some constants $c_2,c_3>0$.

\subsubsection{(b) \textit{Bound for $z\in\Gamma_1$}} For $z\in\Gamma_1$, by Lemma~\ref{lemF} we have that
\begin{equation*}
|e^{N f(z)}|\leq e^{N\ln(4)} e^{-N\phi^2/4}.
\end{equation*}
Using that $\kappa\leq \frac{1}{2}$ in the error term, Lemma~\ref{lemz2} yields
\begin{equation*}
(z^2,z^{-2};q)_\infty= \frac{e^{-\frac{\pi^2}{3\e} +\ln(\e)+\ln(2\pi)}}{\Gamma(-2\I\phi/\e)\Gamma(2\I\phi/\e)} e^{\Or(\phi^2 N^{7/8};N^{-\kappa/2})}
\end{equation*}
while Lemma~\ref{lemAN} leads to
\begin{equation*}
\begin{aligned}
&\frac{1}{(A z,A z^{-1};q)_\infty (C z,C z^{-1};q)_\infty}\\
&= \frac{\Gamma\left(\frac{\tilde A}{N^{1/2}\e}-\I\frac{\phi}{\e}\right)
\Gamma\left(\frac{\tilde A}{N^{1/2}\e}+\I\frac{\phi}{\e}\right)\Gamma\left(\frac{\tilde C}{N^{1/2}\e}-\I\frac{\phi}{\e}\right)\Gamma\left(\frac{\tilde C}{N^{1/2}\e}+\I\frac{\phi}{\e}\right)}{e^{-\frac{2}{3}\pi^2\e^{-1}+2\ln(\e)+2\ln(2\pi) -2(\tilde A+\tilde C)\ln(\e)N^{-1/2}\e^{-1}}} e^{\Or(\phi^2 N^{7/8};N^{-\kappa/2})}.
\end{aligned}
\end{equation*}
Finally, by Lemma~\ref{lemBN}, we get that
\begin{equation*}
\frac{1}{(B z,B z^{-1};q)_\infty (D z,D z^{-1};q)_\infty}= e^{-\frac{\pi^2}{3\e}+4\ln(2)+2(\tilde B+\tilde D)N^{-1/2}\e^{-1}\ln(2)} e^{\Or(\phi^2 N^{7/8};N^{-\kappa/2})}.
\end{equation*}
Collecting all the terms which are not dependent on $\phi$, we define
\begin{equation}\label{def:XiN}
\Xi_N=e^{-\ln(\e)+4\ln(2)-\ln(2\pi)}e^{2(\tilde A+\tilde C)\ln(\e)(N^{1/2}\e)^{-1}}e^{2(\tilde B+\tilde D)(N^{1/2}\e)^{-1}\ln(2)}.
\end{equation}
The terms involving the Gamma functions are given by
\begin{equation*}
G_N(\phi)=\frac{\Gamma\left(\frac{\tilde A}{N^{1/2}\e}-\I\frac{\phi}{\e}\right)
\Gamma\left(\frac{\tilde A}{N^{1/2}\e}+\I\frac{\phi}{\e}\right)\Gamma\left(\frac{\tilde C}{N^{1/2}\e}-\I\frac{\phi}{\e}\right)\Gamma\left(\frac{\tilde C}{N^{1/2}\e}+\I\frac{\phi}{\e}\right)}{\Gamma(-2\I\phi/\e)\Gamma(2\I\phi/\e)}.
\end{equation*}
In these terms $\frac{\tilde A}{N^{1/2}\e}\ll \frac{\phi}{\e}$ since $1\geq \phi\geq \delta_1$. By the identity \eqref{eqGammaIdentity} on Gamma functions, we have that $\frac{1}{\Gamma(-2\I\phi/\e)\Gamma(2\I\phi/\e)}\sim e^{2\pi \phi/\e}$, which is much smaller than $e^{-N\phi^2/4}$. Thus,
\begin{equation*}
e^{-N\phi^2/4} e^{2\pi \phi/\e} = e^{-N\phi^2 \left(\frac14-2 \pi \frac{N^{\kappa-1}}{\phi}\right)}
\end{equation*}
and, for $1\geq \phi\geq \delta_1$,
\begin{equation}\label{eq1.24}
2 \pi \frac{N^{\kappa-1}}{\phi} \leq 2\pi N^{\frac34(2\kappa-1)} \to 0\textrm{ as }N\to\infty.
\end{equation}
By \eqref{eqGamma1} and \eqref{eqGamma2}, the terms with the Gamma functions in the numerator grow at most with a power law in $N$,  and thus are much smaller than the contribution of the two Gamma functions in the denominator (these are very rough bounds, but good enough for $\phi$ in $\Gamma_1$). Thus, for all $N$ large enough,
\begin{equation*}
G_N(\phi)\leq e^{N\phi^2/16}.
\end{equation*}
Furthermore, all the error terms $e^{\Or(\phi^2 N^{7/8};N^{-\kappa/2})}$ together, are also bounded by $e^{N\phi^2/16}$ for all $N$ large enough. As a consequence, for all $N$ large enough, all the $\phi$-dependent terms are bounded by $e^{-N\phi^2/8}$. Since we have $|\phi|\geq \delta_1=N^{-(2\kappa+1)/4}$, this yields a contribution of order $e^{-c_4 N^{1/2-\kappa}}$. Thus, altogether we have
\begin{equation}\label{eq1.20}
e^{-N\ln(4)}\left|\int_{\Gamma_1} \frac{dz}{4\pi\I z} e^{N f(z)} g(z;q)\right|\leq \Xi_N e^{-c_5 N^{(1-2\kappa)/2}}
\end{equation}
for some constant $c_5>0$. We remark that the bound in \eqref{eq1.16} is subleading with respect to \eqref{eq1.20}.

\subsubsection{(c) \textit{Expansion for $z\in\Gamma_0$}} Consider $|\phi|\leq N^{-(2\kappa+1)/4}$. All the expansions we have collected for case (b) still holds true for $z\in\Gamma_0$, i.e.,
\begin{equation}\label{eq1.28}
e^{-N\ln(4)}\int_{\Gamma_0} \frac{dz}{4\pi\I z} e^{N f(z)} g(z;q) = \Xi_N \int_{-N^{-(2\kappa+1)/4}}^{N^{-(2\kappa+1)/4}} \frac{d\phi}{4\pi} e^{-N\phi^2/4} G_N(\phi) e^{\Or(N^{7/8} \phi^2;N^{-\kappa/2})},
\end{equation}
where we recall $\Xi_N$ from \eqref{def:XiN}. By the change of variable $\phi=N^{-1/2} x$, we obtain
\begin{equation}\label{eq1.29}
\eqref{eq1.28} =  \Xi_N N^{-1/2} \int_{-N^{(1-2\kappa)/4}}^{N^{(1-2\kappa)/4}} \frac{dx}{4\pi} e^{-x^2/4} G_N(x N^{-1/2}) e^{\Or(x^2 N^{-1/8};N^{-\kappa/2})}.
\end{equation}
Letting $\tilde\e=N^{\kappa-1/2}/\psi$, the terms involving Gamma functions are given by
\begin{equation*}
G_N(x N^{-1/2})=\frac{\Gamma\big((\tilde A-\I x)\tilde\e\big)
\Gamma\big((\tilde A+\I x)\tilde\e\big)\Gamma\big((\tilde C-\I x)\tilde\e\big)\Gamma\big((\tilde C+\I x)\tilde\e\big)}{\Gamma(-2\I x \tilde\e)\Gamma(2\I x \tilde\e)}.
\end{equation*}
The error term in \eqref{eq1.29} can be removed up to an error $\Or(N^{-1/8};N^{-\kappa/2})$ with respect to the leading term (using the usual inequality $|e^y-1|\leq |y| e^{|y|}$ with $y$ replaced by the error term).
Notice that the integration is now over $|x|\leq N^{(1-2\kappa)/4}$, which means that $|\tilde\e x|\leq \frac{1}{\psi} N^{(2\kappa-1)/4}\to 0$ as $N\to\infty$. Therefore, all entries in the Gamma functions are very small. More precisely, using the expansion \eqref{eqGamma1}, we get that
\begin{equation}\label{eq1.30}
G_N(x N^{-1/2}) = \frac{4 x^2}{(\tilde A^2+x^2)(\tilde C^2+x^2) \tilde\e^2} (1+\Or(x \tilde\e)).
\end{equation}
At this point all the error terms can be estimated by $\Or(N^{(2\kappa-1)/4}; N^{-1/8}; N^{-\kappa/2})=o(1)$, and are thus  smaller than the leading term. Removing the error terms and then extending the integration over $x\in\R$, the error term is not larger than the ones we already have. Consequently, for the denominator, we get that
\begin{equation*}
\int_{|z|=1} \frac{dz}{4\pi\I z} e^{N f(z)} g(z;q) = e^{N\ln(4)}\Xi_N N^{-1/2} \tilde\e^{-2}  (1+o(1))\int_\R \frac{dx}{4\pi} e^{-x^2/4}\frac{4 x^2}{(\tilde A^2+x^2)(\tilde C^2+x^2)}
\end{equation*}
with $o(1)=\Or(N^{(2\kappa-1)/4}; N^{-1/8}; N^{-\kappa/2})$. The prefactors will cancels exactly with the ones of the numerator.
The computations for the denominators are essentially the same. The only difference is that we have the additional factor $\frac{4}{2+z+z^{-1}}-1$, which by Remark \ref{rem:ExtraFactor}, under the change of variables $z=e^{\I x N^{-1/2}}$, is given by
\begin{equation*}
\frac{4}{2+z+z^{-1}}-1 = \frac{x^2}{4N} (1+\Or(x^2/N)).
\end{equation*}
Putting everything together we get the claimed result with an error term given by $o(N^{-1})=N^{-1}\Or(N^{(2\kappa-1)/4}; N^{-1/8}; N^{-\kappa/2})$.
\end{proof}

Given the proof of Proposition \ref{pro:CurrentMaxCurrent1} for $\kappa \in (0,\frac{1}{2})$, we now describe the necessary adjustments in order deduce Proposition \ref{pro:CurrentMaxCurrent3} for $\kappa=\frac{1}{2}$.

\begin{proof}[Sketch of proof of Proposition \ref{pro:CurrentMaxCurrent3} ]
The proof in this case is almost identical to the one of Proposition~\ref{pro:CurrentMaxCurrent1}. Since the Gamma functions remain in the final expression, it is even easier. This time we can simply take $\delta_1=N^{-1/4}$. The arguments apply mutatis mutandis,  except that we do not have to expand the Gamma functions, i.e., we do not need the expansion \eqref{eq1.30}. It is in that step that we used the fact that we can take  $\delta_1=N^{-(2\kappa+1)/4}$ to ensure that $|\tilde \e x|\to 0$. Everything else works also for $\delta_1=N^{-1/4}$ and collecting the error terms with the new choice of $\delta_1$, we get an error term of order $\Or(N^{-1/8})$ smaller than the leading term.
\end{proof}

\subsection{Strict mononicity of the second order current estimates}\label{sec:Strict}
For the proof of Theorem~\ref{thm:MixingTimesTriple}, we will require that  the function $F(\tilde A,\tilde C)$ from \eqref{def:FunctionF} is {strictly monotone increasing in $\tilde A$}. This is the content of the following lemma.
\begin{lemma}\label{lemMonotonicity}
The function $F(\tilde A,\tilde C)$ is positive, symmetric in $\tilde A,\tilde C$, and for fixed $\tilde C\geq 0$, it is strictly increasing in $\tilde A$. Moreover, we have that $\lim_{\tilde A,\tilde C\to\infty} F(\tilde A,\tilde C)=\frac32$.
\end{lemma}
\begin{proof}
The only non-trivial property to be verified is the strict monotonicity. Let us define $\rho(x)=\frac{x^2}{x^2+\tilde C^2} e^{-x^2/4}$. Then
\begin{equation}\label{def:F}
F(\tilde A,\tilde C)=\frac14 \frac{\int dx \rho(x) \frac{x^2}{x^2+\tilde A^2}}{\int dx \rho(x) \frac{1}{x^2+\tilde A^2}} ,
\end{equation}
where here and below all the integrals are over $\R$.
Thus
\begin{equation*}
\frac{d F(\tilde A,\tilde C)}{d\tilde A} = \frac{\int dx \rho(x) \frac{x^2(-2\tilde A)}{(x^2+\tilde A^2)^2}\int dy \rho(y) \frac{1}{y^2+\tilde A^2}
-\int dx \rho(x) \frac{(-2\tilde A)}{(x^2+\tilde A^2)^2}\int dy \rho(y) \frac{y^2}{y^2+\tilde A^2}}{\left(\int dx \rho(x) \frac{1}{x^2+\tilde A^2}\right)^2}.
\end{equation*}
Denoting $f(x)=\rho(x)/(x^2+\tilde A^2)^2$, we can rewrite the derivative as
\begin{equation*}
\begin{aligned}
\frac{d F(\tilde A,\tilde C)}{d\tilde A} &= \frac{\int dx f(x) x^2(-2\tilde A)\int dy f(y) (y^2+\tilde A^2)
-\int dx f(x) (-2\tilde A)\int dy f(y) y^2(y^2+\tilde A^2)}{\left(\int dx f(x) (x^2+\tilde A^2)\right)^2}\\
&= 2 \tilde A \frac{\int dx f(x) \int dy f(y) y^2(y^2+\tilde A^2)-\int dx f(x) x^2\int dy f(y) (y^2+\tilde A^2)}{\left(\int dx f(x) (x^2+\tilde A^2)\right)^2}.
\end{aligned}
\end{equation*}
Thus, we need to prove that
\begin{equation*}
\int dx f(x) \int dy f(y) y^2(y^2+\tilde A^2)>\int dx f(x) x^2\int dy f(y) (y^2+\tilde A^2),
\end{equation*}
which is equivalent to showing that
\begin{equation*}
\int dx f(x) \int dy f(y) y^4>\int dx f(x) x^2\int dy f(y) y^2.
\end{equation*}
If we denote by $d\mu(x)=f(x) dx$, this rewrites as
\begin{equation*}
\left(\langle 1 , x^2 \rangle_{L^2(d\mu)}\right)^2 < \langle 1 , 1 \rangle_{L^2(d\mu)} \langle x^2 , x^2 \rangle_{L^2(d\mu)} = \| 1 \|^2_{L^2(d\mu)} \| x^2\|_{L^2(d\mu)}.
\end{equation*}
Notice that this, with $\leq$ is nothing else than the Cauchy-Schwarz inequality. The strict inequality holds true since the functions $1$ and $x^2$ are not collinear.
Finally, the limit as $\tilde A,\tilde C\to\infty$ is given by
\begin{equation*}
\frac14\frac{\int dx e^{-x^2/4} x^4}{\int dx e^{-x^2/4} x^2}=\frac32,
\end{equation*}
which allows us to conclude.
\end{proof}

\begin{remark}
Note that the limit $\lim_{\tilde A,\tilde C\to\infty} F(\tilde A,\tilde C)=\frac32$ in Lemma \ref{lemMonotonicity} agrees with the fact that when $\alpha=\beta=1$ and $q=0$, the stationary current satisfies  $J_N=Z_{N-1}/ Z_{N}=\mathsf{C}_{N}/ \mathsf{C}_{N+1}=\frac{1}{4}+\frac{3}{8N}+o(N^{-1})$, where $\mathsf{C}_N$ denotes the $N^{\text{th}}$ Catalan number; see for example Section 2.3 in \cite{NS:Stationary}.
\end{remark}

Similarly, we require for the proof of Theorem \ref{thm:MixingTimesKPZ} strict monotonicity for all sufficiently large $\tilde{A},\tilde{C}>0$ for the function $\tilde F(\tilde{A},\tilde{C})$ defined in the following.

\begin{lemma}\label{lemMonotonicityBis}
Let us define
\begin{equation}\label{def:Ftilde}
\tilde F(\tilde A,\tilde C):=\frac{1}{4}\frac{\int_\R dx e^{-x^2/4} x^2 H(\tilde A,\tilde C;x)}{\int_\R dx e^{-x^2/4} H(\tilde A,\tilde C;x)}.
\end{equation}
For any given $\tilde C$, for $\tilde A\gg 1$,
\begin{equation}
\frac{d\tilde F(\tilde A,\tilde C)}{d\tilde A}>0.
\end{equation}
\end{lemma}
\begin{proof}
Let us rewrite $\tilde F(\tilde A,\tilde C)$ as
\begin{equation*}
\tilde F(\tilde A,\tilde C)=\frac14\frac{\int_\R d x  x^2 \rho( x)g(\tilde A, x)}{\int_\R d x \rho( x)g(\tilde A, x)}
\end{equation*}
with
\begin{equation*}
\begin{aligned}
\rho( x)=\frac{e^{- x^2/4}\Gamma((\tilde C+\I x)/\psi)\Gamma((\tilde C-\I x)/\psi)}{\Gamma(2\I x/\psi)\Gamma(-2\I x/\psi)}, \quad 
g(\tilde A, x)=\Gamma((\tilde A+\I x)/\psi)\Gamma((\tilde A-\I x)/\psi).
\end{aligned}
\end{equation*}
We have that
\begin{equation*}
\frac{d \tilde F(\tilde A,\tilde C) }{d\tilde A} = \frac14\frac{\int d x  x^2 \rho( x) \frac{d g(\tilde A, x)}{d\tilde A} \int d\tilde x \rho(\tilde x) g(\tilde A,\tilde x)-\int d x  x^2 \rho( x) g(\tilde A, x) \int d\tilde x \rho(\tilde x) \frac{d g(\tilde A,\tilde x)}{d\tilde A}}{\left(\int d x \rho( x) g(\tilde A, x)\right)^2}
\end{equation*}
and
\begin{equation*}
\frac{d  g(\tilde A, x)}{d\tilde A}=g(\tilde A, x) h(\tilde A, x),
\end{equation*}
where $h(\tilde A, x)=c^{-1}[\Gamma(0,(\tilde A+\I x)/\psi)+\Gamma(0,(\tilde A-\I x)/\psi)]$. Here, $\Gamma(0,z)=\frac{\Gamma'(z)}{\Gamma(z)}$ is the \mbox{$0$-polygamma} function.
Denoting $f(x)=\rho(x) g(\tilde A,x)$, we get that
\begin{equation}\label{eq4.12}
\frac{d \tilde F(\tilde A,\tilde C) }{d\tilde A} = \frac14\frac{\int d x  x^2 f( x) h(\tilde A, x)\int d\tilde x f(\tilde x)-\int d x  x^2 f( x)  \int d\tilde x f(\tilde x) h(\tilde A,\tilde x)}{\left(\int d x f( x) \right)^2}.
\end{equation}
Note that by \eqref{eqGammaIdentity} and \eqref{eqGamma1}, we have that for all $x$ large enough (independent of $\tilde A$), $f(x)\leq e^{-x^2/8}$. Thus, the contribution to the integrals in \eqref{eq4.12} for $|x|,|\tilde x|\geq \tilde A^{1/4}$ is at least $e^{-\tilde A^{1/2}/16}$ smaller than the leading term.
For this reason, we can restrict the integrals in \eqref{eq4.12} to $|x|,|\tilde x|\leq \tilde A^{1/4}$. Next, we use the asymptotic expansion of the $0$-polygamma function for large $z$, which is given in \textbf{6.3.18} of \cite{DAR:Pocketbook} as
\begin{equation*}
\Gamma(0,z)=\ln(z)-\frac{1}{2z}-\frac1{12 z^2}+\Or(z^{-4}),\quad \text{ for } z\to\infty.
\end{equation*}
This leads to
\begin{equation*}
h(\tilde A, x) = \left(\left[\frac{2}{\psi} \ln(\tilde A/\psi)-\frac1{\tilde A}-\frac{\psi}{6 \tilde A^2}\right]+\frac{x^2}{\psi \tilde A^2}\right)(1+\Or(\tilde A^{-1/2}))
\end{equation*}
for $|x|\leq \tilde A^{1/4}$.
Inserting this expansion into \eqref{eq4.12}, the $x$-independent term cancels exactly.  The remaining error terms as well as the terms in $x^2$ yield
\begin{equation*}
\eqref{eq4.12} = \frac{1}{4 c\tilde A^2}\frac{\int d x x^4 f( x) \int d\tilde x f(\tilde x)-\int d x  x^2 f( x)  \int d\tilde x x^2 f(x)}{\left(\int d x f( x) \right)^2} (1+\Or(\tilde A^{-5/2})).
\end{equation*}
By the Cauchy-Schwarz argument as in the proof of Lemma~\ref{lemMonotonicity}, we get that the numerator is strictly positive for all $\tilde{A}$ large enough.
\end{proof}

\section{Moderate deviations for the current of the open ASEP}\label{sec:ModDevCurrent}

In this section, we establish moderate deviation estimates for the current $(\Jo^N_t)_{t \geq 0}$ of the open ASEP. Our goal is to compare the current of the ASEP on $\Z$ {to the current of the open ASEP on a segment of length $N$. Both ASEPs are started from Bernoulli-$\rho$-product measures (on $\{0,1\}^{\Z}$ and on $\{0,1\}^N$,  respectively), and we evaluate the currents at some time $T>0$. We consider two different regimes of the densities $\rho$ and times $T$: 
\begin{itemize}
    \item[(a)] For fixed $\mathfrak{a}\in (0,\frac{1}{2})$, let $\rho \in \left[\frac{1}{2}+N^{-1/2},1-\mathfrak{a}\right]$ and $T=(1-q)^{-1}(1-2\rho)^{-1}N$.
    \item[(b)] Let $\rho=\frac{1}{2}$ and $T=(1-q)^{-1}N^{3/2}$.
\end{itemize}
Under both assumptions, the fluctuations of a second class particle in the ASEP on $\Z$, initially placed at the origin in a Bernoulli-$\rho$-product measure, are at  time $T$ of order at most $N$.}
This suggests that under the basic coupling and a common initial distribution, the ASEP on $\Z$ should differ from the open ASEP until time $T$ in at most order $N$ many sites, and hence their currents at order at most $\sqrt{N}$. We formalize this intuition in Propositions~\ref{pro:CurrentZtoN} and~\ref{pro:CurrentNtoO}.

\subsection{Strategy for the proof}\label{sec:Strategy}

We start by comparing the ASEP on $\Z$ to an ASEP on $\N$ under the basic coupling. Suppose that both processes are started from the same configuration on $\N$ chosen according to a Bernoulli-$\rho$-product measure for some $\rho \in (0,1)$. Recall Definition~\ref{def:BasicCoupling} and note that under the basic coupling, second class particles of types $\Aup$ and $\Bup$ can only enter at site $1$ in the disagreement process. In order to follow the second class particles, we define in Section~\ref{sec:ExtendedBasicCoupling} an extended disagreement process, where instead of annihilation, an update of a type $\Aup / \Bup$ second class particle pair results in a change to types $\onep$ and $\zerop$, respectively. This construction has the advantage that type $\Bup$ and $\zerop$ second class particles can be seen as a coloring of empty sites in an ASEP on $\Z$ within a Bernoulli-$\rho$-product measure.

In Section \ref{sec:ModerateSecondClass}, we establish moderate deviations on the maximal displacement for a collection of second class particles in the ASEP on $\Z$. This is achieved by combining moderate deviations for a single second class particle and the censoring inequality. We get in Section~\ref{sec:ModerateSecondClassMulti} moderate deviations on the rightmost type $\Bup$ or type $\zerop$ second class particle in the extended disagreement process { until time $T \asymp (2\rho-1)^{-1}N(1-q)^{-1}$ when $\rho \geq \frac{1}{2}+N^{-1/2}$,} and $T \asymp N^{3/2}(1-q)^{-1}$ for $\rho=\frac{1}{2}$. Similarly, whenever the disagreement process contains only type $\Aup$ and $\onep$ second class particles beyond a certain site, we can see these type $\Aup$ and type $\onep$ second class particles as a coloring of first class particles in an ASEP on $\Z$ in a Bernoulli-$\rho$-product measure. With a similar argument as for the type $\Bup$ and $\zerop$ second class particles, we obtain moderate deviations for the location of the rightmost type $\Aup$ and $\onep$ second class particles, respectively.

In Section \ref{sec:IntergersToN}, we convert these moderate deviations on the location of the rightmost second class particle in the extended disagreement process into moderate deviations of the respective currents, expressing the difference in the currents as the difference in the number of type $\Aup$ and type $\Bup$ second class particles in the extended disagreement process at a given time.  Controlling the location of the rightmost second class particle in a disagreement process between the ASEP on $\N$ and the open ASEP in a similar way, this allows us to achieve moderate deviations for the current of the open ASEP in Section \ref{sec:NToOpen}.  Together with the results from Section \ref{sec:Prelim} and  \ref{sec:CurrentEstimates} on the stationary current of the open ASEP, we obtain in Section \ref{sec:CurrentBoundsSecondMoment} lower bounds on the probability that the currents of two open ASEPs with different boundary conditions deviate by the order of their stationary currents.

\subsection{The extended basic coupling}\label{sec:ExtendedBasicCoupling}

Recall the basic coupling defined in Section~\ref{sec:CanonicalCoupling} for the ASEPs on $\N$ and $\Z$.
In the following, we introduce two additional types of second class particles: type $\onep$ and type $\zerop$ to which we refer as \textbf{marked second class particles} for type $\Aup$ and type $\Bup$, respectively.
We consider the ordering $\succeq$ given by
\begin{equation}\label{def:ExtendedPartialOrder}
\one \succeq \onep \succeq \Aup \succeq \Bup \succeq \zerop \succeq \zero ,
\end{equation}
where we identify $\one$ with first class particles and $\zero$ with empty sites in the disagreement process.
Note that the ordering $\Aup \succeq \Bup$ is not consistent with the updates of the underlying exclusion process, as $\Aup$ and $\Bup$ are incomparable. We remedy this issue by an extension to the basic coupling, where every time a pair of unmarked second class particles of different types is updated, they receive a mark, that is, we turn a type $\Aup$ and $\Bup$ pair into $\onep$ and $\zerop$. Once a particle becomes marked it remains marked forever.


\begin{definition}[Extended disagreement process]\label{def:BasicCouplingExtended}
We define the extended disagreement process $(\xi^{\modi}_t)_{t \geq 0}$ as a Markov process on the state space
\begin{equation}
\{ \one , \onep , \Aup , \Bup , \zerop, \zero \}^{\Z} .
\end{equation}
For all edges $e=\{ x,x+1\}$, assign a rate $1+q$ clock. Whenever a clock rings at time $t$, we sample a Uniform-$[0, 1]$-random variable $U$ independently of all previous samples. For $x\neq 0$, we distinguish four cases:
\begin{itemize}
\item[(1)] If $U\leq (1+q)^{-1}$ and $\{ \xi^{\modi}_{t_-}(x), \xi^{\modi}_{t_-}(x+1) \}  \neq \{\Aup,\Bup\}$, then sort the endpoints in $\xi^{\modi}_t$ with respect to  \eqref{def:ExtendedPartialOrder} in increasing order,
\item[(2)] if $U> (1+q)^{-1}$ and $\{ \xi^{\modi}_{t_-}(x), \xi^{\modi}_{t_-}(x+1) \}  \neq \{\Aup,\Bup\}$, then sort the endpoints in $\xi^{\modi}_t$ with respect to  \eqref{def:ExtendedPartialOrder} in decreasing order,
\item[(3)] if $U\leq (1+q)^{-1}$ and $\{ \xi^{\modi}_{t_-}(x), \xi^{\modi}_{t_-}(x+1) \}  = \{\Aup,\Bup\}$, then set $\xi^{\modi}_t(x+1)=\onep$ and $\xi^{\modi}_t(x)=\zerop$,
\item[(4)] if $U> (1+q)^{-1}$ and $\{ \xi^{\modi}_{t_-}(x), \xi^{\modi}_{t_-}(x+1) \}  = \{\Aup,\Bup\}$, then set $\xi^{\modi}_t(x)=\onep$ and $\xi^{\modi}_t(x+1)=\zerop$.
\end{itemize}
The rules for updates at site $1$ and along the edge $\{0,1\}$ remain as under the basic coupling, treating $\onep$ as a particle and $\zerop$ as an empty site, i.e.,
\begin{itemize}
\item[(5)] if $U\leq (1+q)^{-1}$ and $\xi^{\modi}_{t_-}(0) = \Aup$ as well as $ \xi^{\modi}_{t_-}(1) = \zero$,  then set $\xi^{\modi}_t(1)= \Aup$  and $\xi^{\modi}_t(0)= \zero$,
\item[(6)] if $U\leq (1+q)^{-1}$ and $\xi^{\modi}_{t_-}(0) = \Aup$ as well as $ \xi^{\modi}_{t_-}(1) \in \{ \Bup ,\zerop\}$, then set $\xi^{\modi}_t(1)=\one$  and $\xi^{\modi}_t(0)= \zero$,
\item[(7)] if $U > (1+q)^{-1}$ and $\xi^{\modi}_{t_-}(0) = \zero$ as well as $ \xi^{\modi}_{t_-}(1) \in \{\onep,1\}$,  then set $\xi^{\modi}_t(1)= \Bup$  and $\xi^{\modi}_t(0)= \Aup$,
\item[(8)] if $U > (1+q)^{-1}$ and $\xi^{\modi}_{t_-}(0) = \zero$ as well as $ \xi^{\modi}_{t_-}(1) = \Aup$,  then set $\xi^{\modi}_t(1)=\zero$  and $\xi^{\modi}_t(0)= \Aup$.
\end{itemize}
Moreover, at rate $\alpha$, we update site $1$ as follows:
\begin{itemize}
\item[(9)] If $\xi^{\modi}_{t_-}(1) \in \{\zero,\zerop\}$, we set $\xi^{\modi}_t(1)=\Bup$,
\item[(10)] if $\xi^{\modi}_{t_-}(1) \in \{\Aup,\onep\}$, we set $\xi^{\modi}_t(1)=\one$.
\end{itemize}
At rate $\gamma$, we perform the following update:
\begin{itemize}
\item[(11)] If $\xi^{\modi}_{t_-}(1) \in \{\zerop,\Bup\}$, we set $\xi^{\modi}_t(1)=\zero$,
\item[(12)] if $\xi^{\modi}_{t_-}(1) \in \{\onep,\one\}$, we set $\xi^{\modi}_t(1)=\Aup$.
\end{itemize}
\end{definition}

\begin{figure}
\centering
\begin{tikzpicture}[scale=0.8]

\def\x{2};
\def\y{1};
\def\z{-0.2};
\def\a{-1.4};
\def\b{-2.6};

 	\node[shape=circle,scale=1.5,draw] (C1) at (0,\y){} ;
 	\node[shape=circle,scale=1.5,draw] (C2) at (\x,\y){} ;
 	\node[shape=circle,scale=1.5,draw] (C3) at (2*\x,\y){} ;
 	\node[shape=circle,scale=1.5,draw] (C4) at (3*\x,\y){} ;
 	\node[shape=circle,scale=1.5,draw] (C5) at (4*\x,\y){} ;
 	\node[shape=circle,scale=1.5,draw] (C6) at (5*\x,\y){} ;
 	\node[shape=circle,scale=1.5,draw] (C7) at (6*\x,\y){} ;
 	\node[shape=circle,scale=1.5,draw] (C8) at (7*\x,\y){} ;

	\node[shape=circle,scale=1.2,fill=red] (D1) at (C1) {};
	\node (F1) at (C2) {$\Aup$};
	\node (F1) at (C3) {$\zerop$};	
	\node[shape=circle,scale=1.2,fill=red] (D5) at (C5) {};
	\node (F1) at (C6) {$\Aup$};
	\node (F1) at (C7) {$\Bup$};		
	
 	\draw[thick,densely dotted] (-0.8,\y) -- (C1);
 	\draw[thick] (C1) -- (C2);
 	\draw[thick] (C2) -- (C3);
 	\draw[thick] (C3) -- (C4);
 	\draw[thick] (C4) -- (C5);
 	\draw[thick] (C5) -- (C6);
 	\draw[thick] (C6) -- (C7);
 	\draw[thick] (C7) -- (C8);
 	\draw[thick,densely dotted] (C8) -- (0.8+7*\x,\y);

 	\node (H1) at (-2,\y) {$t=0$};
 	\node (H2) at (-2,\z) {$t=1$};	
	\node (H3) at (-2,\a) {$t=2$};	
	\node (H3) at (-2,\b) {$t=3$};

 	\node[shape=circle,scale=1.5,draw] (C1) at (0,\z){} ;
 	\node[shape=circle,scale=1.5,draw] (C2) at (\x,\z){} ;
 	\node[shape=circle,scale=1.5,draw] (C3) at (2*\x,\z){} ;
 	\node[shape=circle,scale=1.5,draw] (C4) at (3*\x,\z){} ;
 	\node[shape=circle,scale=1.5,draw] (C5) at (4*\x,\z){} ;
 	\node[shape=circle,scale=1.5,draw] (C6) at (5*\x,\z){} ;
 	\node[shape=circle,scale=1.5,draw] (C7) at (6*\x,\z){} ;
 	\node[shape=circle,scale=1.5,draw] (C8) at (7*\x,\z){} ;

	\node[shape=circle,scale=1.2,fill=red] (D1) at (C2) {};
	\node (F1) at (C1) {$\Aup$};
	\node (F1) at (C3) {$\zerop$};	
	\node[shape=circle,scale=1.2,fill=red] (D5) at (C5) {};
	\node (F1) at (C7) {$\onep$};
	\node (F1) at (C6) {$\zerop$};		
	
 	\draw[thick,densely dotted] (-0.8,\z) -- (C1);
 	\draw[line width=2pt,color=blue] (C1) -- (C2);
 	\draw[thick] (C2) -- (C3);
 	\draw[thick] (C3) -- (C4);
 	\draw[thick] (C4) -- (C5);
 	\draw[thick] (C5) -- (C6);
 	\draw[line width=2pt,color=blue] (C6) -- (C7);
 	\draw[thick] (C7) -- (C8);
 	\draw[thick,densely dotted] (C8) -- (0.8+7*\x,\z);

 	\node[shape=circle,scale=1.5,draw] (C1) at (0,\a){} ;
 	\node[shape=circle,scale=1.5,draw] (C2) at (\x,\a){} ;
 	\node[shape=circle,scale=1.5,draw] (C3) at (2*\x,\a){} ;
 	\node[shape=circle,scale=1.5,draw] (C4) at (3*\x,\a){} ;
 	\node[shape=circle,scale=1.5,draw] (C5) at (4*\x,\a){} ;
 	\node[shape=circle,scale=1.5,draw] (C6) at (5*\x,\a){} ;
 	\node[shape=circle,scale=1.5,draw] (C7) at (6*\x,\a){} ;
 	\node[shape=circle,scale=1.5,draw] (C8) at (7*\x,\a){} ;

	\node[shape=circle,scale=1.2,fill=red] (D1) at (C3) {};
	\node (F1) at (C1) {$\Aup$};
	\node (F1) at (C2) {$\zerop$};	
	\node[shape=circle,scale=1.2,fill=red] (D5) at (C5) {};
	\node (F1) at (C8) {$\onep$};
	\node (F1) at (C6) {$\zerop$};			
	
 	\draw[thick,densely dotted] (-0.8,\a) -- (C1);
 	\draw[thick] (C1) -- (C2);
 	\draw[line width=2pt,color=blue] (C2) -- (C3);
 	\draw[thick] (C3) -- (C4);
 	\draw[thick] (C4) -- (C5);
 	\draw[thick] (C5) -- (C6);
 	\draw[thick] (C6) -- (C7);
 	\draw[line width=2pt,color=blue] (C7) -- (C8);
 	\draw[thick,densely dotted] (C8) -- (0.8+7*\x,\a);

 	\node[shape=circle,scale=1.5,draw] (C1) at (0,\b){} ;
 	\node[shape=circle,scale=1.5,draw] (C2) at (\x,\b){} ;
 	\node[shape=circle,scale=1.5,draw] (C3) at (2*\x,\b){} ;
 	\node[shape=circle,scale=1.5,draw] (C4) at (3*\x,\b){} ;
 	\node[shape=circle,scale=1.5,draw] (C5) at (4*\x,\b){} ;
 	\node[shape=circle,scale=1.5,draw] (C6) at (5*\x,\b){} ;
 	\node[shape=circle,scale=1.5,draw] (C7) at (6*\x,\b){} ;
 	\node[shape=circle,scale=1.5,draw] (C8) at (7*\x,\b){} ;

	\node[shape=circle,scale=1.2,fill=red] (D1) at (C3) {};
	\node (F1) at (C2) {$\Aup$};
	\node (F1) at (C1) {$\zerop$};	
	\node[shape=circle,scale=1.2,fill=red] (D5) at (C6) {};
	\node (F1) at (C8) {$\onep$};
	\node (F1) at (C5) {$\zerop$};			
	
 	\draw[thick,densely dotted] (-0.8,\b) -- (C1);
 	\draw[line width=2pt,color=blue] (C1) -- (C2);
 	\draw[thick] (C2) -- (C3);
 	\draw[thick] (C3) -- (C4);
 	\draw[thick] (C4) -- (C5);
 	\draw[line width=2pt,color=blue] (C5) -- (C6);
 	\draw[thick] (C6) -- (C7);
 	\draw[thick] (C7) -- (C8);
 	\draw[thick,densely dotted] (C8) -- (0.8+7*\x,\b);

	\end{tikzpicture}	
\caption{\label{fig:Extended}Visualization of a possible evolution of the extended disagreement process for $q=0$. Edges which received an update since the previous step are marked in blue. First class particles are red, second class particles are given by their types. Note that the type of second class particles changes if and only if a type $\Aup/\Bup$ pair receives an update (as from $t=0$ to $t=1$). }
 \end{figure}
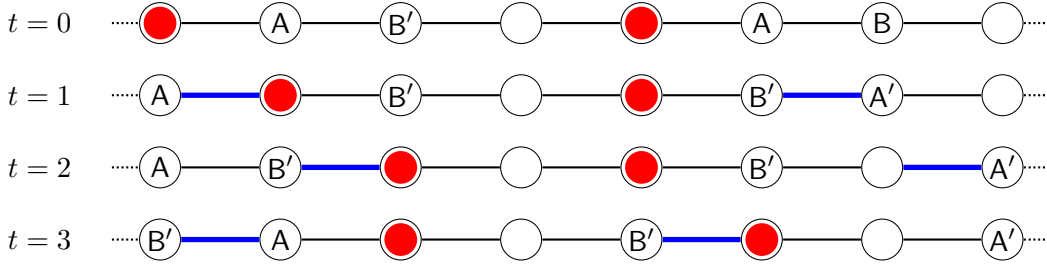

A visualization of this process is given in Figure \ref{fig:Extended}.
 In words, the process $(\xi^{\modi}_t)_{t \geq 0}$ has the same transition rules as a disagreement process under the basic coupling (identifying $\Aup \widehat{=} (1,0)$ and $\Bup\widehat{=}(0,1)$ as well as $\one\widehat{=}(1,1)$ and $\zero\widehat{=}(0,0)$), but every time a pair of unmarked second class particles of different types receives an update, we mark them.
Let $(\eta^{\Aup}_t)_{t \geq 0}$ onto $\{0,1\}^{\Z}$ and $(\eta^{\Bup}_t)_{t \geq 0}$ onto $\{0,1\}^{\N}$ be two projections of the extended disagreement process $(\xi^{\modi}_t)_{t \geq 0}$ given by
\begin{equation}\label{def:ProjectionEta}
 \begin{split}
  \eta^{\Aup}_t(x) &:= \mathds{1}_{ \xi^{\modi}_t(x) \in \{ \Aup,\onep,\one \} }   \\
 \eta^{\Bup}_t(y) &:= \mathds{1}_{ \xi^{\modi}_t(y) \in \{ \Bup, \onep,\one \} }    
 \end{split}
\end{equation} for all $x\in\Z$ and $y \in \N$. Moreover, we define the process $(\xi^{\Aup\Bup}_t)_{t \geq 0}$ as the projection of $(\xi^{\modi}_t)_{t \geq 0}$ onto $\{ (1,1),(1,0),(0,1),(0,0)\}^{\Z}$, which is given for all $x\in \Z$ by
\begin{equation}
\xi^{\Aup\Bup}_t(x) := \begin{cases} (0,1) & \text{ if } \xi^{\modi}_t(x)= \Aup \\
(1,0) & \text{ if } \xi^{\modi}_t(x)= \Bup \\
(1,1) & \text{ if } \xi^{\modi}_t(x) \in \{ \onep,\one \} \\
(0,0) & \text{ if } \xi^{\modi}_t(x) \in \{ \zerop,\zero \} .
\end{cases}
\end{equation}
The next lemma states that the above three projections have a simple interpretation as exclusion and disagreement processes.
\begin{lemma}\label{lem:SecondClassProjection}
The process $(\eta^{\Aup}_t)_{t \geq 0}$ has the law of an ASEP on $\Z$ with respect to drift $q$. The process $(\eta^{\Bup}_t)_{t \geq 0}$ has the law of an ASEP on $\N$ with respect to drift $q$ and entry and exit rates $\alpha$ and $\gamma$, respectively. The process $(\xi^{\Aup\Bup}_t)_{t \geq 0}$ has the law of a disagreement process between $(\eta^{\Aup}_t)_{t \geq 0}$ and $(\eta^{\Bup}_t)_{t \geq 0}$, with the convention that sites on $(-\infty,0]$ in the ASEP on $\Z$ are always either $\Aup$ or $0$ (since $(\eta^{\Bup}_t)_{t \geq 0}$ is only defined on $\{0,1\}^{\N}$.)
\end{lemma}
\begin{proof}
This is follows from verifying the transition rates in the construction of the extended disagreement process. {More precisely, for the process $(\eta^{\Aup}_t)_{t \geq 0}$, note that the rules $(1)$ to $(4)$ in Definition~\ref{def:BasicCouplingExtended} ensure that along all edges $\{x,x+1\}$ for $x \neq 0$, we sort the endpoints in increasing order at rate $1$, and in decreasing order at rate $q$, respectively. Here, we stress that the change from $\Aup,\Bup$ to $\onep,\zerop$ in rules $(3)$ and $(4)$ does not change the projection  $(\eta^{\Aup}_t)_{t \geq 0}$. Similarly, rules $(5)$ and $(6)$ ensure that we sort the endpoints along the edge $\{0,1\}$ in increasing order at rate $1$, while rules $(7)$ and $(8)$ ensure that we sort the endpoints along the edge $\{0,1\}$ in decreasing order at rate $q$. Rules $(9)$ to $(12)$ do not affect the projection $(\eta^{\Aup}_t)_{t \geq 0}$. The processes $(\eta^{\Bup}_t)_{t \geq 0}$ and $(\xi^{\Aup\Bup}_t)_{t \geq 0}$ are treated similarly.}
\end{proof}
Furthermore, we define the two projections $(\xi^{\Aup}_t)_{t \geq 0}$ and $(\xi^{\Bup}_t)_{t \geq 0}$ as follows. We obtain $(\xi^{\Aup}_t)_{t \geq 0}$ by setting
 \begin{equation}\label{def:ASEPaux1}
 \xi^{\Aup}_t(x) :=
\begin{cases} \one &\text{ if }
\xi^{\modi}_t(x) =  \one  \\
\Aup &\text{ if } \xi^{\modi}_t(x) \in \{   \onep , \Aup \} \\
\zero &\text{ if } \xi^{\modi}_t(x) \in \{ \zero , \zerop , \Bup \}
\end{cases}
\end{equation}
for all $t \geq 0$ and $x\in \Z$. Intuitively, $(\xi^{\Aup}_t)_{t \geq 0}$ acts on sites $\geq 2$ like a multi-species ASEP, but where some of the first class particles in $(\eta_t^{\Aup})_{t \geq 0}$ are turned into second class particles. Similarly, let

\begin{equation}\label{def:ASEPaux2}
\xi^{\Bup}_t(x) :=
\begin{cases} \one &\text{ if }
\xi^{\modi}_t(x) \in \{  \one , \onep , \Aup \} \\
\Bup &\text{ if } \xi^{\modi}_t(x) \in \{   \zerop , \Bup \} \\
\zero &\text{ if } \xi^{\modi}_t(x) = \zero
\end{cases}
\end{equation} for all $t \geq 0$ and $x\in \Z$, and note that $(\xi^{\Bup}_t)_{t \geq 0}$ has the same law as $(\eta_t^{\Aup})_{t \geq 0}$, but where some of the empty sites in $(\eta_t^{\Aup})_{t \geq 0}$ are turned into second class particles. Let us stress that in all of the above processes, second class particles can enter and exit only at site $1$.

\begin{remark}\label{rem:ReverseOrder}
Note that instead of \eqref{def:ExtendedPartialOrder}, we can also define the ordering $\succeq_{\rev}$ given by
\begin{equation}\label{def:ExtendedPartialOrderReversed}
\one \succeq_{\rev} \onep \succeq_{\rev} \Bup \succeq_{\rev} \Aup \succeq_{\rev} \zerop \succeq_{\rev} \zero .
\end{equation} where we reverse the priority of the second class particles of types $\onep,\Aup,\Bup$ and $\zerop$. The extended disagreement process is defined accordingly, that is, we have the same transition rules as in a disagreement process between an ASEP on $\Z$ and an ASEP on $\N$. However, every adjacent pair of $\Aup$ and $\Bup$ is replaced by $(\onep,\zerop)$ at rate $1$ and $(\zerop,\onep)$ at rate $q$.
\end{remark}

As we will see in Sections~\ref{sec:ModerateSecondClassMulti} to \ref{sec:NToOpen}, this interpretation of the extended disagreement process as a multi-type exclusion process (with a marking procedure and special rules along the edge $\{0,1\}$ and at site $1$) allows us to study the position of type $\Aup$ and $\Bup$ second class particles in the vein of \cite{BS:OrderCurrent} for the fluctuations of a second class particle in the ASEP on $\Z$.

\subsection{Moderate deviations for second class particles}\label{sec:ModerateSecondClass}

In this section, we establish moderate deviation results for second class particles in the ASEP on $\Z$. Recall the multi-species ASEP $(\zeta_t)_{t \geq 0}$ defined in Section \ref{sec:ASEPs}. We will restrict ourselves to three particle types, $1,2$ and $\infty$, to which we refer as first class particles, second class particles, and empty sites, respectively. We start by considering  $(\zeta_t)_{t \geq 0}$ with initial configuration $\zeta$ such that $\zeta$ contains a single second class particle at the origin and is chosen according to a Bernoulli-$\rho$-product measure for some $\rho \in (0,1)$ on all other sites. We denote by $(Z_t)_{t \geq 0}$ the location of the unique second class particle in $(\zeta_t)_{t \geq 0}$.
We define the following interval of size $2y$ around the expected position of a second class particle at time $t$ starting from $z$, namely
\begin{equation}
\mathcal{I}_{q,\rho}(t,y,z) := \left[ z + (1-q)(1-2\rho)t - y -1, z + (1-q)(1-2\rho)t + y + 1\right],
\end{equation}
where $(1-q)(1-2\rho)$ is the stationary speed of a second class particle, i.e.,
\begin{equation*}
\E[Z_t] = t (1-q)(1-2\rho) 
\end{equation*}
for all $t \geq 0$; see Theorem 2.28 in Part III of \cite{L:Book2}. The following  moderate deviation result for $(Z_t)_{t \geq 0}$ is due to Landon and Sosoe \cite{LS:Tails}.

\begin{theorem}[c.f.~Theorem~2.5 in \cite{LS:Tails}]\label{thm:ModDeviationSingleSecondClass}
Suppose that there exists some $\mathfrak{a}\in (0,\frac{1}{2})$ such that $\rho \in [\mathfrak{a},1-\mathfrak{a}]$.
Then there exist constants $c_0,C_0$, depending only on $\mathfrak{a}$, such that for all $q \in (0,1)$, and all $1 \leq w \leq (1-q)s^{1/3}$
\begin{equation}
\P\big( Z_s \notin \mathcal{I}_{q,\rho}(s,w s^{2/3},0) \big) \leq C_0 \exp(-c_0 w^3) .
\end{equation}
\end{theorem}
\begin{remark}\label{rem:ParametersSecondClass}
As in Remark \ref{rem:ParametersCurrent}, let us stress that the dependence of the constants $c_0,C_0$ on $\mathfrak{a}$ can be found as Proposition 5.3 of \cite{LS:Tails} for the stochastic six vertex model, which transfers to the ASEP using the results by Aggarwal from \cite{A:Convergence}; see also Section 7 in \cite{LS:Tails}.
\end{remark}

Next, we prove a moderate deviation result for the maximal displacement of a second class particle in the ASEP on $\Z$. This will then be further extended to a bound on the maximal displacement of multiple second class particles. 
{To this end, we consider a parameter $N$ and choose the time $T$, depending also on $\rho$ and $q$, in such a way that a second class particle in the ASEP on $\Z$ from a Bernoulli-$\rho$-product measure will have fluctuations of order $N$ at time $T$.}
The next lemma is similar to Theorem 4.7 in  \cite{SY:OpenLight}, which states moderate deviations for the maximal displacement of a single second class particles when $w \geq \log(s)$ above. However, we require refined bounds, which we obtain by a multi-scale argument similar to the chaining argument of Theorem~11.1 in \cite{BSV:SlowBond} for the transversal fluctuations of geodesics in last passage percolation.

{
\begin{lemma}\label{lem:ModDeviationMaxSecondClass} Fix some $\mathfrak{a}\in (0,\frac{1}{2})$. Let $N \in \N$ and $\rho \in [\frac{1}{2}+N^{-1/2},1-\mathfrak{a}]$, and set
\begin{equation}
T=\theta^{-1} (2\rho-1)^{-1}(1-q)^{-1} N
\end{equation}
for some $\theta \geq 1$, allowed to depend on $N$. Let $(Z_t)_{t \geq 0}$ denote the position of a second class particle started from the origin in an ASEP $(\zeta_{t})_{t \geq 0}$ on $ \Z$ in a Bernoulli-$\rho$-product measure. Then there exist  constants $c_0,C_0,c_1>0$, such that for all $q$ from \eqref{def:TriplePointScaling} with $\kappa \in [0,\frac{1}{2}]$ and some $\psi>0$, and all $y$ with
\begin{equation}\label{eq:yAssumptions}
1 \leq y \leq c_1\theta^{-1/3}N^{(1-2\kappa)/3} \log^{-1}(N)(2\rho-1)^{-1/3} ,
\end{equation}
we get that for all $N$ large enough
\begin{equation}\label{eq:MaxDisplacement}
\P\big( \exists s \in [0,T] \colon Z_s \notin \mathcal{I}_{q,\rho}(s,y(\theta^{-1} N(2\rho-1)^{-1})^{2/3},0) \big) \leq C_0 \exp(-c_0 y^{3} ) .
\end{equation} Moreover, when $\rho=\frac{1}{2}$, we obtain with  $T= \theta^{-1} N^{3/2}(1-q)^{-1}$ that
\begin{equation}\label{eq:MaxDisplacementHalf}
\P\big( \exists s \in [0,T] \colon Z_s \notin \mathcal{I}_{q,\rho}(s,y\theta^{-2/3} N,0) \big) \leq C_0 \exp(-c_0 y^{3} )
\end{equation}
for all $1 \leq y \leq c_1\theta^{-1/3}N^{1/2-2\kappa/3}\log^{-1}(N)$, and all $N$ large enough.
\end{lemma}}

\begin{proof} To simplify notation, we will in the following write
\begin{equation}
\mathcal{I}(y,k):=\mathcal{I}_{q,\rho}(k2^{-i}, y(\theta^{-1} N(2\rho-1)^{-1}))^{2/3} ,0) ,
\end{equation} and consider only the case \eqref{eq:MaxDisplacement} as the same arguments as for $\rho =  \frac{1}{2}+\frac{1}{\sqrt{N}}$ will give~\eqref{eq:MaxDisplacementHalf}. We define $(h_i)_{i\in \N}$ as
\begin{equation*}
h_i := \frac{1}{2}\Big(\prod_{m=1}^{\infty}(1+2^{-m/4})\Big)^{-1} \prod_{j=1}^{i}(1+2^{-j/4}) ,
\end{equation*} where we note that $h_{1}>0$ as well as $h_i<\frac{1}{2}$ for all $i\in \N$.
For all $i\in \N$, we define the events
\begin{equation*}
\mathcal{A}_{i} := \left\{ Z_{k2^{-i}T} \in \mathcal{I}(h_i y,k) \, \forall k \in  \lbr 2^{i} \rbr \right\} .
\end{equation*} We will in the following argue that there exist some constants $c_2,C_2>0$ such that for all $i = \mathcal{O}(\log_2 \log(T))$, we get that
\begin{equation}\label{eq:RecursionMaxDisplacement}
\P ( \mathcal{A}_{i+1} \, | \, \mathcal{A}_i ) \geq  1- C_2\exp(-c_2 y^{3} 2^{i/4}) .
\end{equation}  Since we get from Theorem~\ref{thm:ModDeviationSingleSecondClass} that for some  constants $c_3,C_3>0$,
\begin{equation*}
\P (  \mathcal{A}_1 ) \geq 1- C_3 \exp(-c_3 y^{3}) ,
\end{equation*} this yields that by choosing $c_1>0$ in \eqref{eq:yAssumptions} small enough,
\begin{equation}\label{eq:AllMaxDisplacement}
\P \Bigg( {\bigcap_{i \leq \log_2 \log(T)}\mathcal{A}_{i}}  \Bigg) \geq 1 - C_4\exp(-c_4 y^{3})
\end{equation} for some $c_4,C_4>0$ and all $y$ from \eqref{eq:yAssumptions}. In order to show \eqref{eq:RecursionMaxDisplacement}, we define the sets
\begin{align*}
\mathcal{I}^{+}_{k,i} &:=  \mathcal{I}_{q,\rho}\big(k2^{-i}T,h_i y (\theta^{-1} N(2\rho-1)^{-1})^{\frac{2}{3}},\tfrac{1}{4}y^3 2^{\frac{i}{5}}\big) \setminus  \mathcal{I}_{q,\rho}(k2^{-i}T,h_i y (\theta^{-1} N(2\rho-1)^{-1})^{\frac{2}{3}},0) \ \\
\mathcal{I}^{-}_{k,i} &:= \mathcal{I}_{q,\rho}\big(k2^{-i}T,h_i y (\theta^{-1} N(2\rho-1)^{-1})^{\frac{2}{3}},-\tfrac{1}{4}y^3 2^{\frac{i}{5}}\big) \setminus  \mathcal{I}_{q,\rho}(k2^{-i}T,h_i y (\theta^{-1} N(2\rho-1)^{-1})^{\frac{2}{3}},0) ,
\end{align*}
and consider the events
\begin{align*}
\mathcal{B}^{k,+}_i &:= \big\{ \exists v^1_{+},v^{0}_{+} \in \mathcal{I}^{+}_{k,i} , \colon \,  \eta_{k2^{-i}T}(v^1_+)=1 , \ \eta_{k2^{-i}T}(v^{0}_+)=0   \big\} , \\
\mathcal{B}^{k,-}_i &:= \big\{ \exists v^1_{-},v^{0}_{-} \in \mathcal{I}^{-}_{k,i} \, \colon \,  \eta_{k2^{-i}T}(v^{1}_-)=1 , \  \eta_{k2^{-i}T}(v^{0}_-)=0   \big\} .
\end{align*}
We set in the following
\begin{equation*}
\mathcal{B}_{i} := \bigcap_{k=1}^{2^{i}-1} \big( \mathcal{B}^{k,+}_i \cap \mathcal{B}^{k,-}_i \big)
\end{equation*}
and note that for some constants $c_5,C_5>0$,
\begin{equation}\label{eq:Key1A}
\P( \mathcal{B}_i) \geq 1 - C_5 \exp(-c_5  y^3 2^{i/5} ) .
\end{equation}
This follows from the observation that the ASEP at time $k2^{-i}T$ has a Bernoulli-$\rho$-product law on $\mathcal{I}^{+}_{k,i}$ and $\mathcal{I}^{-}_{k,i}$, together with a { standard Chernoff estimate for the events $(\mathcal{B}^{k,+}_i)_{k \in \lbr 2^{i}-1 \rbr}$ and $(\mathcal{B}^{k,+}_i)_{k \in \lbr 2^{i}-1 \rbr}$, and a union bound.} Fix $i \in \N$ and $k \in  \lbr 2^{i}-1 \rbr$. Then on the event $\mathcal{B}^{k,-}_i \cap \mathcal{B}^{k,+}_i \cap \mathcal{A}_i$, we consider four ASEPs on the integers, started from $\eta_{k2^{-i}T}$, where we place second class particles at sites $v^{1}_+,v^{0}_+,v^{1}_-,v^{0}_-$, respectively, and where we replace $\eta_{k2^{-i}T}(Z_{k2^{-i}T})$ by an independent Bernoulli-$\rho$-distributed random variable (the same for all four processes).
Let $Z_{t}^{1,+},Z_{t}^{0,+},Z_{t}^{1,-},Z_{t}^{0,-}$ denote the positions of the second class particles at time $t$ in the respective exclusion processes. Observe that under the basic coupling $\mathbf{P}$ for all five processes, we get that
\begin{equation}\label{eq:Key1B}
 \mathbf{P}\Big(\min(Z_{t}^{1,-},Z_{t}^{0,-}) \leq Z_{k2^{-i}T+t} \leq \max(Z_{t}^{1,+},Z_{t}^{0,+}) \, \Big| \, \mathcal{B}^{k,-}_i \cap \mathcal{B}^{k,+}_i \cap \mathcal{A}_{i} \Big) = 1  .
\end{equation}
To see this, distinguish whether the second class particle in $\eta_{k2^{-i}T}$ is replaced by a first class particle or an empty site. In the first case, observe that $(Z_{t+k2^{-i}T})_{t\geq 0}$ is dominated by $(Z_{t}^{1,+})_{t \geq 0}$ and $(Z_{t}^{1,-})_{t \geq 0}$, in the second case it is dominated by  $(Z_{t}^{0,+})_{t \geq 0}$ and $(Z_{t}^{0,-})_{t \geq 0}$, respectively. For $i$ and $k \in \lbr 2^{i}-1 \rbr$, we define the events $\mathcal{C}_{i}^{k}$ as
\begin{equation}\label{eq:AllSecondMix}
\mathcal{C}_{i}^{k} := \left\{ Z_{2^{-i}T}^{1,+},Z_{2^{-i}T}^{0,+},Z_{2^{-i}T}^{1,-},Z_{2^{-i}T}^{0,-} \in \mathcal{I}(h_{i+1}y,k) \right\} .
\end{equation}
We only show that for some constants $c_6,C_6>0$
\begin{equation}\label{eq:OneSecondClassOut}
\P\Big( Z_{2^{-i}T}^{1,+} \in \mathcal{I}(h_{i+1}y,k)  \,  \Big| \,  \mathcal{A}_{i} {\cap} \mathcal{B}_i \Big) \geq  1- C_6\exp( -c_6 y^{3}2^{i/5}) .
\end{equation} A similar argument applies by symmetry for the other three processes in the events $\mathcal{C}_{i}^{k}$. Note that the law of $\eta_{k2^{-i}T}$ around the particle $Z_{0}^{1,+}$ differs by construction from a Bernoulli-$\rho$-product measure in at most $\frac{1}{4}y^3 2^{i/5}$ many sites (corresponding to the sites inspected to the right of $Z^{1,+}_{0}$). Since $\rho \in [\frac{1}{2}+N^{-1/2},1-\mathfrak{a}]$, by a change of measure to an ASEP on $\Z$ in a Bernoulli-$\rho$-product measure with a single second class particle initially at a fixed site, we see that for some constants $C_7,c_8,C_8>0$, depending only on $\mathfrak{a}$,  and all $y$ from \eqref{eq:yAssumptions}
\begin{align*}
&\sup_{x\in \mathcal{I}_{k,i}^{+}} \P\Big( Z_{k2^{-i}T}^{1,+} \notin \mathcal{I}(h_{i+1}y,k)  \,  \Big| \, Z_{0}^{1,+}=x\Big) \\
&\leq \exp(C_7 y^{3} 2^{i/5}) \sup_{x\in \mathcal{I}_{k,i}^{+}} \P\Big( Z_{k2^{-i}T} \notin \mathcal{I}(h_{i+1}y,k)  \,  \Big| \, Z_{0}=x\Big) \\
 &\leq C_8\exp(-c_8 y^{3}2^{i/4}) .
\end{align*}
Here, we applied Theorem \ref{thm:ModDeviationSingleSecondClass} with $s=2^{-i}T$
and $w=y2^{i/12}$ for the last step, using that $(h_{i+1}-h_i)/h_i = 2^{-i/4}$ for all $i\in \N$. This gives \eqref{eq:OneSecondClassOut} for all four processes in $\mathcal{C}_{i}^{k}$, which by using \eqref{eq:Key1A} and \eqref{eq:Key1B} implies \eqref{eq:RecursionMaxDisplacement}. Next, we define the event
\begin{equation*}
\tilde{\mathcal{A}} := \left\{ Z_s \notin \mathcal{I}_{q,\rho}(s,y(\theta^{-1} N2^{n})^{2/3},0) \big) \, \forall s \in [0,T]\right\} .
\end{equation*}
Condition on the events $\mathcal{A}_i$ for all $i < j= \lceil \log_2(\log(T)) \rceil$ and $\mathcal{B}_{j}$. Then there exist some constants $C_9,c_{10},C_{10},c_{11},C_{11}>0$, such that by a change of measure and a union bound for the first inequality, and Theorem~4.7 of \cite{SY:OpenLight} for the second inequality (which states Lemma~\ref{lem:ModDeviationMaxSecondClass} for $y \geq \log(N)$),
\begin{equation*}\label{eq:RecursionLastStep}
\begin{split}
\P \big( \tilde{\mathcal{A}} \, | \,  \cap_{i=1}^{j-1} (\mathcal{A}_j \cap \mathcal{B}_{j+1})  \big) &\geq 1- 2^{j-1} \exp(C_9 y^{3} 2^{j/5})  \P\big( \exists s\in [0,2^{-j}T] \colon  Z_{s} \notin \mathcal{I}(y/2,0)  \,  \big| \, Z_{0}=0\big) \\
& \geq 1- C_{10} \exp(C_{9} y^{3} 2^{j/5}) 2^{j}\exp(-c_{10} y^3 2^{3j}) \\ &\geq 1-C_{11}\exp(-c_{11} y^3) ,
\end{split}
\end{equation*} for all $y$ from \eqref{eq:yAssumptions} and $N$ large enough.
Together with \eqref{eq:AllMaxDisplacement}, this finishes the proof.
\end{proof}

Next, we recall the microscopic concavity coupling of two ASEPs with second class particles, introduced by Balázs and Seppäläinen in \cite{BS:FluctuationBounds}, which allows us to compare the location of second class particles for different initial conditions.

\begin{theorem}[c.f.\ Theorem 3.1 in \cite{BS:FluctuationBounds}]\label{thm:BSCoupling}
 Let $(\zeta^1_t)_{t \geq 0}$ and $(\zeta^{2}_t)_{t \geq 0}$ be two ASEPs on $\Z$ with initial configurations $\zeta^1_0=\zeta^1$ and $\zeta^2_0=\zeta^{2}$ such that
  \begin{equation}
      \zeta^1(x) \geq \zeta^{2}(x)
  \end{equation} with $\zeta^1(x), \zeta^{2}(x) \in \{0,1\}$ for all $x \neq 0$, and $ \zeta^1(0) = \zeta^{2}(0)=2$.
 Let $(Z_t^{1})_{t \geq 0}$ and $(Z_t^{2})_{t \geq 0}$ be the positions of the respective second class particles. There exists a coupling $\bar{\mathbf{P}}$ such that
\begin{equation}
     \bar{\mathbf{P}} \left(  Z^{2}_t \geq Z^{1}_t  \, \forall  t\geq 0 \right) =1 .
\end{equation}
\end{theorem}

{Let us stress that the coupling $\bar{\mathbf{P}}$ is \emph{not} the basic coupling from Definition~\ref{def:BasicCoupling} or a simple variant of the extend basic coupling from Definition~\ref{def:BasicCouplingExtended}. However, the exact form of this coupling $\bar{\mathbf{P}}$ is not relevant for the following arguments.}

Our goal is to convert moderate deviations on the maximal displacement of a single second class particle to the maximal displacement of order $N^{\kappa}$ many second class particles for some $\kappa \in [0,\frac{1}{2}]$. To achieve this, the second class particles have to be initially placed sufficiently distant from each other. To this end, let us set up some notation. Fix $\mathfrak{a} \in (0,\frac{1}{2})$ and recall $q$ from \eqref{def:TriplePointScaling} with some $\kappa \in [0, \frac{1}{2}]$.  In the following, we set
\begin{equation}\label{def:KappaEpsilon}
\expon=\frac{1}{4}(1-2\kappa) \geq 0  \quad \text{and} \quad \expont=\frac{\expon}{3} \geq 0 .
\end{equation} Note that $\kappa+\expon \leq \frac{1}{2}$. { Let $\rho \in [\frac{1}{2}+N^{-1/2},1-\mathfrak{a}]$. } We define a multi-species ASEP $(\tilde{\zeta}_t)_{t \geq 0}$ on $\Z$, where the initial condition is chosen according to the product measure { $\tilde{\pi}_{N,\rho}$ on $\{1,2,\infty\}^{\Z}$ so that 
\begin{equation}\label{eq:InitialSecondClass1}
\begin{split}
    \tilde{\pi}_{N,\rho}(\eta(x)=1) &= \rho - N^{-1/2}, \\
    \tilde{\pi}_{N,\rho}(\eta(x)=2) &= N^{-1/2}, \\
    \tilde{\pi}_{N,\rho}(\eta(x)=\infty) &=1-\rho ,
\end{split}
\end{equation}
 for all $x\in \lbr N^{\kappa+\expon+1/2} \rbr$, and 
\begin{equation}\label{eq:InitialSecondClass2}
\tilde{\pi}_{N,\rho}( \eta(x)=1) = 1- \tilde{\pi}_{N,\rho}( \eta(x)=\infty) = \rho 
\end{equation}
for all $x \notin \lbr N^{\kappa+\expon+1/2} \rbr$. Moreover, we set $\tilde{\pi}_{N,\frac{1}{2}} :=  \tilde{\pi}_{N,\frac{1}{2}+N^{-1/2}}$.
}
Let $(L_t)_{t \ge 0}$ and $(R_t)_{t \ge 0}$ denote the position of the left-most and right-most second class particle in the multi-species ASEP $(\tilde{\zeta}_t)_{t \geq 0}$ started from $\tilde{\pi}_{N,\rho}$, respectively. We have the following result on the location of second class particles in $(\tilde{\zeta}_t)_{t \geq 0}$ when $\kappa<\frac{1}{2}$.

\begin{lemma}\label{lem:MaxSecondClassOrder}
{Assume that $\kappa<\frac{1}{2}$ and fix some $\mathfrak{a}\in (0,\frac{1}{2})$. Let $N \in \N$ and $\rho \in [\frac{1}{2}+N^{-1/2},1-\mathfrak{a}]$, and set
\begin{equation}
T=\theta^{-1} (2\rho-1)^{-1}(1-q)^{-1} N
\end{equation}
for some $\theta \geq 1$, allowed to depend on $N$. Then there exist $c_0,C_0>0$, depending only on $\mathfrak{a}$,  such that for all $1 \leq y,\theta \leq N^{\expont}$ with the constant $\expont$ from \eqref{def:KappaEpsilon}
\begin{equation}\label{eq:MaxDisplacementOrder}
\P_{\tilde{\pi}_{N,\rho}}\Big(  L_s,R_s \in \mathcal{I}_{q,\rho}(s,y(\theta^{-1}(2\rho-1)^{-1}N)^{2/3}+2\theta^{-1}N,0) \, \forall  s\in [0,T]\Big) \geq 1- C_0 \exp(-c_0 y^{3} )
\end{equation} for all $N$ sufficiently large. Similarly, for $\rho=\frac{1}{2}$ and $T= \theta^{-1}N^{3/2}(1-q)^{-1}$, we get that }
\begin{equation}\label{eq:MaxDisplacementOrderHalf} {
\P_{\tilde{\pi}_{N,\rho}}\Big(  L_s,R_s \in \mathcal{I}_{q,\rho}(s,y\theta^{-2/3}N+2\theta^{-1}N,0) \, \forall s\in [0,T]\Big) \geq 1- C_0 \exp(-c_0 y^{3} ) . }
\end{equation}
\end{lemma}

Let us stress that $\expont$ is not optimal, but sufficient for our purposes. 
Intuitively, Lemma~\ref{lem:MaxSecondClassOrder} ensures that all second class particles in $(\tilde{\zeta}_t)_{t \geq 0}$ started according to $\tilde{\pi}_{N,\rho}$ satisfy similar moderate deviation bounds as a single second class particle.
Before giving the proof of Lemma~\ref{lem:MaxSecondClassOrder}, we provide a corresponding  statement when $\kappa=\frac{1}{2}$. We let the initial distribution $\hat{\pi}_{N,\rho}$ of $(\tilde{\zeta}_t)_{t \geq 0}$ be defined similar to the measure $\tilde{\pi}_{N,\rho}$ in \eqref{eq:InitialSecondClass1} and \eqref{eq:InitialSecondClass2}, i.e.,
\begin{equation}\label{eq:InitialSecondClass1Hat}
\begin{split}
    \hat{\pi}_{N,\rho}(\eta(x)=1) &= \rho - a^{-2}N^{-1/2}\log(N), \\
    \hat{\pi}_{N,\rho}(\eta(x)=2) &= a^{-2} N^{-1/2}\log(N), \\
    \tilde{\pi}_{N,\rho}(\eta(x)=\infty) &=1-\rho ,
\end{split}
\end{equation}
 for all $x\in \lbr a N\rbr$, and 
\begin{equation}\label{eq:InitialSecondClass2Hat}
\tilde{\pi}_{N,\rho}( \eta(x)=1) = 1- \tilde{\pi}_{N,\rho}( \eta(x)=\infty) = \rho 
\end{equation}
for all $x \notin \lbr a N \rbr$. 
Here, $a > 0$ is a sufficiently small constant chosen later on at the end of the proof of Lemma~\ref{lem:MaxSecondClassOrderCritical}.
Again, for $\rho=\frac{1}{2}$, we have the convention that
\begin{equation*}
    \hat{\pi}_{N,\frac{1}{2}} := \hat{\pi}_{N,\frac{1}{2}+N^{-1/2}} . 
\end{equation*} 
 The following lemma is the analogue of Lemma~\ref{lem:MaxSecondClassOrder} when $\kappa=\frac{1}{2}$.

\begin{lemma}\label{lem:MaxSecondClassOrderCritical}
Let $\kappa=\frac{1}{2}$. Let $N \in \N$ and assume that
\begin{equation}\label{eq:FirstSetrhoCrit}
    \rho \in \left[\frac{1}{2}+N^{-1/2},\frac{1}{2}+c_1\log^{-1}(N)\right] , 
\end{equation}
where we recall the constant $c_1>0$ from Lemma~\ref{lem:ModDeviationMaxSecondClass}. Set
\begin{equation}
T=\theta^{-1} (2\rho-1)^{-1} N^{3/2} \log(N)^{-1}
\end{equation}
for $\theta>0$, and consider an ASEP on $\Z$ with initial distribution $\hat{\pi}_{N,\rho}$. Write $(L_t)_{t\geq 0}$ and $(R_t)_{t\geq 0}$ for the location of the left-most and right-most second class particle, respectively. Then for every $\phi>0$, we find some $a>0$ in the definition of $\hat{\pi}_{N,\rho}$ and some $\theta>0$, both depending only on $\phi$ and $c_1$, such that
\begin{equation}\label{eq:MaxDisplacementOrderHalfCritical}
\P\big( L_s,R_s \in [-\phi N , 2 \phi N] \, \forall s\in [0,T] \big) \geq 1- N^{-20}
\end{equation}  for all $N$ large enough. 
Similarly, when $\rho=\frac{1}{2}$, we find some  $c_2>0$ such that
\eqref{eq:MaxDisplacementOrderHalfCritical} holds for
\begin{equation}
T= \frac{c_2}{\log(N)} N^{2}  .
\end{equation} 
\end{lemma}
We start by proving Lemma \ref{lem:MaxSecondClassOrder} on the maximal displacement of the second class particles with $\kappa< \frac{1}{2}$, starting according to $\tilde{\pi}_{N,\rho}$. The proof of Lemma \ref{lem:MaxSecondClassOrderCritical} will be similar, and we will only provide the required adjustments from the proof of Lemma \ref{lem:MaxSecondClassOrder} instead of full details.
\begin{proof}[Proof of Lemma \ref{lem:MaxSecondClassOrder}]
{The line $\{(s, (1-q)(1-2(\rho-N^{-1/2}))T - y N^{1-\expont})\}_{s\in[0,T]}$ is to the right of the left-side of the cylinder $\{(s,\mathcal{I}_{q,\rho}(s,y\theta^{-2/3}N+2\theta^{-1}N,0))\}_{s\in [0,T]}$ for all $N$ large enough. Thus, it suffices for \eqref{eq:MaxDisplacementOrder} to show that there exist $c_0,C_0>0$ with 
\begin{equation}\label{eq:LowerBoundLeft} 
\P(  L_s  >   (1-q)(1-2(\rho-N^{-1/2}))T - y N^{1-\expont}  \, \forall s\in [0,T]) \geq 1- C_0 \exp(-c_0 y^{3} ) 
\end{equation} for all $1 \leq y,\theta  \leq N^{\expont}$ and $N$ large enough.} The upper bound on $(R_t)_{t \geq 0 }$ follows analogously, {swapping the roles of first particles and empty sites in the proof by the particle-hole duality}.
{For a given initial configuration $\tilde{\zeta}_0 \sim \tilde{\pi}_{N,\rho}$ of the multi-species ASEP $(\tilde{\zeta}_t)_{t \geq 0}$ on $\Z$, we
define another multi-species configuration $\bar{\zeta}_0$ as follows. 
Let $(B_i)_{i \in \lbr N\rbr}$ be a family of independent Bernoulli-$N^{-1/2}$-random variables, and set
\begin{equation}
    x_{\ast} := \inf\left\{ x\geq -N^{\kappa+\expon+1/2} \, \colon \zeta_0(x)=\infty \right\} . 
\end{equation}
We define 
\begin{equation}\label{eq:xiProcess2}
\bar{\zeta}_0(x) := \begin{cases}
3 & \text{ if } x\in \big[-N^{\kappa+\expon+1/2},0\big] \setminus \{x_{\ast}\} \text{ and }B_x=1 \text{ and } \zeta_0(x) = \infty  , \\
4 & \text{ if } x=x_{\ast}  ,  \\
\zeta_0 & \text{ otherwise. }
\end{cases}
\end{equation} 
In words, we flip every empty site in $\tilde{\zeta}_0$ within the interval $[-N^{\kappa+\expon+1/2},0]$ with probability $N^{-1/2}$ into a type $3$ particle, apart from the left-most empty site, which is assigned type $4$.  
Let $(\bar{\zeta}_t)_{t \geq 0}$ denote the corresponding multi-species ASEP started from $\bar{\zeta}_0$. We write $(\bar{L}_t)_{t\geq 0}$ and $(\bar{L}_t)_{t\geq 0}$ for the location of the left-most and right-most second class particle in $(\bar{\zeta}_t)_{t \geq 0}$, respectively, and observe that
\begin{equation}\label{eq:CoupleTildeBar}
    L_t =\bar{L}_t \quad \text{ and } \quad     R_t =\bar{R}_t
\end{equation} for all $t\geq 0$ when evolving $(\tilde{\zeta}_t)_{t \geq 0}$ and $(\bar{\zeta}_t)_{t \geq 0}$ together under the basic coupling. }
We enumerate the locations of the type $2$ particles in $(\bar{\zeta}_t)_{t \geq 0}$ from left to right as $(X_t^{1})_{t \geq 0}, (X_t^{2})_{t \geq 0},\dots$, and the locations of the type $3$ particles from left to right as $(Y_t^{1})_{t \geq 0}, (Y_t^{2})_{t \geq 0},\dots$. The location of the type $4$ particle is denoted by $(Z_t)_{t \geq 0}$.
A visualization is given in Step $0$ of Figure~\ref{fig:Projection}.
We define the events
\begin{align}\label{def:A123Events}
\begin{split}
\mathcal{A}_1 &:= \Big\{ \sum_{x\in \Z} \mathds{1}_{ \{ \bar{\zeta}_0(x)=3 \} } \geq \frac{1}{5} N^{\kappa+\expon} \Big\} \\
\mathcal{A}_2 &:= \Big\{ Y_t^{\frac{1}{10} N^{\kappa+\expon}} < X_t^{1} \, \forall t \in [0,N^3] \Big\} \\
\mathcal{A}_3 &:= \Big\{ Z_t < Y_t^{\frac{1}{10} N^{\kappa+\expon}} \, \forall  t \in [0,N^3] \Big\}
\end{split}
\end{align}
on the number and location of the particles of types $2,3$ and $4$ in the multi-species ASEP~$(\bar{\zeta}_t)_{t \geq 0}$.
We will now argue that for our choice of $\expon$ and $\expont$ in \eqref{def:KappaEpsilon},
\begin{equation}\label{eq:LowerEventBound}
\P(\mathcal{A}_1 \cap \mathcal{A}_2 \cap \mathcal{A}_3) \geq 1- C_3\exp( - c_3 N^{3\expont}  )
\end{equation} for some constants $c_3,C_3>0$ and all $N$ large enough. This follows by showing that there exist some constants $c_4,C_4>0$ such that
\begin{equation}\label{eq:BoundonEvents123}
\P(\mathcal{A}_i) \geq 1- C_4 \exp(-c_4 N^{3\expont}) .
\end{equation} for all $i\in \lbr 3 \rbr$, and all $N$ large enough, and a union bound. For $i=1$, \eqref{eq:BoundonEvents123} follows by  a standard tail bound for the sum of independent Bernoulli random variables, {noting that every site on $[-N^{\kappa+\expon+1/2},0]$ is with probability at least $N^{-1/2}/4$ occupied by a type $3$ particle.} For the lower bound on the events $\mathcal{A}_2$ and $\mathcal{A}_3$, we will in the following only consider $\mathcal{A}_2$ as the argument for $\mathcal{A}_3$ is similar. Let $M_2$ and $M_3$ denote the number of type $2$ and type $3$ particles in $(\bar{\zeta}_t)_{t \geq 0}$, respectively. Then for every $t\geq 0$, we define a configuration $\eta_t^{\ast} \in \Omega_{M_2+M_3,M_2}$, where we set
\begin{equation}\label{eq:ClosedStateSpace}
\Omega_{m,k} := \Big\{ \eta \in \{0,1\}^{m} \, \colon  \, \sum_{i=1}^{m} \eta(i) = k \Big\}
\end{equation} for all $m\in \N$ and $k\in \lbr m \rbr$,
by first deleting all particles and empty sites in $\bar{\zeta}_t$ (as well as removing their sites and merging the resulting edges), which are not of type $2$ or $3$. We then  map all type $3$ particles to empty sites and type $2$ particles to first particles to obtain a configuration  $\eta_{t}^{\ast} \in \Omega_{M_2+M_3,M_2}$. Observe that the resulting process $(\eta_t^{\ast})_{t \geq 0}$ can be interpreted as an asymmetric simple exclusion process on $\Omega_{M_2+M_3,M_2}$ with censoring, i.e., an edge $\{x,x+1\}$ is contained in the censoring scheme at time $t$ if and only if in the construction of $\bar{\zeta}_t$ from $\eta_t^{\ast}$, we erased a particle (and its site) between the two sites which get mapped to $x$ and $x+1$, respectively. A visualization of this projection is provided in Figure \ref{fig:Projection}.
 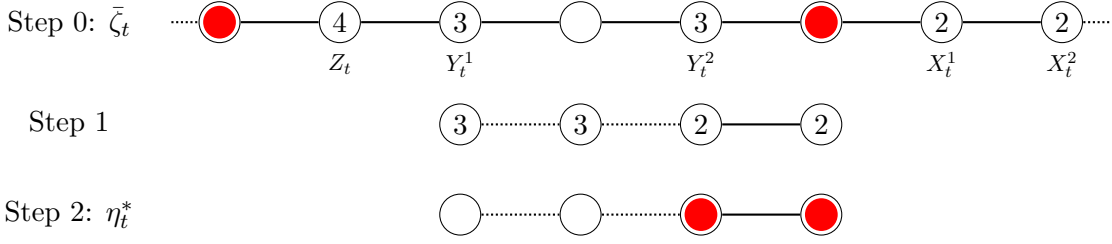
\begin{figure}
\centering
\begin{tikzpicture}[scale=0.8]

\def\x{2};
\def\y{1.3};
\def\z{-0.4};
\def\a{-1.9};

 	\node[shape=circle,scale=1.5,draw] (C1) at (0,\y){} ;
 	\node[shape=circle,scale=1.5,draw] (C2) at (\x,\y){} ;
 	\node[shape=circle,scale=1.5,draw] (C3) at (2*\x,\y){} ;
 	\node[shape=circle,scale=1.5,draw] (C4) at (3*\x,\y){} ;
 	\node[shape=circle,scale=1.5,draw] (C5) at (4*\x,\y){} ;
 	\node[shape=circle,scale=1.5,draw] (C6) at (5*\x,\y){} ;
 	\node[shape=circle,scale=1.5,draw] (C7) at (6*\x,\y){} ;
 	\node[shape=circle,scale=1.5,draw] (C8) at (7*\x,\y){} ;

	\node[scale=0.8] (G1) at (2,0.6) {$Z_t$};
	\node[scale=0.8] (G2) at (4,0.6) {$Y^{1}_t$};
	\node[scale=0.8] (G3) at (8,0.6) {$Y^{2}_t$};
	\node[scale=0.8] (G4) at (12,0.6) {$X^{1}_t$};
	\node[scale=0.8] (G5) at (14,0.6) {$X^{2}_t$};
 	
	\node[shape=circle,scale=1.2,fill=red] (D1) at (C1) {};
	\node (F1) at (C2) {$4$};
	\node (F1) at (C3) {$3$};	
	\node[shape=circle,scale=1.2,fill=red] (D5) at (C6) {};
	\node (F1) at (C5) {$3$};
	\node (F1) at (C7) {$2$};		
	\node (F1) at (C8) {$2$};	
	
 	\draw[thick,densely dotted] (-0.8,\y) -- (C1);
 	\draw[thick] (C1) -- (C2);
 	\draw[thick] (C2) -- (C3);
 	\draw[thick] (C3) -- (C4);
 	\draw[thick] (C4) -- (C5);
 	\draw[thick] (C5) -- (C6);
 	\draw[thick] (C6) -- (C7);
 	\draw[thick] (C7) -- (C8);
 	\draw[thick,densely dotted] (C8) -- (0.8+7*\x,\y);

 	\node (H1) at (-2.5,\y) {Step 0: $\bar{\zeta}_t$};
 	\node (H2) at (-2.5,\z) {Step 1};	
	\node (H3) at (-2.5,\a) {Step 2: $\eta_t^{\ast}$};

 	\node[shape=circle,scale=1.5,draw] (C3) at (2*\x,\z){} ;
 	\node[shape=circle,scale=1.5,draw] (C4) at (3*\x,\z){} ;
 	\node[shape=circle,scale=1.5,draw] (C5) at (4*\x,\z){} ;
 	\node[shape=circle,scale=1.5,draw] (C6) at (5*\x,\z){} ;

	\node (F1) at (C3) {$3$};
	\node (F1) at (C4) {$3$};	
	\node (F1) at (C5) {$2$};
	\node (F1) at (C6) {$2$};

 	\draw[thick,densely dotted] (C3) -- (C4);
 	\draw[thick,densely dotted] (C4) -- (C5);
 	\draw[thick] (C5) -- (C6);

 	\node[shape=circle,scale=1.5,draw] (C3) at (2*\x,\a){} ;
 	\node[shape=circle,scale=1.5,draw] (C4) at (3*\x,\a){} ;
 	\node[shape=circle,scale=1.5,draw] (C5) at (4*\x,\a){} ;
 	\node[shape=circle,scale=1.5,draw] (C6) at (5*\x,\a){} ;

	\node[shape=circle,scale=1.2,fill=red] (D1) at (C5) {};
	\node[shape=circle,scale=1.2,fill=red] (D1) at (C6) {};

 	\draw[thick,densely dotted] (C3) -- (C4);
 	\draw[thick,densely dotted] (C4) -- (C5);
 	\draw[thick] (C5) -- (C6);

	\end{tikzpicture}
\caption{\label{fig:Projection}Construction of a configuration $\eta^{\ast}_t$ from $\bar{\zeta}_t$ for some $t\geq 0$ and $M_2=M_3=2$. First class particles are depicted in red, while second class particles are drawn with their types. In Step 1, we erase all sites apart from those which contain second class particles of types $2$ and $3$. In Step 2, we convert the second class particles of types $2$ and $3$ to first class particles and empty sites, respectively. Censored edges are drawn dashed.}
 \end{figure}
The bound on $\mathcal{A}_2$ in \eqref{eq:BoundonEvents123} is now immediate from Lemma~\ref{lem:BlockingMeasureMaximum} and Remark~\ref{rem:ASEPRestricted} applied to the process $(\eta_t^{\ast})_{t \geq 0}$. {The argument for the event $\mathcal{A}_3$ is similar considering the relative position of type $3$ and type $4$ particles in $(\bar{\zeta}_t)_{t \geq 0}$.}
Observe that 
\begin{equation}\label{eq:DominationBar}
\P \big(Z_s \leq { \bar{L}_s } \, \forall s\in[0,N^3] \, \big| \, \mathcal{A}_1 \cap \mathcal{A}_2 \cap \mathcal{A}_3 \big) =1 .
\end{equation}
By Theorem~\ref{thm:BSCoupling}, $(Z_t)_{t \geq 0}$ dominates the location of a second class particle $(Z^{\prime}_t)_{t \geq 0}$ started from $Z_0$ under a Bernoulli-$(\rho-N^{-1/2})$-product measure. Hence, the desired lower bound \eqref{eq:LowerBoundLeft} on the maximal displacement of the left-most second class particle in $(\bar{\zeta}_t)_{t \geq 0}$, {and hence by \eqref{eq:CoupleTildeBar} also for the left-most second class particle in $(\tilde{\zeta}_t)_{t \geq 0}$},  follows from Lemma~\ref{lem:ModDeviationMaxSecondClass} for $(Z^{\prime}_t)_{t \geq 0}$ and all $1 \leq y,\theta \leq N^{\expont}$. The same arguments also apply in the case $\rho=\frac{1}{2}$, which yields \eqref{eq:MaxDisplacementOrderHalf}.
\end{proof}
\begin{remark}\label{rem:RelaxDomination}
 {   We consider in Lemma~\ref{lem:MaxSecondClassOrder} the initial distribution $\tilde{\pi}_{N,\rho}$ for the multi-species ASEP on $\Z$, which is a product measure. However, a closer inspection of the above arguments shows that this assumption can be relaxed to only requiring that the single-species projections in $\tilde{\pi}_{N,\rho}$ are stochastically dominated from above and below by  Bernoulli-$\big( \rho \pm C^{\prime}N^{-1/2}\big)$-product measures for some constant $C^{\prime}>0$. }
\end{remark}

Next, we describe the necessary changes in the proof of Lemma~\ref{lem:MaxSecondClassOrder} in order to obtain Lemma~\ref{lem:MaxSecondClassOrderCritical} on the maximal displacement of second class particles when $\kappa=\frac{1}{2}$.

\begin{proof}[Proof of Lemma \ref{lem:MaxSecondClassOrderCritical}]
Similar to Lemma \ref{lem:MaxSecondClassOrder}, we argue that there exists some $a=a(\phi)>0$ and some $\theta=\theta(\phi,a)>0$ {in the definition of $T$} so that
\begin{equation}\label{eq:LowerBoundLeftCrit}
\P(  L_s  >  - \phi N  \, \forall s\in [0,T]) \geq 1- N^{-20}/2
\end{equation} for all $N$ large enough. The corresponding bound on $(R_t)_{t \geq 0}$ is similar. Choose an initial configuration $\tilde{\zeta}_0 \sim \hat{\pi}_{N,\rho}$ and define 
\begin{equation*}
\bar{\zeta}_0(x) := \begin{cases}
3 & \text{ if } x\in [-aN,0] \setminus \{x_{\ast}\} \text{ and }B_x=1 \text{ and } \zeta_0(x) = 1  , \\
4 & \text{ if } x=x_{\ast}  ,  \\
\zeta_0 & \text{ otherwise, }
\end{cases}
\end{equation*} 
where we set
\begin{equation}
    x_{\ast} := \inf\left\{ x\geq -aN \, \colon \zeta_0(x)=\infty \right\} , 
\end{equation}
and $(B_i)_{i \in \Z}$ are independent Bernoulli-$(a^{-3}N^{-1/2})$-random variables. 
We define a multi-species ASEP $(\bar{\zeta}_t)_{t \geq 0}$  on $\Z$ with initial configuration $\bar{\zeta}_0$, evolving together with $(\tilde{\zeta}_t)_{t \geq 0}$ according to the basic coupling. 
Enumerating the particles as in Lemma \ref{lem:MaxSecondClassOrder}, let
\begin{align}\label{def:A123EventsPrime}
\begin{split}
\mathcal{A}^{\prime}_1 &:= \Big\{ \sum_{x\in \Z } \mathds{1}_{\{ \bar{\zeta}_0(x)=3 \} } \geq \frac{1}{ a} N^{1/2}\log(N) \Big\} ,\\
\mathcal{A}^{\prime}_2 &:= \Big\{ Y_t^{\frac{1}{2a}N^{1/2}\log(N)} < X_t^{1} \, \forall  t \in [0,N^3] \Big\} , \\
\mathcal{A}^{\prime}_3 &:= \Big\{ Z_t < Y_t^{\frac{1}{2 a}N^{1/2}\log(N)} \, \forall   t \in [0,N^3]  \Big\}  .
\end{split}
\end{align}
The remainder follows from the same arguments as Lemma \ref{lem:MaxSecondClassOrder}, first showing that the event $\mathcal{A}^{\prime}_1 \cap \mathcal{A}^{\prime}_2 \cap \mathcal{A}^{\prime}_3$ holds with probability at least $1-N^{-20}/4$ for all $N$ large enough. Now for $a$ sufficiently small, a standard Chernoff bound for order $a N$ many Bernoulli-$(a^{-3}N^{-1/2}\log(N))$-random variables ensures that $\mathcal{A}_1'$ occurs with probability at least  $1-N^{-20}/6$ for all $N$ large enough. 
Similarly, we use Lemma~\ref{lem:BlockingMeasureMaximum} with $x=\frac{1}{2a}\log(N)$ and Remark~\ref{rem:ASEPRestricted}, to bound the relative position of the type $2$ and type $3$ particles among each other for the event $\mathcal{A}^{\prime}_2$, as well as to bound the relative position of  type $3$ and type $4$ particles for the event $\mathcal{A}^{\prime}_3$. Letting $(\bar{L}_t)_{t \geq 0}$ denote the left-most second class particle in $(\tilde{\zeta}_t)_{t \geq 0}$, we get that
\begin{equation*}
\P \big(Z_s \leq \bar{L}_s \ \forall s\in[0,N^3] \, \big| \, \mathcal{A}^{\prime}_1 \cap \mathcal{A}^{\prime}_2 \cap \mathcal{A}^{\prime}_3 \big) =1 .
\end{equation*} Now choose $a$ in the definition of $\hat{\pi}_{N,\rho}$ small enough so that $\mathcal{A}^{\prime}_1 \cap \mathcal{A}^{\prime}_2 \cap \mathcal{A}^{\prime}_3$ holds with probability at least $1-N^{-20}/4$ for all $N$ large enough. {By the microscopic concavity coupling from Theorem~\ref{thm:BSCoupling} together with Lemma~\ref{lem:ModDeviationMaxSecondClass} for moderate deviations on the maximal displacement of a single second class particle in a Bernoulli-$(\rho-N^{-1/2}\log(N)a^{-3})$-product measure, we find some $\theta>0$ in the definition of $T$ so that  
\begin{equation}
\P(  \bar{L}_s  >  - \phi N  \, \forall s\in [0,T]) \geq 1- N^{-20}/2
\end{equation} for all $N$ large enough. The desired lower bound in \eqref{eq:LowerBoundLeftCrit} follows by evolving $(\zeta_t)_{t \geq 0}$ and $(\bar{\zeta}_t)_{t \geq 0}$ according to the basic coupling. }
\end{proof}

\subsection{Moderate deviations for the extended disagreement process}\label{sec:ModerateSecondClassMulti}

Next, our goal is to compare the number of type $\Aup$ and $\Bup$ second class particles in the extended disagreement process $(\xi^{\modi}_t)_{t \geq 0}$ between an ASEP $(\eta_t)_{t \geq 0}$ on $\Z$, and an ASEP $(\eta^{\N}_t)_{t \geq 0}$ on $\N$. {Recall that $\rho_{\Lup}$ denotes the effective density of the reservoir of the ASEP on $\N$ from \eqref{def:EffectiveDensities}. We assume in the following that both processes are started from a Bernoulli-$\rho$-product measure for some $\rho \in (0,1)$ (on $\{0,1\}^{\N}$ for the ASEP on $\N$, and on $\{0,1\}^{\Z}$ for the ASEP on $\Z$) so that they have the same initial configuration on sites $\N$, i.e.,
\begin{equation}\label{eq:Agreement}
\eta_0(x)=\eta^{\N}_0(x)
\end{equation} for all $x\in \N$, and so that they are coupled according to the extended disagreement process $(\xi_t^{\textup{mod}})_{t \geq 0}$ from Definition \ref{def:BasicCouplingExtended}. }
Let $(Z^{\Aup}_t)_{t \geq 0}$ and $(Z^{\Bup}_t)_{t \geq 0}$ denote the position of the rightmost particle of type $\Aup$ and of type $\Bup$ in $(\xi_t^{\textup{mod}})_{t \geq 0}$, respectively (with the convention that the location is $-\infty$ if there is no such particle). The following proposition bounds the speed of propagation of second class particles in the extended disagreement process.

\begin{proposition}\label{pro:SpeedOfDisagreement}
Let $N \in\N$ and consider $q$ from \eqref{def:TriplePointScaling} with $\kappa<\frac{1}{2}$. { Fix $\mathfrak{a} \in (0,\frac{1}{2})$, and let $\rho \in [\frac{1}{2}+N^{-1/2},1-\mathfrak{a}]$. Recall $\expont>0$ from \eqref{def:KappaEpsilon}. Then there exist some constants $c_0,C_0>0$, depending only on $\mathfrak{a}$ and $\kappa$, }such that for all $1 \leq y,\theta \leq N^{\expont}$, and all $N$ large enough
\begin{equation}\label{eq:SpeedPropagation1}
\P\left( \sup_{t \in [0,\theta^{-1}N^{1+\kappa}(2\rho-1)^{-1}]} \max(Z^{\Aup}_t,Z^{\Bup}_t) \geq \theta^{-1} y N \right) \leq C_0 \exp(-c_0 y^3\theta^{-1}) .
\end{equation} Similarly, when $\rho=\frac{1}{2}$, we get that for all $N$ large enough
\begin{equation}\label{eq:SpeedPropagation2}
\P\left( \sup_{t \in [0,\theta^{-1}N^{3/2+\kappa}]} \max(Z^{\Aup}_t,Z^{\Bup}_t) \geq \theta^{-1} y N \right) \leq C_0 \exp(-c_0 y^3\theta^{-1}) .
\end{equation} {When $\kappa=\frac{1}{2}$, assume that
\begin{equation}\label{eq:FirstSetrho}
    \rho \in \left[\frac{1}{2}+N^{-1/2},\frac{1}{2}+c_1\log^{-1}(N)\right] , 
\end{equation}
where we recall the constant $c_1>0$ from Lemma~\ref{lem:ModDeviationMaxSecondClass}. }
Then for every $\phi>0$, we find some $z>0$ such that for all $N$ large enough
\begin{equation}\label{eq:SpeedPropagation1Critical}
\P\left( \sup_{t \in [0,zN^{3/2}(2\rho-1)^{-1}\log^{-1}(N)]} \max(Z^{\Aup}_t,Z^{\Bup}_t) \geq \phi N \right) \leq \frac{1}{N^{12}} .
\end{equation} Similarly, when $\kappa=\frac{1}{2}$ and $\rho=\frac{1}{2}$, we get that for some $z>0$, and $N$ large enough
\begin{equation}\label{eq:SpeedPropagation2Critical}
\P\left( \sup_{t \in [0,zN^{2}\log^{-1}(N)]} \max(Z^{\Aup}_t,Z^{\Bup}_t) \geq \phi N \right) \leq \frac{1}{N^{12}} .
\end{equation}
\end{proposition}
In words, Proposition~\ref{pro:SpeedOfDisagreement} states that the speed of propagation of a perturbation introduced at the boundary in the extended disagreement process has at most the speed of a second class particle on the integers (with a logarithmic correction for the case $\kappa=\frac{1}{2}$). \\

In order to show Proposition \ref{pro:SpeedOfDisagreement}, we require some setup.
Recall the projection $(\xi^{\Bup}_t)_{t \geq 0}$ from \eqref{def:ASEPaux2} for the extended disagreement process, and let
$(U^{\Bup}_t)_{t \geq 0}$ denote the position of the rightmost type $\Bup$ second class particle in $(\xi^{\Bup}_t)_{t \geq 0}$ (corresponding to the rightmost type $\Bup$ or type $\zerop$ second class particle in $(\xi^{\modi}_t)_{t \geq 0}$), with the convention that $U^{\Bup}_t=-\infty$ if $\xi^{\Bup}_t$ contains no type $\Bup$ second class particles at time $t$. Note that $Z^{\Bup}_t=U^{\Bup}_t$ almost surely for all $t \geq 0$.
Recall that the process $(\xi^{\Bup}_t)_{t \geq 0}$ has the law of a multi-species ASEP on $\Z$ (with special rules for the updates at site $0$ and $1$) and the law of a standard ASEP on $\Z$ with a Bernoulli-$\rho$-product law when projecting type $\Bup$ second class particles to empty sites. 
In order to study $(U^{\Bup}_t)_{t \geq 0}$, it will be convenient to add a new second class particle type {$\Cup$ to the hierarchy  as
\begin{equation}\label{def:ModifiedExtendedPartialOrder}
\one \succeq \onep \succeq \Aup \succeq \Cup \succeq \Bup \succeq \zerop \succeq \zero .
\end{equation} 
We will now define an extension of the process $(\xi^{\Bup}_t)_{t \geq 0}$ on the state space $\{\zero,\one,\Bup,\Cup\}^{\Z}$, which will enable us to bound the location of the right-most type $\Bup$ particle in $(\xi^{\Bup}_t)_{t \geq 0}$. More precisely, define a configuration $\xi_0^{\Bup\Cup} \in \{\zero,\one,\Bup,\Cup\}^{\Z}$  as follows.  
Assume that $\kappa<\frac{1}{2}$, and recall $\expon$ and $\expont>0$ from \eqref{def:KappaEpsilon}, as well as the projection $(\xi^{\Aup}_t)_{t \geq 0}$ from \eqref{def:ASEPaux1}. Let $1 \leq y,\theta \leq N^{\expont}$ be fixed, and let $(\mathcal{U}_x)_{x\in \Z}$ be a family of independent Bernoulli-$N^{-1/2}$-distributed random variables. We obtain the initial configuration $\xi^{\Bup\Cup}_0$ from $\xi^{\Bup}_0$ by setting
 \begin{equation}\label{eq:ErasedDisagProcess}
 \xi_0^{\Bup\Cup}(x) := \begin{cases} \Cup &
  \text{ for }  \lfloor \theta^{-1} y N/4 \rfloor \leq x \leq  \lfloor \theta^{-1} y N/4 \rfloor + N^{\kappa+\expon+\frac{1}{2}}\ ,  \mathcal{U}_x=\xi_0^{\Aup}=1 , \\
  \xi_0^{\Bup}(x) & \text{ otherwise.}
 \end{cases}
 \end{equation} 
 In words, we obtain the initial configuration  $\xi_0^{\Bup\Cup}$ from $\xi_0^{\Bup}$ by changing some of the first class particles on the interval $[\lfloor \theta^{-1} y N/4 \rfloor , \lfloor \theta^{-1} y N/4 \rfloor + N^{\kappa+\expon+\frac{1}{2}}]$ into type $\Cup$ particles. Similarly, when $\kappa=\frac{1}{2}$, we set
 \begin{equation}\label{eq:ErasedDisagProcess3}
 \xi_0^{\Bup\Cup}(x) := \begin{cases} \Cup &
  \text{ for }  \lfloor \phi N/4 \rfloor \leq x \leq  \lfloor \phi N/4 \rfloor + aN\ ,  \mathcal{U}'_x=\xi_0^{\Aup}=1 , \\
  \xi_0^{\Bup}(x) & \text{ otherwise,}
 \end{cases}
\end{equation}
where $(\mathcal{U}'(x))_{x \in \Z}$ are independent Bernoulli-$(N^{-1/2}\log(N)a^{-3})$-random variables, and where 
we recall the constant $a$ from Lemma~\ref{lem:MaxSecondClassOrderCritical}.
Intuitively, we draw the type $\Cup$ particles in the  configuration $\xi^{\Bup\Cup}_0$ according to the second class particles in the measures $\tilde{\pi}_{N,\rho}$ from \eqref{eq:InitialSecondClass1} and \eqref{eq:InitialSecondClass2} when $\kappa<\frac{1}{2}$, shifted by $\lfloor \theta^{-1} y N/4 \rfloor$, as well as the measures $\hat{\pi}_{N,\rho}$ from \eqref{eq:InitialSecondClass1Hat} and \eqref{eq:InitialSecondClass2Hat} for $\kappa=\frac{1}{2}$, shifted by $\lfloor \phi N / 4 \rfloor$. When $\kappa=\frac{1}{2}$ and $\rho=\frac{1}{2}$, then use \eqref{eq:ErasedDisagProcess3} with respect to the density$\rho=\frac{1}{2}+N^{-1/2}$.
}
 
 We write $\pi_{N,\rho}^{\Bup\Cup}$ for the law of $\xi_0^{\Bup\Cup}$, and
define $(\xi^{\Bup\Cup}_t)_{t \geq 0}$ as the corresponding multi-species ASEP on $\{\zero,\one,\Bup,\Cup\}^{\Z}$ started from $\xi_0^{\Bup\Cup}(x)$, treating the type $\Cup$ particles in $(\xi^{\Bup\Cup}_t)_{t \geq 0}$ in the same way as type $\one$ particles at site $1$. 
Let $\tau_{\Cup}$ with
\begin{equation*}
\tau_{\Cup} := \inf\left\{ t \geq 0 \, \colon \, \xi^{\Bup\Cup}_t(1)=\Cup \right\}
\end{equation*} be the first time at which a type $\Cup$ second class particle reached site $1$, and let $(U^{\Bup\Cup}_t)_{t \geq 0}$ denote the location of the rightmost type $\Bup$ second class particle in $(\xi^{\Bup\Cup}_t)_{t \geq 0}$. 

{Note that up to time $\tau_{\Cup}$, when projecting the type $\Cup$ particles in $(\xi^{\Bup\Cup}_t)_{t \geq 0}$ onto type $\one$, we retrieve the process $(\xi^{\Bup}_t)_{t \geq 0}$. 
Moreover, observe that until time $\tau_{\Cup}$, the type $\Cup$ particles have the same law as second class particles in an ASEP on $\Z$, started from $\tilde{\pi}_{N,\rho}$ when $\kappa<\frac{1}{2}$, and from $\hat{\pi}_{N,\rho}$ for $\kappa=\frac{1}{2}$ (shifted by $\lfloor \phi N/4 \rfloor$). In particular, we get that almost surely under the basic coupling
\begin{equation}\label{eq:keyIdent}
U^{\Bup\Cup}_t = U^{\Bup}_t \ \forall 0 \leq t \leq \tau_{\Cup},
\end{equation}
since $\Cup$ and $\one$ have the same priority with respect to $\Bup$. }
Let $M_{\Cup}$ denote the number of type $\Cup$ particles in the configuration $\xi^{\Bup\Cup}_0$.  We denote the locations of the type $\Cup$ particles in $(\xi^{\Bup\Cup}_t)_{t \geq 0}$ from left to right by $(Z^{\Cup,i}_t)_{t \in [0,\tau_{\Cup}]}$ for $i\in \lbr M_{\Cup} \rbr$. We now control the rightmost particle $(U^{\Bup\Cup}_t)_{t \geq 0}$ of type $\Bup$ in $(\xi^{\Bup\Cup}_t)_{t \geq 0}$.

\begin{lemma}\label{lem:TypeBRightofSecond} Fix some $\mathfrak{a} \in (0,\frac{1}{2})$ and let $N \in \N$. Consider the process $(\xi^{\Bup\Cup}_t)_{t \geq 0}$ {with initial distribution $\pi_{N,\rho}^{\Bup\Cup}$. For $\kappa<\frac{1}{2}$ and $\rho\in [\frac{1}{2}+N^{-1/2},1-\mathfrak{a}]$, set
\begin{equation}
    T:= \theta^{-1}N^{1+\kappa}(2\rho-1)^{-1} .
\end{equation}
There exist constants $c_1,C_1>0$ such that for all $1 \leq \theta,y \leq N^{\expont}$,
\begin{equation}\label{eq:FirstCensor}
\mathbf{P}\left( U^{\Bup\Cup}_t \leq Z^{\Cup,M_{\Cup}}_t  \, \forall t \in [0,\min(\tau_{\Cup},T)] \right)  \geq 1 - C_0 \exp( - c_0 N^{\expont})
\end{equation} for all $N$ large enough. When $\kappa<\frac{1}{2}$  and $\rho=\frac{1}{2}$, then \eqref{eq:FirstCensor} holds with $T=\theta^{-1}N^{3/2+\kappa}$. Similarly, when $\kappa=\frac{1}{2}$ and $\rho\in [\frac{1}{2}+N^{-1/2},1-\mathfrak{a}]$, then there exists some $z>0$ such that for $T=zN^{3/2}(2\rho-1)^{-1}\log^{-1}(N)$,
\begin{equation}\label{eq:SecondCensor}
\mathbf{P}\left( U^{\Bup\Cup}_t \leq Z^{\Cup,M_{\Cup}}_t  \, \forall t \in [0,\min(\tau_{\Cup},T)] \right) \geq 1-\frac{1}{4} N^{-12}.
\end{equation} for all $N$ large enough. When $\rho=\frac{1}{2}$, \eqref{eq:SecondCensor} holds with $T=zN^{2}\log^{-1}(N)$. }
\end{lemma}
\begin{proof}
Note that for every configuration $\xi_t^{\Bup\Cup}$ with some $0 \leq t \leq \tau_{\Cup}$, we can associate a configuration $\tilde{\eta}_t^{\Cup} \in \mathfrak{A}_0$ with $\mathfrak{A}_0$ from \eqref{def:BlockingStateSpace} as follows.
Let $M_{\Bup}(t)$ be the number of type $\Bup$ second class particles in $(\xi_t^{\Bup\Cup} )_{t \geq 0}$ at time $t \geq 0$, and denote by $\sigma_t \in \Omega_{M_{\Cup}+M_{\Bup}(t),M_{\Cup}}$ the configuration which we obtain by first deleting all first class particles and empty sites in $\xi_t^{\Cup}$ (as well as removing their sites and merging the resulting edges), and then mapping all type $\Cup$ particles to first class particles and type $\Bup$ second class particles to empty sites; see also Figure~\ref{fig:Projection}.
Note that we can extend each configuration in $\Omega_{M_{\Cup}+M_{\Bup}(t),M_{\Cup}}$ uniquely to an element $\tilde{\eta}_t^{\Cup}$ of $\mathfrak{A}_0$ by adding particles to the right and empty sites to the left of $\sigma_t$. In particular, we have that $\sigma_0=\vartheta_0$ for $\vartheta_0$ defined in \eqref{def:GroundStates}.
As for Lemma \ref{lem:MaxSecondClassOrder}, observe that we can interpret the process $(\tilde{\eta}_t^{\Cup})_{t \geq 0}$ until time $\tau_{\Cup}$ as an asymmetric simple exclusion process on $\mathfrak{A}_0$ with censoring. More precisely, an edge $\{x,x+1\}$ is censored at time $t$ if and only if in the construction of $\tilde{\eta}^{\Cup}_t$ from $\xi_t^{\Bup\Cup}$, we erased a particle (and its site) between the two sites which get mapped to $x$ and $x+1$, respectively, or  we added either $x$ or $x+1$ in the construction of $\tilde{\eta}^{\Cup}_t$ from $\sigma_t$.
To see that \eqref{eq:FirstCensor} holds, note that there exists some constants $c_0,C_0>0$ such that for all $N$ large enough
\begin{equation}\label{eq:AuxEvent}
\P\Big( M_{\Cup}  \geq  \tfrac{1}{8}N^{\kappa+\expont}\Big) \geq 1- C_0\exp(- c_0 N^{\expont})
\end{equation} by the choice of $\xi_0^{\Bup\Cup}$. 
The result now follows by Lemma \ref{lem:BlockingMeasureMaximum} for $(\tilde{\eta}_t^{\Cup})_{t \geq 0}$, noting that whenever the event in \eqref{eq:AuxEvent} holds, we get that $\{ U^{\Bup\Cup}_t \leq Z^{\Cup,M_{\Cup}}_t \}$ occurs at time $t\geq 0$ if the rightmost empty site in $\tilde{\eta}_t^{\Cup}$ is to the left of position $\frac{1}{8}N^{\kappa+\expont}$.
The argument for \eqref{eq:SecondCensor} is analogous.
\end{proof}

{
We are now ready to bound the displacement of the type $\Bup$ particles in $(\xi_t^{\textup{mod}})_{t \geq 0}$.

\begin{lemma}\label{lem:SpeedOfDisagreementAuxi}
Let $N \in\N$ and consider $q$ from \eqref{def:TriplePointScaling} with $\kappa<\frac{1}{2}$. { Fix $\mathfrak{a} \in (0,\frac{1}{2})$, and let $\rho \in [\frac{1}{2}+N^{-1/2},1-\mathfrak{a}]$. Recall $\expont>0$ from \eqref{def:KappaEpsilon}. Then there exists some constants $c_0,C_0>0$, depending only on $\mathfrak{a}$ and $\kappa$, }such that for all $1 \leq y,\theta \leq N^{\expont}$, and all $N$ large enough
\begin{equation}\label{eq:SpeedPropagation1B}
\P\left( \sup_{t \in [0,\theta^{-1}N^{1+\kappa}(2\rho-1)^{-1}]} Z^{\Bup}_t \geq \frac{1}{2}\theta^{-1} y N \right) \leq C_0 \exp(-c_0 y^3) .
\end{equation} Similarly, when $\rho=\frac{1}{2}$, we get that for  all $N$ large enough
\begin{equation}\label{eq:SpeedPropagation2B}
\P\left( \sup_{t \in [0,\theta^{-1}N^{3/2+\kappa}]} Z^{\Bup}_t \geq \frac{1}{2}\theta^{-1} y N \right) \leq C_0 \exp(-c_0 y^3) .
\end{equation} {When $\kappa=\frac{1}{2}$, assume that
\begin{equation}\label{eq:FirstSetrhoB}
    \rho \in \left[\frac{1}{2}+N^{-1/4},\frac{1}{2}+c_1\log^{-1}(N)\right] , 
\end{equation}
with $c_1>0$ from Lemma~\ref{lem:ModDeviationMaxSecondClass}. }
Then for every $\phi>0$, we find some $z>0$ such that
\begin{equation}\label{eq:SpeedPropagation1CriticalB}
\P\left( \sup_{t \in [0,zN^{3/2}(2\rho-1)^{-1}\log^{-1}(N)]} Z^{\Bup}_t \geq \frac{1}{2} \phi N \right) \leq \frac{1}{2}N^{-12}
\end{equation}  for all $N$ large enough. When $\kappa=\frac{1}{2}$ and $\rho=\frac{1}{2}$, we get that for $z>0$, and $N$ large enough
\begin{equation}\label{eq:SpeedPropagation2CriticalB}
\P\left( \sup_{t \in [0,zN^{2}\log^{-1}(N)]} Z^{\Bup}_t \geq \frac{1}{2}\phi N \right) \leq \frac{1}{2}N^{-12} .
\end{equation}
\end{lemma}
}
\begin{proof}
We will only give the proof for  \eqref{eq:SpeedPropagation1B} when $\kappa<\frac{1}{2}$ and $\rho \in [\frac{1}{2}+N^{-1/2},1-\mathfrak{a}]$. The remaining cases follow by the same arguments, using Lemma~\ref{lem:MaxSecondClassOrderCritical} in place of Lemma~\ref{lem:MaxSecondClassOrder} when $\kappa=\frac{1}{2}$.  Define the event
\begin{equation*}
\mathcal{D}^{y,\theta}_N := \left\{  1 < Z^{\Cup,1}_t \leq Z^{\Cup,M_{\Cup}}_t \leq \lfloor \theta^{-1} y N/2 \rfloor  \, \forall t \in [0,\min(\tau_{\Cup},\theta^{-1}N^{1+\kappa}(2\rho-1)^{-1})]  \right\} .
\end{equation*} Observe that the type $\Cup$ particles $((Z^{\Cup,i}_t)_{t \in [0,\tau_{\Cup}]})_{i\in \lbr M_{\Cup} \rbr}$ can until time $\tau_{\Cup}$ be coupled with the second class particles in an ASEP on $\Z$ started from $\tilde{\pi}_{N,\rho}$, which is shifted by $\lfloor \theta^{-1} y N/4 \rfloor$. Then we get from Lemma~\ref{lem:MaxSecondClassOrder} that for all $1\leq y,\theta \leq N^{\expont}$ and some $c_1,C_1>0$,
\begin{equation}\label{eq:DeventBound}
\P( \mathcal{D}^{y,\theta}_N \cap \{ \tau_{\Cup} \geq \theta^{-1}N^{1+\kappa}(2\rho-1)^{-1} \} ) \geq 1-C_1 \exp( - c_1 y^3 )
\end{equation} for all $N$ large enough. Next, consider the time 
\begin{equation*}
    \tau_{\Bup} := \inf\{ t \geq 0 \colon U^{\Bup}_t =  \lfloor \theta^{-1} y N/2 \rfloor \}
\end{equation*}
and define the event 
\begin{equation}\label{def:EventENew}
\mathcal{E}_N := \left\{ \tau_{\Bup} >\theta^{-1}N^{1+\kappa}(2\rho-1)^{-1}  \right\} . 
\end{equation} 
In words, the event $\mathcal{E}_N$ ensures that the rightmost type $\Bup$ particle in $(\xi^{\Bup}_t)_{t \geq 0}$, and hence all type $\Bup$ or type $\zerop$  particles in the extended disagreement process $(\xi^{\modi}_t)_{t \geq 0}$,  will not reach location $\theta^{-1} y N/2$ by time $\theta^{-1}N^{1+\kappa}(2\rho-1)^{-1}$. From \eqref{eq:DeventBound}, together with \eqref{eq:keyIdent} to couple the location of right-most type $\Bup$ particle in $(\xi^{\Bup}_t)_{t \geq 0}$ and $(\xi^{\Bup\Cup}_t)_{t \geq 0}$ and  Lemma~\ref{lem:TypeBRightofSecond} to compare the location of the type $\Bup$ and type $\Cup$ particles in $(\xi^{\Bup\Cup}_t)_{t \geq 0}$,  we find  $c_2,C_2>0$ such that 
\begin{equation}\label{eq:EnLowerBound}
\P( \mathcal{E}_N ) \geq 1 - C_2\exp(-c_2 y^3) ,
\end{equation} for  all $1 \leq y,\theta \leq N^{\expont}$ and all $N$ large enough, allowing us to conclude. 
\end{proof}

We have now all tools to show Proposition \ref{pro:SpeedOfDisagreement} on the location of second class particles in the extended disagreement process $(\xi_t^{\textup{mod}})_{t \geq 0}$.

\begin{proof}[Proof of Proposition \ref{pro:SpeedOfDisagreement}]

We will in the following focus on the proof of \eqref{eq:SpeedPropagation1} when $\kappa<\frac{1}{2}$ and $\rho \in [\frac{1}{2}+N^{-1/2},1-\mathfrak{a}]$, and describe the necessary adjustments for the remaining cases at the end. Due to Lemma~\ref{lem:SpeedOfDisagreementAuxi}, it remains to bound the location of the type $\Aup$ second class particles in  $(\xi_t^{\textup{mod}})_{t \geq 0}$. Recall that $(Z^{\Aup}_t)_{t \geq 0}$ denotes the position of the rightmost particle of type $\Aup$ in  $(\xi_t^{\textup{mod}})_{t \geq 0}$. 

Note that on the event in \eqref{eq:SpeedPropagation1B}, no type $\Bup$ particle will reach location $\frac{1}{2}\theta^{-1}yN$ by time $T=\theta^{-1}N^{1+\kappa}(2\rho-1)^{-1}$. Recall from \eqref{def:ProjectionEta} the projections $(\eta^{\Aup}_t)_{t \geq 0}$ and $(\eta^{\Bup}_t)_{t \geq 0}$. By Lemma~\ref{lem:SecondClassProjection},  $(\eta^{\Aup}_t)_{t \geq 0}$ is an ASEP on $\Z$ and $(\eta^{\Bup}_t)_{t \geq 0}$ an ASEP on $\N$, and both evolve according to the basic coupling. We define 
\begin{equation}
    \zeta^{\star}_{0}(x) := \begin{cases} 2 &\text{ if }  x\in [\frac{3}{8}\theta^{-1}yN,\frac{7}{8}\theta^{-1}yN] \text{ and } \eta^{\Aup}_0(x)=U_x=1   \\
        \infty &\text{ if }  \eta^{\Aup}_0(x)=0 , \\
        1 &\text{ otherwise, } 
    \end{cases}
\end{equation} for all $x\in \Z$, where $(U_x)_{x \in \Z}$ are independent Bernoulli-$(\frac{1}{32}yN^{-1/2})$-random variables. Let $M^{\star}_2$ denote the number of type $2$ particles in $\zeta^{\star}_0$. We let $(\zeta^{\star}_{t})_{t \geq 0}$ denote a multi-species ASEP on $\Z$, started from $\zeta^{\star}_{0}$, and evolving together with $(\eta^{\Aup}_t)_{t \geq 0}$ under the basic coupling. Note that we obtain $(\eta^{\Aup}_t)_{t \geq 0}$ from $(\zeta^{\star}_{t})_{t \geq 0}$  by projecting all second class particles in $(\zeta^{\star}_{t})_{t \geq 0}$ to first class particles. Define the events
\begin{equation}\label{eq:EtildeN}
\begin{split}
     \tilde{\mathcal{E}}^{1}_N &:= \left\{ M^{\star}_2 \geq \frac{1}{100}y^2\theta^{-1}N^{1/2}\right\}, \\
     \tilde{\mathcal{E}}^{2}_N &:=\left\{ \zeta^{\star}_{t}(x) \in \{0,1\} \forall t\in [0,\theta^{-1}N^{1+\kappa}(2\rho-1)^{-1}] , x \notin   \left[\frac{1}{2}\theta^{-1}yN,\theta^{-1}yN\right] \right\} , 
\end{split}
\end{equation} and set 
\begin{equation*}
    \tilde{\mathcal{E}}_N := \tilde{\mathcal{E}}^{1}_N \cap \tilde{\mathcal{E}}^{2}_N . 
\end{equation*}
We claim that there exist constants $c_1,C_1>0$, depending only on $\mathfrak{a}$ and $\kappa$, such that
\begin{equation}
    \P( \tilde{\mathcal{E}}_N ) \geq 1- C_1\exp(-c_1y^3)
\end{equation} for all $1\leq y,\theta \leq N^{\expont}$. To see this, note that we obtain a lower on the probability of $\tilde{\mathcal{E}}^{1}_N$ by a standard Chernoff estimate,
using the fact that $\rho \in [\frac{1}{2},1-\mathfrak
{a}]$. For a lower bound on the probability of $\tilde{\mathcal{E}}^{2}_N$ we apply the same arguments as for Lemma~\ref{lem:MaxSecondClassOrder}, but with second class particle density $yN^{-1/2}$ instead of $N^{-1/2}$,  noting that
\begin{equation}
 T (1-q) \left( \frac{1}{32}yN^{-1/2}\right) \leq  \frac{1}{16} \theta^{-1}y N
\end{equation} for all $\rho \in [\frac{1}{2}+N^{-1/2},1-\mathfrak{a}]$ in order to compare the expected displacement of a second class particle in a Bernoulli-$\rho$-product measure to a Bernoulli-$(\rho-\frac{1}{32}yN^{-1/2})$-product measure.
%
%
Recall that $\eta^{\Aup}_0(x)=\eta^{\Bup}_0(x)$ for all $x\in \N$ and  let the processes $(\eta^{\Aup}_t)_{t \geq 0}$, $(\eta^{\Bup}_t)_{t \geq 0}$ and $(\zeta^{\star}_{t})_{t \geq 0}$ evolve together under the basic coupling. Recall the event $\mathcal{E}_N$ from \eqref{def:EventENew}, and note that  $\mathcal{E}_N$ is measurable with respect to $(\eta^{\Aup}_t,\eta^{\Bup}_t)_{t \geq 0}$. Define
\begin{equation}\label{eq:AtildeN}
    \tilde{\mathcal{A}}_N := \big\{ \exists t \in [0,\theta^{-1}N^{1+\kappa}(2\rho-1)^{-1}] \colon Z_t^{\Aup} = \lfloor \theta^{-1}y N \rfloor \big\} 
\end{equation} to be the event that a type $\Aup$ particle reaches position $\lfloor \theta^{-1}y N \rfloor$ by time $\theta^{-1}N^{1+\kappa}(2\rho-1)^{-1}$.
We claim that when $\tilde{\mathcal{E}}_N$ holds for $(\zeta^{\star}_{t})_{t \geq 0}$, and the events $\mathcal{E}_N$ and $\tilde{\mathcal{A}}_N$  hold for $(\eta^{\Aup}_t,\eta^{\Bup}_t)_{t \geq 0}$, then the event 
\begin{equation}
   \mathcal{G}_N := \left\{ \sum_{x\in [\frac{1}{2}\theta^{-1}yN,\theta^{-1}yN]} \eta_{T}^{\Aup}-\eta_{T}^{\Bup} \geq \frac{1}{100}y^2\theta^{-1}N^{1/2} \right\}  
\end{equation}
occurs. To see this, note that whenever an edge $\{x,x+1\}$ containing a type $\Aup$ particle (i.e., $\zeta^{\star}_{t_-}(x)=\eta^{\Aup}_{t_-}(x)=1$ and $\eta^{\Bup}_{t_-}(x)=0$) at site $x$ and a second class particle at site $x+1$ (i.e., $\zeta^{\star}_{t_-}(x+1)=2$ and $\eta^{\Aup}_t(x+1)=\eta^{\Bup}_{t_-}(x+1)=1$) is updated at time $t$, then this results with probability $(1+q)^{-1}$ in 
\begin{equation*}
    (\zeta^{\star}_{t}(x),\eta^{\Aup}_{t}(x),\eta^{\Bup}_{t}(x))= (2,1,0) \quad \text{ and } \quad (\zeta^{\star}_{t}(x+1),\eta^{\Aup}_{t}(x+1),\eta^{\Bup}_{t}(x+1))= (1,1,1) , 
\end{equation*} and with probability $q(1+q)^{-1}$ in 
\begin{equation*}
 (\zeta^{\star}_{t}(x),\eta^{\Aup}_{t}(x),\eta^{\Bup}_{t}(x))= (1,1,1)    \quad \text{ and } \quad (\zeta^{\star}_{t}(x+1),\eta^{\Aup}_{t}(x+1),\eta^{\Bup}_{t}(x+1))= (2,1,0).  
\end{equation*}
In particular, the type $\Aup$ particle is absorbed and the second class particle in $(\zeta^{\star}_{t})_{t \geq 0}$ is paired with an empty site. Thus, under the above events $\tilde{\mathcal{E}}_N$, $\mathcal{E}_N$ and $\tilde{\mathcal{A}}_N$, every second class particle in $\zeta^{\star}_{T}$ must indeed be paired with an empty site in $\eta^{\Bup}_{T}$. In particular, this ensures that $\mathcal{G}_N$ occurs (as due to the event $\mathcal{E}_N$ there are no type $\Bup$ particles in $[\frac{1}{2}\theta^{-1}yN,\theta^{-1}yN]$ by time $T$). 

Since by our assumptions, $\eta^{\Aup}_{T}$ and $\eta^{\Bup}_{T}$ are both distributed according to a Bernoulli-$\rho$-product measure, a standard moderate deviation estimate on the sum of $\lfloor \frac{1}{2}\theta^{-1}yN \rfloor$ many independent Bernoulli-$\rho$-random variables yields
\begin{equation}
    \P(\mathcal{G}_N) \leq  C_2\exp(-c_2 y^3 \theta^{-1} )
\end{equation}
for some $c_2,C_2>0$ and all $1 \leq y,\theta \leq N^{\expont}$. In total, we get that
\begin{equation}
\begin{split}
     \P( \tilde{\mathcal{A}}_N) &\leq \P(\tilde{\mathcal{A}}_N\cap \mathcal{E}_N \cap \tilde{\mathcal{E}}_N) + \P((\mathcal{E}_N)^{\complement})  +  \P((\tilde{\mathcal{E}}_N)^{\complement}) \\
     &\leq \P(\mathcal{G}_N) + \P((\mathcal{E}_N)^{\complement})  +  \P((\tilde{\mathcal{E}}_N)^{\complement}) \\
     &\leq C_3\exp(-c_3 y^3 \theta^{-1}  )
\end{split}
\end{equation}
for some  $c_3,C_3>0$, allowing us to conclude when $\kappa<\frac{1}{2}$ and $\rho \in [\frac{1}{2}+N^{-1/2},1-\mathfrak{a}]$. Note that for $\rho=\frac{1}{2}$, the same arguments apply, substituting $(2\rho-1)^{-1}$ by $N^{1/2}$ in the definition of $T$ as well as in the definition of the events $\tilde{\mathcal{E}}_N^2$ in \eqref{eq:EtildeN} and $\tilde{\mathcal{A}}_N$ in \eqref{eq:AtildeN}. When $\kappa=\frac{1}{2}$, we again apply the same arguments (with Lemma~\ref{lem:MaxSecondClassOrderCritical} in place of Lemma~\ref{lem:MaxSecondClassOrder}) for $y=C_{\star}\log(N)$ and $\theta=2y/\phi$ with a sufficiently large constant $C_{\star}>0$.
\end{proof}

\subsection{From the ASEP on $\Z$ to the ASEP on $\N$}\label{sec:IntergersToN}

In the following, we aim to show that under the basic coupling, the current of the ASEP on $\Z$ and the ASEP on $\N$ started from a common initial configuration chosen according to a Bernoulli-$\rho$-product measure differ after a time of order $T=(1-q)^{-1}(2\rho-1)^{-1}N$ at most by order $\sqrt{N}$. {Here, we assume that the effective density $\rho_{\Lup}$ from \eqref{def:EffectiveDensities} agrees with the parameter $\rho$.}
Recall that we denote by $(\Jz_t)_{t \geq 0}$ and $(\Jn_t)_{t \geq 0}$ the current of the ASEPs on $\Z$ and $\N$ through site $1$, respectively.

\begin{proposition}\label{pro:CurrentZtoN} Let $q$ satisfy \eqref{def:TriplePointScaling} for some $\kappa<\frac{1}{2}$ and $\psi>0$, and recall $\expont>0$ from \eqref{def:KappaEpsilon}. Fix some $\mathfrak{a}\in (0,1\frac{1}{2}$, and let $N\in \N$. Assume that $\rho \in [\frac{1}{2}+N^{-1/2},1-\mathfrak{a}]$, and set
\begin{equation}
T =  (1-q)^{-1}(2\rho-1)^{-1}N.
\end{equation} Consider the ASEP on $\N$ started from a Bernoulli-$\rho$-product measure. Then there exist constants $c_0,C_0>0$ such that for all $1 \leq x \leq N^{\expont} $, and $N$ large enough
\begin{equation}\label{eq:CurrentCompareZN}
 \P\left( | \Jn_T - T \rho(1-\rho)(1-q) | > x \sqrt{N} \right) \leq C_0 \exp( -  c_0 x^{1/2} ) .
\end{equation} For $\rho=\frac{1}{2}$, the statement \eqref{eq:CurrentCompareZN} holds with respect to $T=N^{3/2}(1-q)^{-1}$. 
When $\kappa=\frac{1}{2}$ and 
\begin{equation}
    \rho \in \left[ \frac{1}{2} + N^{-1/2}, \frac12 +c_1 \log^{-1}(N) \right]
\end{equation}
for the constant $c_1>0$ from Lemma \ref{lem:ModDeviationMaxSecondClass}, there exists some constant $C_1>0$ such that for all $N$ large enough
\begin{equation}\label{eq:CurrentCompareZNCritical}
 \P\left( | \Jn_T - T \rho(1-\rho)(1-q) | > C_1 \log(N)\sqrt{N} \right) \leq N^{-11} ,
\end{equation} where we set $T = (2\rho-1)^{-1} N^{3/2}$. For $\rho=\frac{1}{2}$, we get that \eqref{eq:CurrentCompareZNCritical} holds with $T=N^2$.
\end{proposition}

{In preparation for the proof of Proposition~\ref{pro:CurrentZtoN}, we recall the following version of Theorem~\ref{thm:CurrentASEPOriginal}, adapted to our setup.

\begin{theorem}[c.f. Theorem~2.4 in \cite{LS:Tails}]\label{thm:CurrentASEP} Fix some $\mathfrak{a}\in (0,\frac{1}{2})$ 
Let $q$ satisfy \eqref{def:TriplePointScaling} for some $\kappa \in [0,\frac{1}{2}]$ and $\psi>0$.  Let $\rho\in [\frac{1}{2}+N^{-1/2},1-\mathfrak{a}]$. Set
\begin{equation}
T=\theta^{-1} (1-2\rho)^{-1}  N(1-q)^{-1}
\end{equation}
 for some $\theta \geq 1$, allowed to depend on $N$. Then there exist constants $c_0,C_0>0$, depending only on $\mathfrak{a}$, such that for all $\theta^{-1/3} \leq y \leq (1-2\rho)^{-1} N^{1/2}\theta^{-1}$ and any fixed $m\in \Z$, we get 
\begin{equation}\label{eq:CurrentModerate}
\P( | \Jz_T(m) - T \rho(1-\rho)(1-q) | \geq y N^{1/2} ) \leq C_0 \exp\left(-c_0 \min(y^{3/2} \sqrt{\theta}, y^2 \theta ) \right)
\end{equation} for all $N$ large enough. The same result holds when $\rho=\frac{1}{2}$ and $T=N^{3/2}(1-q)^{-1}$.
\end{theorem}
\begin{proof}
This follows from Theorem~\ref{thm:CurrentASEPOriginal} with a suitable change of notation. More precisely, set $w = y  \theta^{1/3} N^{1/6}(2\rho-1)^{1/3} \geq y \theta^{1/3}$, and note that we need to ensure $w \geq 1$ and 
\begin{equation*}
w= y  \theta^{1/3} N^{1/6}(2\rho-1)^{1/3} \leq  ((1-q)T)^{2/3} =\theta^{-2/3}(2\rho-1)^{-2/3} N^{2/3}  
\end{equation*} which is the case for all $\theta^{-1/3} \leq y \leq (2\rho-1)^{-1} N^{1/2}\theta^{-1}$, and the above choices for $\rho$.
\end{proof}

We have now all tools to compare the current of the ASEP on $\Z$ and the ASEP on $\N$. }

\begin{proof}[Proof of Proposition~\ref{pro:CurrentZtoN}] We will in the following only give the arguments when $\kappa<\frac{1}{2}$. The case $\kappa=\frac{1}{2}$ is similar.
We aim to show that there exist constants $c_0,C_0>0$ such that for all $1 \leq x \leq N^{\expont}$
\begin{equation}\label{eq:UpperBoundCurrentN}
\P\left( \Big| \Jn_T -  T\rho(1-\rho)(1-q) \Big| \leq x\sqrt{N} \right) \geq 1- C_0\exp(-c_0 x^{1/2} )
\end{equation} for all $N$ large enough. For every fixed $z \in \lbr  N^{\expont} \rbr$ and $k\in \lbr z \rbr$, we define the events 
\begin{equation}\label{eq:UpperAux}
\mathcal{A}_k := \left\{ \Big| \Jn_{kT/z}-\Jn_{(k-1)T/z} -  \rho(1-\rho)(1-q)T/z  \Big| \leq \sqrt{N}  \right\} . 
\end{equation} 
Note that it suffices for \eqref{eq:UpperBoundCurrentN} to show that
\begin{equation}\label{eq:TargetUpperAux}
\P(\mathcal{A}_1)=\P(\mathcal{A}_k) \geq 1- C_2\exp(-c_2 z^{1/2})
\end{equation}
for some constants $C_2,c_2>0$. Since the ASEP on $\N$ is stationary, and the event $\mathcal{A}_k$ is measurable with respect to $(\eta^{\N}_t)_{t \geq 0}$, this yields the equality in \eqref{eq:TargetUpperAux}.
In order to show the inequality in \eqref{eq:TargetUpperAux}, let
\begin{equation*}
\mathcal{B}_k := \left\{ | \Jz_{kT/z} - \Jz_{(k-1)T/z}  -  T\rho(1-\rho)(1-q)/z | \leq  \sqrt{N} \right\} .
\end{equation*}
By Theorem \ref{thm:CurrentASEP} on moderate deviations for the current of the ASEP on $\Z$ (with $\theta=z$ and $y=1$), we get that for some constants $c_3,C_3>0$ and all $z \in \lbr N^{\expont} \rbr$ with $\expont$ from \eqref{def:KappaEpsilon}
\begin{equation*}
\P(\mathcal{B}_k) = \P(\mathcal{B}_1) \geq 1- C_3\exp(-c_3 z^{1/2}) ,
\end{equation*} using the stationarity of the ASEP on $\Z$. Note that $\eta^{\Z}_{(k-1)T/z}$ and $\eta^{\N}_{(k-1)T/z}$ have the same law on $\N_0$, and that $\Jz_{kT/z} - \Jz_{(k-1)T/z}$ as well as $\Jn_{kT/z} - \Jn_{(k-1)T/z}$ are measurable with respect to  $(\eta^{\Z}_{t})_{t \in [(k-1)T/z,kT/z)}$ and $(\eta^{\N}_{t})_{t \in [(k-1)T/z,kT/z)}$, respectively. Suppose that there exists some coupling $\mathbf{P}_k$ such that the event
\begin{equation}\label{eq:CurrentComCk}
\mathcal{C}_k := \left\{ \Big| (\Jz_{kT/z}-\Jz_{(k-1)T/z}) - (\Jn_{kT/z}-\Jn_{(k-1)T/z}) \Big| \leq \sqrt{N} \right\}
\end{equation}
satisfies for some constants $c_4,C_4 >0$ and all $z \in \lbr N^{\expont} \rbr$ and $k\in \lbr z \rbr$
\begin{equation}\label{eq:TargetUpperAuxNew}
\mathbf{P}_k\left( \mathcal{C}_k \right) \geq 1- C_4\exp(-c_4 z^{1/2}) .
\end{equation}
Then we get by a union bound on $k\in \lbr z \rbr$
\begin{equation*}
\P\left( \Big| \Jn_T -  \rho(1-\rho)(1-q)T \Big| \leq 2z \sqrt{N} \right)\geq 1 - z \P(\mathcal{B}_1^{\complement}) - \sum_{k=1}^{\lbr z \rbr}\mathbf{P}_k(\mathcal{C}_k^{\complement}) ,
\end{equation*} implying that \eqref{eq:TargetUpperAux} follows indeed from \eqref{eq:TargetUpperAuxNew} and a change of variables. \\

It remains to show that \eqref{eq:TargetUpperAuxNew} holds for all $k \in \N$. For the coupling $\mathbf{P}_k$, since $\eta^{\Z}_{(k-1)T/z}$ and $\eta^{\N}_{(k-1)T/z}$ have the same law, assume that $\eta^{\Z}_{(k-1)T/z}=\eta^{\N}_{(k-1)T/z}$, and evolve both processes according to the basic coupling $\mathbf{P}$ between times $(k-1)T/z$ and $kT/z$. Since the processes $(\eta^{\Z}_t)_{t \geq 0}$ and $(\eta^{\N}_t)_{t \geq 0}$ are stationary, it suffices to show  \eqref{eq:TargetUpperAuxNew} for $k=1$. Recall from the definition of the extended disagreement process that the current of the ASEP on $\N$ is at time $T/z$ given by the number of type $\Bup$ second class particles minus type $\Aup$ second class particles plus the current of the ASEP on $\Z$, i.e.,
\begin{equation}\label{eq:CurrentIdentity}
\Jn_{T/z} = \Jz_{T/z} + \sum_{x\in \N_0} \mathds{1}_{\big\{\xi^{\modi}_{T/z}(x)=\Bup\big\}} - \mathds{1}_{\big\{\xi^{\modi}_{T/z}(x)=\Aup\big\}} .
\end{equation} From Proposition \ref{pro:SpeedOfDisagreement} (with $\theta=z$ and $y=\frac{1}{2}z^{1/2}$), we get that for constants $c_5,C_5>0$
\begin{equation*}
\mathbf{P}\bigg( \sum_{v > Nz^{-1/2}}  \mathds{1}_{\big\{\xi^{\modi}_{T/z}(v)=\Bup\big\}}+\mathds{1}_{\big\{\xi^{\modi}_{T/z}(v)=\Aup\big\}} >0 \bigg) \leq C_5 \exp(-c_5 z^{1/2})
\end{equation*}  for all $z \in  \lbr N^{\expont} \rbr$, and $N$ large enough.
Moreover, there exist constants $c_6,C_6>0$ such that
\begin{equation}\label{eq:ParticleDifferenceBulk}
\mathbf{P}\bigg( \Big| \sum_{v \leq Nz^{-1/2}} \mathds{1}_{\big\{\xi^{\modi}_{T/z}(v)=\Bup\big\}} - \mathds{1}_{\big\{\xi^{\modi}_{T/z}(v)=\Aup\big\}}  \Big| \geq  \sqrt{N} \bigg) \leq C_6 \exp(-c_6 z) .
\end{equation} for all $z\leq N^{\expont}$ and $N$ large enough. This follows from the observation that by Lemma \ref{lem:SecondClassProjection}, the marginals $\eta^{\Aup}_{T/z}$ and $\eta^{\Bup}_{T/z}$ are Bernoulli-$\rho$-product measures on sites $\leq Nz^{-1/2}$, and a standard moderate deviation bound. Combining now \eqref{eq:CurrentIdentity} and \eqref{eq:ParticleDifferenceBulk}, we obtain \eqref{eq:TargetUpperAuxNew}.
\end{proof}

We get the following immediate consequence on the moderate deviations for the current of the ASEP on $\N$. This complements results by Barraquand et al.~in \cite{BBCW:HalfspaceASEP} and by He in \cite{H:Boundary}, who obtained limit theorems for the current of the ASEP on $\N$ starting from empty initial conditions, and which are expected to similarly hold for stationary initial data. 

\begin{corollary}\label{cor:ModDevitationCurrentN}
Let $\kappa<\frac{1}{2}$. Then under the same setup as in Proposition \ref{pro:CurrentZtoN},
\begin{equation}\label{eq:VarianceOnN}
\Var[\Jn_T] \leq C_0 N
\end{equation} for some constant $C_0$ and all $N$ large enough. Similarly, for $\kappa=\frac{1}{2}$, there exists a constant $C_0^{\prime}$ such that for all $N$ large enough
\begin{equation}\label{eq:VarianceOnNCritical}
\Var[\Jn_T] \leq C_0^{\prime} N \log^4(N) .
\end{equation}
\end{corollary}
\begin{proof}
We will only consider the case that $\kappa<\frac{1}{2}$ as the arguments for $\kappa=\frac{1}{2}$ are again similar. In view of \eqref{eq:CurrentCompareZN}, it suffices for \eqref{eq:VarianceOnN} to argue that for all $N$ large enough
\begin{equation}\label{eq:PreCS}
\E\big[ (\Jn_T)^{2} \mathds{1}_{\mathcal{A}} \big] \leq 1 ,
\end{equation}
 where we define the event
\begin{equation*}
\mathcal{A} := \left\{ | \Jn_T - { T \rho(1-\rho)(1-q) }| \geq N^{1/2}\log^{2}(N)  \right\} .
\end{equation*} Applying the Cauchy--Schwarz inequality, we see that
\begin{equation}\label{eq:CS}
\E\big[ (\Jn_T)^{2} \mathds{1}_{\mathcal{A}} \big]  \leq \E\big[ (\Jn_T)^{4} \big]^{\frac{1}{2}} \P(\mathcal{A})^{\frac{1}{2}}  .
\end{equation} Using a simple Poisson bound, we get that for some constant $c_1>0$,
\begin{equation}\label{eq:PoissonCurrentOnN}
 \P( \Jn_T \geq 4 y T  ) \leq \exp( - c_1 y T )
\end{equation} for all $y\geq 1$ and $N\in \N$, counting the number of times the Poisson clocks at the origin rings until time $T$. In particular, we have the very rough bound $\E\big[ (\Jn_T)^{4} \big] \leq N^9$ for all $N$ large enough. Together with \eqref{eq:CurrentCompareZN} for $x=\log^{2}(N)$, this yields \eqref{eq:PreCS}, and thus \eqref{eq:VarianceOnN}.
\end{proof}

\subsection{From the ASEP on $\N$ to the open ASEP}\label{sec:NToOpen}

In the following, we transfer the moderate deviation result for the current of the ASEP $(\eta^{\N}_t)_{t \geq 0}$ on $\N$ to the current of the open ASEP $(\eta_t)_{t \geq 0}$. Recall that we denote by $\rho_{\Lup}$ and $\rho_{\Rup}$ the effective reservoir densities of the open ASEP. We assume that $(\eta^{\N}_t)_{t \geq 0}$ and $(\eta_t)_{t \geq 0}$ share the same boundary parameters $\alpha>0$ and $\gamma \geq 0$. Recall the basic coupling $\mathbf{P}$ from Definition \ref{def:BasicCoupling} between $(\eta^{\N}_t)_{t \geq 0}$ and $(\eta_t)_{t \geq 0}$, and note that second class particles are only created at site $N$. {Let $\mathfrak{a}\in (0,\frac{1}{2})$.} Suppose that
\begin{equation}\label{eq:ProductCondition}
\rho_{\Lup}=\rho_{\Rup}=\rho
\end{equation}
{for some $\rho \in [\frac{1}{2}+N^{-1/2},1-\mathfrak{a}] \cup \{\frac{1}{2}\}$.} Moreover, assume that $\eta^{\N}_0(v)=\eta_0(v)$
 for all $v \in \lbr N \rbr$, and that $\eta_0 \sim \text{Ber}_{\N}(\rho)$, {where we recall that $\text{Ber}_{\N}(\rho)$ denotes a Bernoulli-$\rho$-product  measure on $\{0,1\}^{\N}$.} Let $(Z_t)_{t \geq 0}$ denote the position of the left-most second class particle in the disagreement process $(\xi_t)_{t \geq 0}$ between $(\eta^{\N}_t)_{t \geq 0}$ and $(\eta_t)_{t \geq 0}$, with the convention $Z_t=-\infty$ if a second class particle has exited by time $t$ at site $1$. We have the following lemma on $(Z_t)_{t \geq 0}$. Since we will apply the same arguments as in Proposition~\ref{pro:SpeedOfDisagreement}, we will only describe the necessary changes in the proof.

\begin{lemma}\label{lem:SpeedOfPropagationCoupling}
Let $q$ from \eqref{def:TriplePointScaling} satisfy $\kappa<\frac{1}{2}$. Fix some $\mathfrak{a}\in (0,\frac{1}{2})$. Let $N\in \N$, and let  $\rho \in [\frac{1}{2}+N^{-1/2},1-\mathfrak{a}]$. 
Let $\theta \geq 1$ and set
\begin{equation}
    T = \theta^{-1}N^{1+\kappa}(2\rho-1)^{-1} . 
\end{equation}
Then there exists constants $c_0,C_0>0$ and $\expont>0$ such that for all $1 \leq \theta, y \leq N^{\expont}$,
\begin{equation}\label{eq:OpenSecond1}
\P( Z_s \geq N - y\theta^{-1} N \, \forall s \in [0,T] ) \geq 1- C_0 \exp(-c_0 y^{3}\theta^{-1})
\end{equation} for all $N$ large enough. Similarly, for $\rho=\frac{1}{2}$, we get that \eqref{eq:OpenSecond1} holds with  $T=\theta^{-1}N^{3/2+\kappa}$. 
Now let $\kappa=\frac{1}{2}$. Then for every $\phi>0$, there exists some $z>0$ such that for all 
\begin{equation}
    \rho \in \left[\frac{1}{2}+N^{-1/2},\frac{1}{2}+c_1 \log^{-1}(N)\right]
\end{equation}
 with $c_1>0$ from Lemma \ref{lem:ModDeviationMaxSecondClass}, we get
\begin{equation}\label{eq:OpenSecond2}
\P( Z_s \geq (1-\phi)N \ \forall s \in [0,T] ) \geq 1- \frac{1}{2}N^{-10}
\end{equation} with $T=zN^{3/2}\log^{-1}(N)(2\rho-1)^{-1}$, for all $N$ large enough. Similarly, when $\rho=\frac{1}{2}$, then we get that \eqref{eq:OpenSecond2} holds with respect to $T=zN^{2}\log^{-1}(N)$.
\end{lemma}
\begin{proof}
Consider the extended disagreement process $(\xi^{\modi}_t)_{t \geq 0}$ between $(\eta_t)_{t \geq 0}$ and $(\eta^{\N}_t)_{t \geq 0}$ defined as  in Definition \ref{def:BasicCouplingExtended}, but where second class particles of types $\Aup$ and $\Bup$ are created at site $N$. We proceed as in the proof of Proposition \ref{pro:SpeedOfDisagreement} to bound the location of the leftmost second class particle of types $\Aup,\onep,\Bup,\zerop$ in $(\xi^{\modi}_t)_{t \geq 0}$,  respectively. However, since second class particles enter from the left in the present setup (corresponding to creating second class particles at site $N$ in Proposition \ref{pro:SpeedOfDisagreement}), we use the reversed partial order from Remark~\ref{rem:ReverseOrder} and first bound the location of the leftmost type $\Aup$ and type $\onep$ second class particles and then the leftmost type $\Bup$ and type $\zerop$ second class particle. Moreover, we require that the estimates from Lemma~\ref{lem:MaxSecondClassOrder} and Lemma~\ref{lem:MaxSecondClassOrderCritical} continue to hold  for a collection of second class particles placed in the ASEP on $\N$. This is ensured by Proposition~\ref{pro:SpeedOfDisagreement} for the basic coupling between a multi-species ASEP on $\Z$ and on $\N$.
\end{proof}

We are now ready to state the main result on the current of the open ASEP when the invariant measure $\mu_N$ is a Bernoulli-$\rho$-product measure with {$\rho \in [\frac{1}{2}+N^{-1/2},1-\mathfrak{a}] \cup \{\frac{1}{2}\}$}. Since the arguments are analogous to Proposition~\ref{pro:CurrentZtoN}, we will again only describe the necessary adjustments required for the proof.

\begin{proposition}\label{pro:CurrentNtoO}
Let $q$ satisfy \eqref{def:TriplePointScaling} for some $\kappa<\frac{1}{2}$ and recall $\expont$ from \eqref{def:KappaEpsilon}. Fix some $\mathfrak{a} \in (0,\frac{1}{2})$.  
For all $N\in \N$, suppose that $ \rho \in \left[\frac{1}{2}+N^{-1/2},1-\mathfrak{a}\right]$
and set
\begin{equation}
T = (2\rho-1)^{-1} N (1-q)^{-1}.
\end{equation}  {Consider a stationary open ASEP $(\eta_t)_{t \geq 0}$ on the interval of length $N$, and assume that the boundary parameters $\alpha,\beta,\gamma,\delta \geq 0$ are such that $\mu_N$ is a Bernoulli-$\rho$-product measure on $\{0,1\}^{N}$. Recall that we denote by $(\Jo_t^N)_{t \geq 0}$ the current of $(\eta_t)_{t \geq 0}$. Then there exist constants $c_0,C_0>0$, depending only on $\mathfrak{a}$,  such that for all $1 \leq x \leq N^{\expont}$, and $N$ large enough
\begin{equation}\label{eq:CurrentOpen}
 \P\left( | \Jo_T^N - T\rho(1-\rho)(1-q) | > x \sqrt{N} \right) \leq C_0 \exp\big( -  c_0 x^{1/2} \big) .
\end{equation} For $\rho=\frac{1}{2}$, the statement \eqref{eq:CurrentOpen} holds with respect to $T=N^{3/2}(1-q)^{-1}$. When $\kappa=\frac{1}{2}$ and 
\begin{equation}
    \rho \in \left[\frac{1}{2}+N^{-1/2},\frac{1}{2}+c_1 \log^{-1}(N)\right]
\end{equation} with $c_1>0$ from Lemma~\ref{lem:ModDeviationMaxSecondClass},} there exists some constant $C_1>0$ such that
\begin{equation}\label{eq:CurrentOpenCritical}
 \P\left( | \Jo_T^N - T \rho(1-\rho)(1-q) | > C_1 \log(N)\sqrt{N} \right) \leq N^{-10}
\end{equation} for $T = (2\rho-1)^{-1}N(1-q)^{-1}$. For $\rho=\frac{1}{2}$, we get that \eqref{eq:CurrentOpenCritical} holds with $T=N^{2}$.
\end{proposition}
\begin{proof}
Let $\kappa<\frac{1}{2}$ and let $(\eta^\N_t)_{t \geq 0}$ an ASEP on $\N$ with the same parameters $\alpha,\gamma \geq 0$ as $(\eta_t)_{t \geq 0}$. Recall that  $(\Jn_{t})_{t \geq 0}$ denotes the current of $(\eta^\N_t)_{t \geq 0}$.
Note that it suffices for \eqref{eq:CurrentOpen} to show that for all $1 \leq z \leq N^{\expont}$,
\begin{equation}\label{eq:UpperAuxOpen}
\begin{split}
\tilde{\mathcal{B}}_k &:= \left\{ \Big| \Jn_{kT/z}-\Jn_{(k-1)T/z} -  \rho(1-\rho)(1-q)T/z  \Big| \leq \sqrt{N}  \right\}  \\
\tilde{\mathcal{C}}_k &:= \left\{ \Big| \Jn_{kT/z}-\Jn_{(k-1)T/z} + \Jo_{kT/z}^N-\Jo_{(k-1)T/z}^N \Big| \leq \sqrt{N} \right\}
\end{split}
\end{equation} satisfy for all $k \in \lbr z \rbr$ with some constants $c_2,C_2,c_3,C_3>0$
\begin{align}
\label{eq:BkBound}\P(\tilde{\mathcal{B}}_k) &= \P(\tilde{\mathcal{B}}_1) \geq 1- C_2\exp(-c_2 z^{1/2}) \\
\label{eq:CkBound}\P(\tilde{\mathcal{C}}_k) &= \P(\tilde{\mathcal{C}}_1) \geq 1- C_3\exp(-c_3 z^{1/2}) .
\end{align}
The lower bound in \eqref{eq:BkBound} follows from \eqref{eq:TargetUpperAux} in Proposition~\ref{pro:CurrentZtoN}. The lower bound in \eqref{eq:CkBound} follows from the same arguments as \eqref{eq:TargetUpperAuxNew} in the proof of Proposition~\ref{pro:CurrentZtoN}, replacing Proposition \ref{pro:SpeedOfDisagreement} by Lemma  \ref{lem:SpeedOfPropagationCoupling} to bound the location of the second class particles in the extended disagreement process between $(\eta^\N_t)_{t \geq 0}$ and $(\eta_t)_{t \geq 0}$. The case $\kappa=\frac{1}{2}$ follows by the same arguments. 
\end{proof}



Using Propositions~\ref{pro:CurrentZtoN} and~\ref{pro:CurrentNtoO}, we  have the following moderate deviation bound for the current of the open ASEP. The proof is identical to Corollary~\ref{cor:ModDevitationCurrentN} and therefore omitted.

\begin{corollary}\label{cor:ModDevitationCurrentO}
{Fix some $\mathfrak{a} \in (0,\frac{1}{2})$, and recall $q$ from \eqref{def:TriplePointScaling} with some $\kappa \in [0,\frac{1}{2}]$. Consider an open ASEP started from its stationary distribution $\mu_N=\textup{Ber}_N(\rho)$ for some $\rho \in [\frac{1}{2}+N^{-1/2},1-\mathfrak{a}]$, and set
\begin{equation}
    T = (1-q)^{-1}N(2\rho-1)^{-1} . 
\end{equation}
For $\kappa<\frac{1}{2}$,  we get that
\begin{equation}\label{eq:VarianceO}
\Var[\Jo_T^N] \leq C_1 N
\end{equation} for some constant $C_1>0$ and all $N$ large enough. Similarly, when $\rho=\frac{1}{2}$, \eqref{eq:VarianceO} holds with respect to $T=(1-q)^{-1}N^{3/2}$. 
For $\kappa=\frac{1}{2}$, we set $T=N^{3/2}(1-2\rho)^{-1}$ when $\rho \in [\frac{1}{2}+N^{-1/2},1-\mathfrak{a}]$, and $T=N^2$ when $\rho=\frac{1}{2}$. Then we get that
\begin{equation}\label{eq:VarianceOCritical}
\Var[\Jo_T^N] \leq C_2 N \log^4(N)
\end{equation} for some constant $C_2>0$, and all $N$ large enough. }
\end{corollary}

Next, we study moderate deviations for the current of the open ASEP for  general effective densities $\rho_{\Lup} \neq \rho_{\Rup}$ close to the triple point. More precisely, we consider two open ASEPs $(\eta_t^1)_{t \geq 0}$ and $(\eta_t^2)_{t \geq 0}$ with currents $(\mathcal{J}_t^{(1)})_{t \geq 0}$ and $(\mathcal{J}_t^{(2)})_{t \geq 0}$ and invariant measures $\mu_N^{1},\mu_N^{2}$, respectively, where the associated boundary parameters satisfy
\begin{equation}\label{eq:ParameterOrdering}
\max\big(|\alpha^{2} - \alpha^{1}|, |\beta^{2} - \beta^{1}|, |\gamma^{2} - \gamma^{1}|, |\delta^{2} - \delta^{1}|\big) = \mathcal{O}(N^{-1/2}) ,
\end{equation}
and the invariant measure $\mu_N^{1}$ of $(\eta_t^1)_{t \geq 0}$ is a Bernoulli-$\frac{1}{2}$-product measure on $\{ 0,1 \}^{N}$.

\begin{lemma}\label{lem:CurrentOpenASEPgeneral}
Let $q$ from \eqref{def:TriplePointScaling} satisfy $\kappa<\frac{1}{2}$ and $\psi>0$, and assume \eqref{eq:ParameterOrdering}. Then
\begin{equation}\label{eq:VarianceOG}
\Var[\Jo^{(2)}_{N^{3/2}(1-q)^{-1}}] \leq C_1 N
\end{equation} holds for some constant $C_1>0$, and all $N$ large enough. For $\kappa=\frac{1}{2}$,
\begin{equation}\label{eq:VarianceOGCritical}
\Var[\Jo^{(2)}_{N^2}] \leq C_2 N \log^4(N)
\end{equation} holds for some constant $C_2>0$, and all $N$ large enough.
\end{lemma}
\begin{proof} We will only consider the case $\kappa<\frac{1}{2}$ as similar arguments apply for $\kappa=\frac{1}{2}$.
Consider the open ASEP $(\eta_t^2)_{t \geq 0}$ started from the Bernoulli-$\frac{1}{2}$-product measure $\mu_N^1$. Then using the same arguments as in Lemma~\ref{lem:SpeedOfPropagationCoupling} and Proposition \ref{pro:CurrentNtoO} in order to couple $(\eta_t^1)_{t \geq 0}$ and $(\eta_t^2)_{t \geq 0}$ under a common starting configuration (but different boundary parameters), there exist constants $c_0,C_0>0$ such that for all $1 \leq x\leq N^{\expont}$, the current of $(\eta_t^2)_{t \geq 0}$ satisfies
\begin{equation}\label{eq:CurrentOpenEst}
 \P\Big( \big| \Jo^{(2)}_{N^{3/2}(1-q)^{-1}} - \frac{1}{4}N^{3/2} \big| > x \sqrt{N} \, \Big| \, \eta_0^2 \sim \mu_N^1 \Big) \leq C_0 \exp( -  c_0 x^{1/2} ) 
\end{equation} for all $N$ large enough. Observe that due to assumption \eqref{eq:ParameterOrdering} and Lemma \ref{lem:Bernoulli}, there exists a coupling $\bar{\mathbf{P}}$ between $\tilde{\eta}_0^{1} \sim \mu_N^{1}$ and $\tilde{\eta}^{2}_0\sim \mu_N^{2}$ and constants $\tilde{c}_0,\tilde{C}_0>0$ such that for all $1 \leq y \leq N^{\expont}$ and $N$ large enough, by a standard moderate deviation estimate, 
\begin{equation}\label{eq:InitialDiff}
 \bar{\mathbf{P}}\Big( \Big| \sum_{v \in \lbr N \rbr} \tilde{\eta}_0^{2}(v) - \sum_{v \in \lbr N \rbr}\tilde{\eta}_0^{1}(v) \Big| \geq y \sqrt{N} \Big) \leq \exp(-c_0 y^2) .
\end{equation}
Thus, using the basic coupling for $(\eta_t^2)_{t \geq 0}$ started from $\tilde{\eta}_0^{1}$ and $\tilde{\eta}_0^{2}$, respectively, we see that \eqref{eq:CurrentOpenEst} holds true for $\eta_0^2 \sim \mu_N^2$. Using now the same arguments as in Corollary \ref{cor:ModDevitationCurrentN} to convert the moderate deviations estimates to a variance bound, we conclude.
\end{proof}

\begin{remark}\label{rem:ExtendedMaxCurrent} {We conjecture that Lemma~\ref{lem:CurrentOpenASEPgeneral} can be extended to the entire maximal current phase of the open ASEP by proving a similar bound as \eqref{eq:InitialDiff} for all boundary parameters in the maximal current phase (by using for example the results from \cite{BW:Density,BW:AskeyWilsonProcess}).}
\end{remark}

\subsection{Current bounds for the open ASEP}\label{sec:CurrentBoundsSecondMoment}

In this subsection, we record several consequences of the bounds on the stationary current and the moderate deviations for the current of the open ASEP. { For the upcoming results, it will be convenient to reparametrize the densities $\rho$, i.e., we set
\begin{equation}
    \rho_{n} := \frac{1}{2} + 2^{-n}
\end{equation} for all $n\in \N$. 
}
We start with the weakly high density phase and $\kappa<\frac{1}{2}$.

\begin{lemma}\label{lem:CurrentHighLowFinal}
Let $q$ satisfy \eqref{def:TriplePointScaling} for $\kappa<\frac{1}{2}$.
For all $N\in \N$, let $n\in \lbr \frac{1}{2}\log_2(N)-1 \rbr$ and
\begin{equation}
T = 2^{n{+1}} N (1-q)^{-1} .
\end{equation}
Recall $\expont$ from \eqref{def:KappaEpsilon}. Consider two stationary open ASEPs $(\eta^{1}_t)_{t \geq 0}$ and $(\eta^{2}_t)_{t \geq 0}$  such that the respective effective densities of the two processes satisfy
\begin{align*}
\rho_{\Lup}^{(1)}&=\rho^{(1)}_N \quad \text{ and  } \quad \rho_{\Rup}^{(1)}=\rho^{(1)}_N \quad \text{ where } \quad \rho^{(1)}_N=\frac{1}{2}+2^{-n} , \\
\rho_{\Lup}^{(2)}&=\rho_N^{(1)} \quad \text{ and  } \quad \rho_{\Rup}^{(2)}=\rho^{(2)}_N \quad \text{ where } \quad \rho^{(2)}_N=\frac{1}{2}+2^{-(n+1)} .
\end{align*} For any coupling $\bar{\mathbf{P}}$ of $(\eta^{1}_t)_{t \geq 0}$ and $(\eta^{2}_t)_{t \geq 0}$, the currents $(\mathcal{J}^{(1)}_t)_{t \geq 0}$ and $(\mathcal{J}^{(2)}_t)_{t \geq 0}$ satisfy
\begin{equation}\label{eq:LowerDeviation}
\bar{\mathbf{P}}\left( \mathcal{J}^{(2)}_{T} \geq \mathcal{J}^{(1)}_{T} +  2^{-(n+2)}N \right) \geq 1- C_0\exp\big(-c_0\min(2^{-n}N^{1/2},N^{\expont})^{1/2}\big)
\end{equation} for some constants $c_0,C_0>0$, depending only on $\kappa$, and for all $N$ large enough.
 \end{lemma}
\begin{proof}
We get from Lemma \ref{lem:CurrentHighLow} that the stationary currents satisfy
\begin{equation*}
  \E\big[  \mathcal{J}^{(2)}_T \big] - \E\big[\mathcal{J}^{(1)}_T \big]  \geq T (2^{-2n} - 2^{-2(n+1)})(1-q) \geq  2^{-(n+1)}N  .
\end{equation*}
Moreover, we have that
\begin{align*}
\bar{\mathbf{P}}\left( \mathcal{J}^{(2)}_{T} \geq \mathcal{J}^{(1)}_{T} +  2^{-(n+2)}N \right) \geq 1 &- \P\left(\mathcal{J}^{(2)}_{T} \leq  \E\big[  \mathcal{J}^{(2)}_T \big] - 2^{-(n+3)}N \right) \\
&- \P\left(\mathcal{J}^{(1)}_{T} \geq  \E\big[  \mathcal{J}^{(1)}_T \big] + 2^{-(n+3)}N \right) .
\end{align*} We conclude \eqref{eq:LowerDeviation} using the moderate deviations in Proposition~\ref{pro:CurrentNtoO} with $y=\min(N^{\expont},2^{-n}N^{1/2})$ to bound the fluctuations of the current of $(\eta^{1}_t)_{t \geq 0}$ and $(\eta^{2}_t)_{t \geq 0}$ at time $T$, respectively.
\end{proof}

We have the following result on the current in the maximal current phase of the open ASEP when $\kappa<\frac{1}{2}$, which is similar to Lemma~\ref{lem:CurrentHighLowFinal}.

\begin{lemma}\label{lem:CurrentMaxPhase}
Let $q$ satisfy \eqref{def:TriplePointScaling} for some $\kappa<\frac{1}{2}$ and $\psi>0$. Fix some $m \in \N$ and let
\begin{equation}
T = m N^{3/2} (1-q)^{-1} .
\end{equation}
Consider two stationary open ASEPs $(\eta^{1}_t)_{t \geq 0}$ and $(\eta^{2}_t)_{t \geq 0}$ such that $(\eta^{1}_t)_{t \geq 0}$ satisfies  assumptions \eqref{def:LiggettTypeCondition} and \eqref{eq:ScalingCondition} for some $(\alpha,\beta,\gamma,\delta,q)$. Let  $C_L,C_R>0$ be constants, depending only on $\kappa$ and $\psi>0$, and assume that $(\eta^{2}_t)_{t \geq 0}$ has the same boundary parameters $\gamma$ and $\delta$ as $(\eta^{1}_t)_{t \geq 0}$, and  $\alpha' \geq \alpha$ as well as $\beta' \geq \beta$ so that the respective effective densities satisfy
\begin{align*}
\rho_{\Lup}^{(2)}= \frac{1}{2}+ C_L\frac{1}{\sqrt{N}}  \quad \text{and} \quad
\rho_{\Rup}^{(2)}= \frac{1}{2}- C_R\frac{1}{\sqrt{N}}  .
\end{align*} Then we find constants $c_0,c_1>0$ such that under the basic coupling $\mathbf{P}$, the respective currents $(\mathcal{J}^{(1)}_t)_{t \geq 0}$ and $(\mathcal{J}^{(2)}_t)_{t \geq 0}$ of $(\eta^{1}_t)_{t \geq 0}$ and $(\eta^{2}_t)_{t \geq 0}$ satisfy
\begin{equation}\label{eq:LowerDeviationMax}
\mathbf{P}\left( \mathcal{J}^{(2)}_{T} \geq \mathcal{J}^{(1)}_{T} + m c_0 \sqrt{N} \right) \geq  \frac{c_1}{m^2}
\end{equation} for all $m \in \N$, and all $N$ large enough.
\end{lemma}
\begin{proof}
From Proposition \ref{pro:CurrentMaxCurrent1}, verifying that $B(\beta,\delta,q)$ and $D(\alpha,\gamma,q)$ as well as $B(\beta^{\prime},\delta,q)$ and $D(\alpha^{\prime},\gamma,q)$ have the desired form, and Lemma~\ref{lemMonotonicity} for strict monotonicity, we get that
there exist some constants $c_2,C_2>0$ such that for any $m \in \N$
\begin{equation*}
m c_2 \sqrt{N} \leq  \E\big[  \mathcal{J}^{(2)}_{T} \big] -  \E\big[\mathcal{J}^{(1)}_{T} \big]  \leq   m C_2  \sqrt{N}
\end{equation*} and all $N$ large enough. Note that under the basic coupling, we can ensure that
\begin{equation*}
\mathbf{P}\left( \mathcal{J}^{(2)}_{T} \geq \mathcal{J}^{(1)}_{T} \right) = 1 .
\end{equation*}
A standard fact for (not necessarily independent) random variables $(X_i)_{i \in \lbr m \rbr}$ and $(Y_i)_{i \in \lbr m \rbr}$ is that by Cauchy--Schwarz
\begin{equation}
\Var\Big( \sum_{i \in \lbr m \rbr} (X_i-Y_i) \Big) \leq m^{4} \max_{i \in \lbr m \rbr} ( \max(\Var(X_i),\Var(Y_i)))  .
\end{equation}
Hence, writing the current by time increments, we get that for all $N$ large enough \begin{equation*}
\Var(\mathcal{J}^{(2)}_{T}- \mathcal{J}^{(1)}_{T}) \leq m^4 \max_{i\leq \lbr m \rbr , j\in \{1,2\} }\Var\Big(\mathcal{J}^{(j)}_{iT/m} - \mathcal{J}^{(j)}_{(i-1)T/m} \Big) \leq  m^{4} C_1 N
\end{equation*}
with some  $C_1>0$, where we use Corollary~\ref{cor:ModDevitationCurrentO} for the last inequality. Using the Paley--Zygmund inequality, noting that $\mathcal{J}^{(2)}_{T} \geq \mathcal{J}^{(1)}_{T}$ almost surely under the basic coupling $\mathbf{P}$, 
\begin{equation*}
\mathbf{P}\left( \mathcal{J}^{(2)}_{T} \geq \mathcal{J}^{(1)}_{T} + \frac{m c_2}{2}  \sqrt{N}\right) \geq \frac{\E\big[  \mathcal{J}^{(2)}_{T} -\mathcal{J}^{(1)}_{T} \big]^2}{4\Var\big(\mathcal{J}^{(2)}_{T} - \mathcal{J}^{(1)}_{T}\big)+4\E\big[\mathcal{J}^{(2)}_{T} -\mathcal{J}^{(1)}_{T}\big]^2}\geq c_3 m^{-2}
\end{equation*} for some constant $c_3>0$, which finishes the proof.
\end{proof}

In the case where $q$ satisfies \eqref{def:TriplePointScaling} with $\kappa=\frac{1}{2}$, we have the following result on the current of the open ASEP in the high density phase.

\begin{lemma}\label{lem:CurrentCriticalKappaHigh}
Let $q$ satisfy \eqref{def:TriplePointScaling} for $\kappa=\frac{1}{2}$ and $\psi>0$. For all $N\in \N$, and $n\in \N$ set
\begin{equation}
T = 2^{n+1} N (1-q)^{-1} .
\end{equation}
Consider two stationary open ASEPs $(\eta^{1}_t)_{t \geq 0}$ and $(\eta^{2}_t)_{t \geq 0}$  such that the respective effective densities which satisfy
\begin{align*}
\rho_{\Lup}^{(1)}&=\rho^{(1)}_N \quad \text{ and  } \quad \rho_{\Rup}^{(1)}=\rho^{(1)}_N \quad \text{ where } \quad \rho^{(1)}_N=\frac{1}{2}+2^{-n} , \\
\rho_{\Lup}^{(2)}&=\rho_N^{(1)} \quad \text{ and  } \quad \rho_{\Rup}^{(2)}=\rho^{(2)}_N \quad \text{ where } \quad \rho^{(2)}_N=\frac{1}{2}+2^{-(n+1)} .
\end{align*}
Then there exists a constant $c_2>0$ such that for all $n$ with
\begin{equation}\label{eq:nChoice}
n \in \big[ \log_2(c_1^{-1}\log(N)), c_2 \log_2(\sqrt{N}\log^{-1}(N))\big] , 
\end{equation}
 where the constant $c_1>0$ is taken from Lemma \ref{lem:ModDeviationMaxSecondClass}, and for all $N$ large enough,
\begin{equation}
\label{eq:LowerDeviationSpecial}
\mathbf{P}\left( \mathcal{J}^{(2)}_{T} \geq \mathcal{J}^{(1)}_{T} + 2^{-(n+2)}N \log(N) \right) \geq 1- N^{-9} .
\end{equation}
\end{lemma}
\begin{proof}
As for Lemma \ref{lem:CurrentHighLowFinal}, we get from Lemma \ref{lem:CurrentHighLow} that the expected currents satisfy
\begin{equation*}
  \E\big[  \mathcal{J}^{(2)}_T \big] - \E\big[\mathcal{J}^{(1)}_T \big]  \geq T (2^{-2n} - 2^{-2(n+1)})(1-q) \geq  2^{-(n+1)}N \geq 4C_1 N^{1/2} \log(N) ,
\end{equation*} where $C_1>0$ is taken from Proposition~\ref{pro:CurrentNtoO}. Note that the current $\mathcal{J}^{(2)}_T$ is stochastically dominated by the current of a stationary open ASEP in a Bernoulli-$\rho_N^{(2)}$-product measure. The result now follows from \eqref{eq:CurrentOpenCritical} in Proposition~\ref{pro:CurrentNtoO}.
\end{proof}

Similarly, we can estimate the current in the maximal current phase when $\kappa=\frac{1}{2}$.

\begin{lemma}\label{lem:CurrentCriticalKappaMax}
Let $q$ satisfy \eqref{def:TriplePointScaling} for $\kappa=\frac{1}{2}$ and $\psi>0$. For all $N\in \N$, consider two stationary open ASEPs $(\eta^{1}_t)_{t \geq 0}$ and $(\eta^{2}_t)_{t \geq 0}$  with effective densities
\begin{align*}
\rho_{\Lup}^{(1)}&= \frac{1}{2} + \frac{C_L}{\sqrt{N}} \quad \text{ and  } \quad \rho_{\Rup}^{(1)}= \frac{1}{2} + \frac{C_R}{\sqrt{N}} ,   \\
\rho_{\Lup}^{(2)}&= \frac{1}{2} + \frac{C^{\prime}_L}{\sqrt{N}} \quad \text{ and  } \quad \rho_{\Rup}^{(2)}= \frac{1}{2} + \frac{C^{\prime}_R}{\sqrt{N}}
\end{align*} such that  \eqref{def:LiggettTypeCondition} and \eqref{eq:ScalingCondition} holds, as well as the condition $\tilde{B},\tilde{D}>-\psi$ in Proposition \ref{pro:CurrentMaxCurrent3} is satisfied for both sets of boundary parameters.
Let $T= mN^{2}\log(N)$ for some $m\in \N$.
For $C_L,C_R \in \R$ and $C_0>0$, we can choose $C_L^{\prime},C_R^{\prime}>0$ above such that
\begin{equation}
\label{eq:LowerDeviationMaxSpecial}
\mathbf{P}\left( \mathcal{J}^{(2)}_{T} \geq \mathcal{J}^{(1)}_{T} + c_0 m \sqrt{N}\log(N) \right) \geq  \frac{c_1}{\log^{2}(N)m^2}
\end{equation}
with some constants $c_0,c_1>0$, all $m\in \N$ fixed, and all $N$ large enough.
\end{lemma}
\begin{proof}
By Proposition \ref{pro:CurrentMaxCurrent3} and Lemma \ref{lemMonotonicityBis}, there exist $c_1,C_1>0$ such that
\begin{equation*}
c_1 m \sqrt{N}\log(N) \leq  \E\big[  \mathcal{J}^{(2)}_{T} \big] -  \E\big[\mathcal{J}^{(1)}_{T} \big]  \leq  C_1 m \sqrt{N}\log(N) 
\end{equation*} for all $N \in \N$ large enough. Here, we choose  $C_L^{\prime},C_R^{\prime}$ large enough by increasing $\alpha$ and $\beta$ while decreasing $\gamma$ and $\delta$ to meet the assumptions of Proposition \ref{pro:CurrentMaxCurrent3}. By Lemma \ref{lem:CurrentOpenASEPgeneral},
\begin{equation*}
\Var(\mathcal{J}^{(1)}_{T}-\mathcal{J}^{(2)}_{T}) \leq 4 m^4\max_{i\leq \lbr m \rbr , j\in \{1,2\} }\Var\Big(\mathcal{J}^{(j)}_{iT/m} - \mathcal{J}^{(j)}_{(i-1)T/m} \Big) \leq C_2 m^4 N\log^{4}(N)
\end{equation*} for some constant $C_2>0$ and all $N$ large enough.
As in Lemma \ref{lem:CurrentMaxPhase}, we apply the Paley--Zygmund inequality to conclude.
\end{proof}

\section{Mixing times in the weakly high and low density phase}\label{sec:MixingHighLow}

In this section, we prove estimates on the mixing times for the {weakly high and weakly low density phase, i.e. the parameters $A$ and $C$ from \eqref{def:a} and \eqref{def:c} satisfy $A>\max(C,1)$ (respectively $C>\max(A,1)$) with $A,C\rightarrow 1$ as $N \rightarrow \infty$.} We provide an iterative scheme to bound the number of second class particles in the segment over time. This follows ideas from Section~7 of \cite{GNS:MixingOpen} for mixing times in the high and low density phase with constant boundary and bias parameters.

\subsection{Iterative bounds on the mixing time}\label{sec.IterativeMixingTimes}

Throughout this section, we assume that $A>\max(1,C)$ as well as that $q$ satisfies \eqref{def:TriplePointScaling} for some $\kappa \in [0,\frac{1}{2}]$. We have the following result on the mixing time of the open ASEP in the weakly high density phase.

\begin{proposition}\label{pro:MixingTimesWeaklyHighLow}
Suppose that $\kappa<\frac{1}{2}$ and assume that $\mu_N=\textup{Ber}_N(\rho_n)$ such that
\begin{equation}
\rho_n = \frac{1}{2} + 2^{-n}
\end{equation} for some $n\in \lbr \frac{1}{2}\log_2(N)\rbr$.
There exist constants $C_0,c_1,C_1>0$ such that the $\varepsilon$-mixing time of the open ASEP satisfies for all $N$ large enough
\begin{equation}
 \frac{t^{N}_{\mix}(\varepsilon)}{(1-q)^{-1} 2^{n} N} \leq C_0
\end{equation}
for all $\varepsilon \in (0,1)$ with $\varepsilon \geq C_1\exp(-c_1\min(2^{-n}N^{1/2},N^{\expont})^{1/2})$, where we take $\expont$ from \eqref{def:KappaEpsilon}.
\end{proposition}

In order to show Proposition \ref{pro:MixingTimesWeaklyHighLow}, we will prove a recursion on the \textbf{coupling time} $t_{\cou}=t_{\cou}^{N,n}(\varepsilon)$, where we set for all $\varepsilon\in (0,1)$
\begin{equation}\label{def:MixingTimeCoupling}
t_{\cou}^{N,n}(\varepsilon) := \inf\Big\{ t \geq 0 \colon \mathbf{P}\big(\tau^{N,n}_{\textup{coal}} \leq t \big) \geq 1-\varepsilon \Big\}
\end{equation}
as the first time $t$ such that the probability under the basic coupling that the open ASEP from any pair of initial states has coalesced by time $t$ is larger than $1-\varepsilon$. Here, $\tau^{N,n}_{\textup{coal}}$ denotes the coalescence time of two open ASEPs under the basic coupling $\mathbf{P}$ with the same parameters and stationary distribution $\mu_N=\textup{Ber}_N(\rho_n)$, started from the extremal configurations $\one$ and $\zero$, respectively.
Note that $t^{N,n}_{\cou}(\varepsilon) \geq t^{N}_{\mix}(\varepsilon)$ using the coupling representation of the total variation distance -- see Lemma 2.2 in \cite{GNS:MixingOpen} -- and that $t_{\cou}^{N,n}(\varepsilon)$ is decreasing in $\varepsilon$.

\begin{lemma}\label{lem:IterateMixingTimesHigh}
Suppose that $\kappa<\frac{1}{2}$ and let $\mu_N=\textup{Ber}_N(\rho_n)$ be such that
\begin{equation}
\rho_n = \frac{1}{2} + 2^{-n}
\end{equation} for some $n\in \lbr \frac{1}{2}\log_2(N)\rbr$.
Then there exist constants $c_0,C_0>0$ such that the coupling  time of the open ASEP satisfies for any $\varepsilon \in (0,1)$ 
\begin{equation}\label{eq:MixingHighIteration}
t_{\cou}^{N,n}\left(2\varepsilon + C_0\exp(- c_0\min(2^{-n}N^{1/2},N^{\expont})^{1/2}\right) \leq t_{\cou}^{N,n-1}\big(\varepsilon\big) + 4(1-q)^{-1}2^{n}N
\end{equation} for all $N$ large enough, Moreover, there exist  constants $C_1,C_2,c_2>0$ such that 
\begin{equation}\label{eq:MixingBaseCase}
t_{\cou}^{N,1}(\varepsilon) \leq C_1 (1-q)^{-1} N
\end{equation} for all $\varepsilon \geq C_2\exp(-c_2N^{\expont/2})$, and all $N$ large enough.
\end{lemma}

\begin{remark}\label{rem:MixingTimesWeaklyLow}
Using the particle-hole duality, the same result holds for the coupling time and the mixing time of the open ASEP in the low density phase,  where the densities $\rho_n$ take the form $\rho_n=\frac{1}{2} - 2^{-n}$ with some $n\in \lbr \frac{1}{2}\log_2(N)\rbr$.
\end{remark}

Using the recursion on the coupling time for the open ASEP in the high density phase, we can deduce Proposition \ref{pro:MixingTimesWeaklyHighLow}.

\begin{proof}[Proof of Proposition \ref{pro:MixingTimesWeaklyHighLow} using Lemma \ref{lem:IterateMixingTimesHigh}]

We define the function $f$ by 
\begin{equation}
   f(n):=C_0\exp\big(- c_0\min(2^{-n}N^{1/2},N^{\expont})^{1/2}\big) 
\end{equation}
 for all $n\in \N$. Then we can rewrite \eqref{eq:MixingHighIteration} as
\begin{equation}\label{eq:ItEq}
t_{\cou}^{N,n}(\varepsilon)\leq t_{\cou}^{N,n-1}\big(2^{-1}\varepsilon- 2^{-1}f(n)\big) + 4(1-q)^{-1}2^{n}N .
\end{equation}
Iterating \eqref{eq:ItEq}, we get that
\begin{equation}
t_{\cou}^{N,n}(\varepsilon)\leq t_{\cou}^{N,1}\big(2^{-n}\varepsilon- F(n)\big) + 8(1-q)^{-1}2^{n}N,
\end{equation}
where we set $F(n):=\sum_{k=1}^n 2^{-k}f(k)$.
Choosing now constants $c_3,C_3>0$ accordingly, depending on the choice of $c_2,C_2>0$ in Lemma \ref{lem:IterateMixingTimesHigh}, a simple computation shows that for any $n \in \lbr \frac{1}{2} \log_2(N) \rbr$, any $\varepsilon \geq C_3 \exp(-c_3 \min(2^{-n}N^{1/2},N^{\expont})^{1/2})$, and all $N$ large enough,
\begin{equation}
2^{-n}\varepsilon- F(n)\geq C_2 \exp\big(-c_2 N^{\expont/2}\big) .
\end{equation}
Using \eqref{eq:MixingBaseCase} and the fact that $t_{\cou}^{N,n}(\varepsilon) \geq t_{\mix}^{N,n}(\varepsilon)$ for all $\varepsilon \in (0,1)$, we conclude.
\end{proof}

Before we give the proof of Lemma \ref{lem:IterateMixingTimesHigh},
we state a corresponding result for the case where $q$ from \eqref{def:TriplePointScaling} satisfies $\kappa=\frac{1}{2}$. Again, note that  the results stated for the high density phase also apply in the low density phase using the particle hole symmetry.

\begin{proposition}\label{pro:MixingTimesWeaklyHighLowCritical}
Let $q$ from \eqref{def:TriplePointScaling} satisfy $\kappa=\frac{1}{2}$.  Let $\mu_N=\textup{Ber}_N(\rho_n)$ be such that we have $
\rho_n = \frac{1}{2} + 2^{-n}$ for some $n$ with
\begin{equation}\label{eq:nassumption}
n \in \big\lbr \log_2(c_1^{-1}\log(N)), \log_2(c_2 \sqrt{N}\log^{-1}(N))\big\rbr
\end{equation} and some constant $c_2>0$, where the constant $c_1>0$ is taken from Lemma \ref{lem:ModDeviationMaxSecondClass}.  Then there exists a constant $C_0>0$ such that the mixing time of the open ASEP satisfies for all $\varepsilon \geq N^{-8}$
\begin{equation}
\limsup_{N \rightarrow \infty} \frac{t^{N}_{\mix}(\varepsilon)}{\max(2^{n},n\log(N)) N^{3/2} } \leq C_0 .
\end{equation}
\end{proposition}

Similar to Lemma~\ref{lem:IterateMixingTimesHigh}, we have the following recursion on the coupling time. 

\begin{lemma}\label{lem:IterateMixingTimesHighCritical}
Let $q$ from \eqref{def:TriplePointScaling} satisfy $\kappa=\frac{1}{2}$.  Let $\mu_N=\textup{Ber}(\rho_n)$ for $
\rho_n = \frac{1}{2} + 2^{-n}
$ be such that $n-1$ and $n$ satisfies \eqref{eq:nassumption}.
Then there exists a constant $C_0>0$ such that the coupling time of the open ASEP satisfies
\begin{equation}\label{eq:MixingHighIterationCritical}
t^{N,n}_{\cou}\left(2\varepsilon + C_0 N^{-9}\right) \leq t^{N,n-1}_{\cou}(\varepsilon) + (1-q)^{-1}2^{n+6}N
\end{equation} for all $N$ large enough. Moreover, there exists a constant $C_1>0$ such that for $n = \lceil \log_2(c_1^{-1}\log(N)) \rceil $, we get that
\begin{equation}\label{eq:MixingBaseCaseCritical}
t^{N,n}_{\cou}(\varepsilon) \leq C_1 N^{3/2} \log(N)
\end{equation} for all $\varepsilon \geq 2N^{-9}$, and all $N$ large enough.
\end{lemma}

\begin{proof}[Proof of Proposition~\ref{pro:MixingTimesWeaklyHighLowCritical} using Lemma~\ref{lem:IterateMixingTimesHighCritical}]
    This is immediate from the same recursion as in the case $\kappa<\frac{1}{2}$.
\end{proof}

\subsection{A multi-species open ASEP with partial ordering}\label{sec:MultiSpeciesOpenASEP}

We deduce Lemma \ref{lem:IterateMixingTimesHigh} and Lemma  \ref{lem:IterateMixingTimesHighCritical}
using similar a setup as for the mixing time of the open ASEP in the high and low density phase in \cite{GNS:MixingOpen}.
To this end, we define the partially ordered multi-species open ASEP and a corresponding diminished exclusion process. The arguments are  similar to \cite{GNS:MixingOpen},  except that we require refined estimates on the current established in Section~\ref{sec:CurrentBoundsSecondMoment} -- see also Remark~7.4 in \cite{GNS:MixingOpen} -- and additional classes of particles. Let us stress that this multi-species open ASEP will be different from the multi-species ASEP on $\Z$ defined in Section~\ref{sec:ASEPs} as the different particle types only satisfy a partial order. \\

In order to define this multi-species ASEP, we couple four exclusion processes on $\lbr N \rbr$ according to the basic coupling. Adapting the notation from Section 7 in \cite{GNS:MixingOpen}, for $j \in \lbr 4 \rbr$, we define the open ASEPs $(\eta^{j}_t)_{t \geq 0}$. The processes $(\eta^{1}_t)_{t \geq 0}$, $(\eta^{2}_t)_{t \geq 0}$ and $(\eta^{3}_t)_{t \geq 0}$ are defined with respect to parameters $(q,\alpha,\beta,\gamma,\delta)$
and started from (random) configurations $\eta^{(i)}$ such that under the basic coupling
\begin{equation}\label{eq:OrderedInitialConditions}
\mathbf{P}\Big( \eta^{(1)}_t \succeq_{\textup{c}} \eta^{(3)}_t \succeq_{\textup{c}} \eta^{(2)}_t \, \forall t \geq 0 \Big) = 1 ,
\end{equation}  and such that $(\eta^{3}_t)_{t \geq 0}$ is stationary.
We define the fourth open ASEP $(\eta^{4}_t)_{t \geq 0}$  with respect to the same $q$ as the other three processes, but parameters $(\alpha^{\prime},\beta^{\prime},\gamma^{\prime},\delta^{\prime})$ for some $\alpha^{\prime} \geq \alpha$ and $\beta^{\prime} \geq \beta$, as well as $\gamma \geq \gamma^{\prime}$ and $\delta \geq \delta^{\prime}$ specified later on, and with stationary initial data. In the following, let $(\xi_t)_{t \geq 0}$ denote the disagreement process between $(\eta^{1}_t)_{t \geq 0}$ and $(\eta^{2}_t)_{t \geq 0}$ and let $(\zeta_t)_{t \geq 0}$ be the disagreement process between $(\eta^{3}_t)_{t \geq 0}$ and $(\eta^{4}_t)_{t \geq 0}$. Recall from Remark~\ref{rem:Disagreement} that we can interpret $(\xi_t)_{t \geq 0}$ and $(\zeta_t)_{t \geq 0}$ as Markov processes on $\{0,1,2\}^N$ (due to \eqref{eq:OrderedInitialConditions}) and $\{\zero,\one,\Aup,\Bup,\onep,\zerop\}^N$, respectively. Here, we use  the basic coupling for $(\xi_t)_{t \geq 0}$ and the extended disagreement process from Definition \ref{def:BasicCouplingExtended} for $(\zeta_t)_{t \geq 0}$ with respect to the same Poisson clocks. Since $(\eta^{3}_t)_{t \geq 0}$ and $(\eta^{4}_t)_{t \geq 0}$ are stationary, we can assume without loss of generality that the process $(\zeta_t)_{t \geq 0}$ is stationary as well.
Our key observation is that $\zeta_t(x) =\zero$ implies that $\xi_t(x) \in \{0,2\}$, due to assumption \eqref{eq:OrderedInitialConditions}. Similarly, $\zeta_t(x) = \one$ implies that $\xi_t(x) \in \{1,2\}$ for any $t \geq 0$ and $x\in \lbr N \rbr$. In particular, $(\zeta_t,\xi_t)_{t \geq 0}$ does not attain $(\zero,1)$ and $(\one,0)$.

\begin{definition}[Partially ordered multi-species open ASEP]\label{def:PartiallyOrdered}
Consider the combined disagreement process $(\zeta_t,\xi_t)_{t \geq 0}$, which is a Markov process on $(\{\zero,\one,\Aup,\Bup,\onep,\zerop\} \times \{0,1,2\})^{N}$. We denote by $(\chi_t)_{t \geq 0}$ a partially ordered multi-species exclusion process on
\begin{equation}
\tilde{\Omega}_{N} := \{ 1, 2_{-1},2_0,2_1,2_2, 2_3,2_4, 2_5 , 0 \}^N ,
\end{equation}
and refer to $2_i$ as a \textbf{second class particle of type} $\mathbf{i}$. Let $\chi_0(x)$ for all $x\in \lbr N \rbr$ be given by
\begin{equation}
\chi_0(x) := \begin{cases} 0 & \text{ if } \xi_0(x)=0 \text{ and }\zeta_0(x)=\zero,  \\
2_0 & \text{ if }\  \xi_0(x)=0 \text{ and }  \zeta_0(x) \in \{\Aup,\Bup,\onep,\zerop \} , \\
2_1 & \text{ if }\ \xi_0(x)=2 \text{ and } \zeta_0(x)=\zero, \\
2_2 & \text{ if }\ \xi_0(x)=2 \text{ and } \zeta_0(x) \in \{\Aup,\Bup,\onep,\zerop \} ,  \\
2_3 & \text{ if }\ \xi_0(x)=2 \text{ and } \zeta_0(x)=\one , \\
2_4 & \text{ if }\ \xi_0(x)=1 \text{ and } \zeta_0(x) \in \{\Aup,\Bup,\onep,\zerop \}  , \\
1 & \text{ if }\ \xi_0(x)=1 \text{ and }\zeta_0(x)=\one .
\end{cases}
\end{equation}
 We obtain the process $(\chi_t)_{t \geq 0}$ by following the updates of $(\zeta_t,\xi_t)_{t \geq 0}$ as a multi-species exclusion process under the basic coupling, i.e., along each edge we sort the states at rate $1$ in increasing order, and at rate $q$ in decreasing order.
However, we have the following exceptions: Whenever under the basic coupling we set $\zeta_t(N)=\Bup$ and $\xi_t(N)=1$ at some time $t$, we place a second class particle of type $4$ in $(\chi_t)_{t \geq 0}$ at site $N$, i.e., $\chi_t(N)=2_4$. Similarly, whenever we set $\zeta_t(1)=\Aup$ and $\xi_t(N)=0$ at some time $t$, we place a second class particle of type $0$ in $(\chi_t)_{t \geq 0}$ at site $1$, i.e., $\chi_t(N)=2_0$. Moreover, when a type $3$ and a type $4$ second class particle receive an update, we turn the type $3$ particle into type $2$, and the type $4$ particle into type $5$. Similarly, when a type $0$ and a type $1$ second class particle are updated, we turn the type $0$ particle into type $-1$, and the type $1$ particle into type $2$.
\end{definition}
While Definition \ref{def:PartiallyOrdered} may seem daunting at first glance, it is a key tool to relate the exit time of second class particles to a stationary system.
Since $(\zeta_t,\xi_t)_{t \geq 0}$ can not attain the values $(\one,0)$ and $(\zero,1)$ by construction, $(\zeta_t,\xi_t)_{t \geq 0}$ and $(\chi_t)_{t \geq 0}$ are in one-to-one correspondence after projecting type $2_5$  to first class particles, type $2_{-1}$ to empty sites, and types $\Aup,\Bup,\onep,\zerop$ to second class particles.
Observe that the second class particles of types $-1$ to $5$ obey the partial ordering indicated in Figure \ref{fig:SecondClassHierachy}.
\begin{figure}
\centering
\begin{tikzpicture}[scale=0.8]
\draw[rounded corners] (-4, 0) rectangle (-3, 2);

\draw[rounded corners] (0, -1.8) rectangle (1, 0.2);

\node (X) at (2,-1.3) {$(\zero,2)\widehat{=} 2_1$};

\node (X) at (2,3.3) {$(\two,0)\widehat{=} 2_{0}$};

\draw[rounded corners] (0, 1.8) rectangle (1, 3.8);


\node (X) at (4.5-8,-0.5) {$(\zero,0)\widehat{=} 2_{-1}$};

\draw[rounded corners] (4, 0) rectangle (5, 2);

\node (X) at (4.5,-0.5) {$(\two,2)\widehat{=} 2_2$};

\draw[rounded corners] (8, -1.8) rectangle (9, 0.2);

\node (X) at (10,-1.3) {$(\one,2)\widehat{=} 2_3$};

\draw[rounded corners] (8, 1.8) rectangle (9, 3.8);

\node (X) at (10,3.3) {$(\two,1)\widehat{=} 2_{4}$};

\draw[rounded corners] (12, 0) rectangle (13, 2);

\node (X) at (12.5,-0.5) {$(\one,1)\widehat{=} 2_5$};

\draw[thick,->] (8-8, -0.8) -> (5-8, 0.5);
\draw[thick,->] (8-8, 2.8) -> (5-8, 1.5);
\draw[thick,->] (12-8, 1.5) -> (9-8, 2.8);
\draw[thick,->] (12-8, 0.5) -> (9-8,-0.8);

\draw[thick,dashed] (8.5-8,1) -- (8.5-8,1.8);
\draw[thick,dashed] (8.5-8,1) -- (8.5-8,0.2);
\draw[thick, dashed, ->] (8.5-8,1) -- (5-8, 1);
\draw[thick, dashed, ->] (8.5-8,1) -- (12-8, 1);

\draw[thick,->] (8, -0.8) -> (5, 0.5);
\draw[thick,->] (8, 2.8) -> (5, 1.5);
\draw[thick,->] (12, 1.5) -> (9, 2.8);
\draw[thick,->] (12, 0.5) -> (9,-0.8);

\draw[thick,dashed] (8.5,1) -- (8.5,1.8);
\draw[thick,dashed] (8.5,1) -- (8.5,0.2);
\draw[thick, dashed, ->] (8.5,1) -- (5, 1);
\draw[thick, dashed, ->] (8.5,1) -- (12, 1);

\node[shape=circle,scale=1.5,draw] (E) at (0.5-4,0.5){} ;

\node[shape=circle,scale=1.5,draw] (E) at (0.5-4,1.5){} ;

\node[shape=circle,scale=1.5,draw] (E) at (0.5,2.3){} ;
\node[shape=circle,scale=1.5,draw] (E) at (0.5,3.3){} ;
\node[shape=star,star points=5,star point ratio=2.5,fill=red,scale=0.55] (Y1) at (0.5,3.3) {};

\node[shape=circle,scale=1.5,draw] (E) at (0.5,-1.3){} ;
\node[shape=star,star points=5,star point ratio=2.5,fill=red,scale=0.55] (Y1) at (0.5,-1.3) {};
\node[shape=circle,scale=1.5,draw] (E) at (0.5,-0.3){} ;

\node[shape=circle,scale=1.5,draw] (E) at (4.5,1.5){} ;
\node[shape=star,star points=5,star point ratio=2.5,fill=red,scale=0.55] (Y1) at (4.5,1.5) {};

\node[shape=circle,scale=1.5,draw] (E) at (4.5,0.5){} ;
\node[shape=star,star points=5,star point ratio=2.5,fill=red,scale=0.55] (Y1) at (4.5,0.5) {};

\node[shape=circle,scale=1.5,draw] (E) at (8.5,-1.3){} ;
\node[shape=star,star points=5,star point ratio=2.5,fill=red,scale=0.55] (Y1) at (8.5,-1.3) {};

\node[shape=circle,scale=1.5,draw] (E) at (8.5,-0.3){} ;
\node[shape=circle,scale=1.2,fill=red] (Y1) at (8.5,-0.3) {};

\node[shape=circle,scale=1.5,draw] (E) at (8.5,2.3){} ;
\node[shape=circle,scale=1.2,fill=red] (Y1) at (8.5,2.3) {};

\node[shape=circle,scale=1.5,draw] (E) at (8.5,3.3){} ;
\node[shape=star,star points=5,star point ratio=2.5,fill=red,scale=0.55] (Y1) at (8.5,3.3) {};

\node[shape=circle,scale=1.5,draw] (E) at (12.5,1.5){} ;
\node[shape=circle,scale=1.2,fill=red] (Y1) at (12.5,1.5) {};

\node[shape=circle,scale=1.5,draw] (E) at (12.5,0.5){} ;
\node[shape=circle,scale=1.2,fill=red] (Y1) at (12.5,0.5) {};
	\end{tikzpicture}	
\caption{\label{fig:SecondClassHierachy}
Illustration of the possible pairs for $(\zeta_t,\xi_t)_{t \geq 0}$  in Definition \ref{def:PartiallyOrdered} and as well as their corresponding values in $(\chi_t)_{t \geq 0}$. The partial ordering is indicated by the directions of the (solid) arrows. The vertical dashed line indicates which pairs are not comparable under the partial ordering, and the horizontal dashed arrows indicate the outcome when an incomparable pair receives an update. Here, $\two$ for $\zeta_t$ stands for any value in the set $\{\Aup,\onep,\Bup,\zerop\}$.}
 \end{figure}
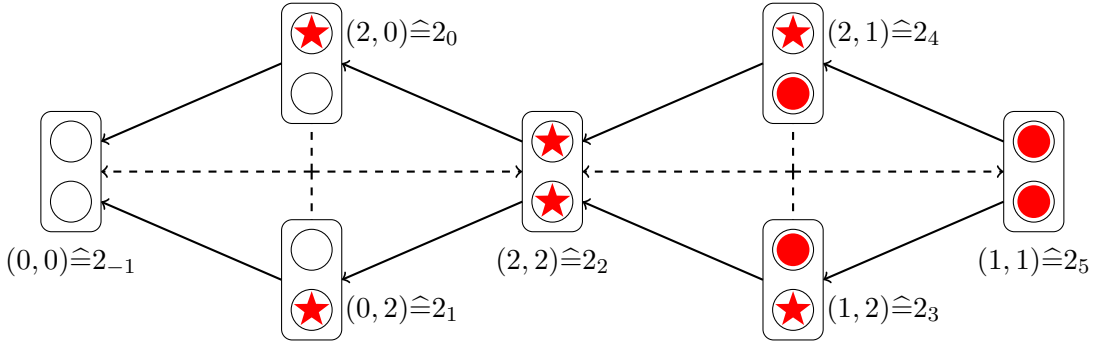
Note that all second class particles which enter at site $N$ must have type $4$ and all second class particles which enter at sites $1$ must have type~$0$. \\
We will now follow the construction of a diminished process $(\chi^{\star}_t)_{t \geq 0}$ as in the proof of Lemma~7.2 of \cite{GNS:MixingOpen}, which uses the partially ordered multi-species open ASEP $(\chi_t)_{t \geq 0}$ (however, in contrast to \cite{GNS:MixingOpen}, we will allow in the partially ordered multi-species open ASEP $(\chi_t)_{t \geq 0}$ also for type $-1$ and type $0$ particles).  In this process $(\chi^{\star}_t)_{t \geq 0}$, we delete all sites which are either empty,  occupied by a first class particle or of type $0$ or $-1$ in $\chi_t$, merging edges if necessary. Then replace all type $1,2,3$ second class particles by empty sites, and all type $4$ and type $5$ second class particles by first class particles. Finally, we extended the configuration to an element of $\{0,1\}^{\Z}$. We will now formalize this construction.
\begin{definition}[Diminished partially ordered multi-species ASEP]\label{def:PartiallyDiminished} Given the process $(\chi_t)_{t \geq 0}$, we start by constructing a family of vectors $(v_t)_{t \geq 0}$, where $v_t \in  \{ 0,1\}^k$ with some $k=k(t)\in \N \cup \{ 0\}$. For all $t\geq 0$, let $v_t$ denote the vector of type $1$ to $5$ second class particles which have left the segment at the site $1$ by time $t$. More precisely, we place
\begin{itemize}
\item  $1$ at position $k+1-i$  if the $i^{\text{th}}$ second class particle which exited is of type $4$ or $5$,
\item  $0$ at position $k+1-i$ if the $i^{\text{th}}$ second class particle which exited is of type $1,2,3$.
\end{itemize}
Let us stress that we only consider the type $1$ to $5$ second class particles exiting for $(v_t)_{t \geq 0}$.
For all $t\geq 0$, we assign a configuration $\chi^{\star}_t=\chi^{\star}_t(v_t) \in \{0,1\}^\Z$  as follows.
\begin{itemize}
\item Delete all vertices in $\chi_t$ which are empty, contain a first class particle or a type $-1$ or type $0$ second class particle to obtain a configuration $\chi^{\prime}_t$.
\item Concatenate the vector $v_t$ at the left-hand side of the diminished configuration $\chi^{\prime}_t$.
\item Turn all second class particles in $(v_t,\chi^{\prime}_t)$ to empty sites if they are of type $1,2$ or $3$ and turn them into first class particles if they are of type $4$ or $5$ to get an element~$\chi^{\prime\prime}_t$.
\item Extend $\chi^{\prime\prime}_t$ to a configuration $\chi^{\star}_t \in \{0,1\}^\Z$ by adding empty sites at the left-hand side and first class particles at the right-hand side of the segment.
\end{itemize}
Note that so far, $\chi^{\star}_t$ is only defined up to translations on $\Z$. We define the process $(\chi^{\star}_t)_{t \geq 0}$ from $(\chi_t)_{t \geq 0}$ such that $\chi^{\star}_0 \in \mathfrak{A}_0$, for $\mathfrak{A}_0$ defined in \eqref{def:BlockingStateSpace}. For $t>0$, suppose that $\chi^{\star}_{t_-} \in \mathfrak{A}_m$ holds for some $m\in\Z$. If at time $t$ a second class particle of type $1,2$ or $3$ exits at the right-hand side boundary in $\chi_t$, we choose the updated configuration such that $\chi^{\star}_{t} \in \mathfrak{A}_{m-1}$ holds. In all other cases, we choose $\chi^{\star}_{t} \in \mathfrak{A}_{m}$.
\end{definition}

A visualization of this construction (without type $-1$ and type $0$ particles) can be found as Figure 8 in \cite{GNS:MixingOpen}. Note that by construction, the process $(\chi^{\star}_t)_{t \geq 0}$ is supported on $\mathfrak{A}=\bigcup_{m \in \Z} \mathfrak{A}_m$, with $\mathfrak{A}_m$ from \eqref{def:BlockingStateSpace},  and thus has almost surely a finite left-most particle $(L^{\star}_t)_{t \geq 0}$ and right-most empty site $(R^{\star}_t)_{t \geq 0}$.

 \begin{lemma}\label{lem:DiminishedMultiASEP}
Let $q$ satisfy \eqref{def:TriplePointScaling} for some $\kappa \in [0,\frac{1}{2}]$ and $\psi>0$. Suppose that $\xi_0$ contains at most $y$ many second class particles, and hence $\chi_0$ at most $y$ many second class particles of types $1,2,3$. Then there exist constants $c_0,C_0,C_1>0$ such that for all $x \geq C_1 \log(N)$, and all $N$ large enough,
\begin{equation}\label{eq:DiminishedMulti}
 \P \left(  L^{\star}_{t},R^{\star}_{t} \in \left[ -xN^{\kappa}-y,  xN^{\kappa}\right] \,  \forall 0 \leq t \leq N^3  \right) \geq 1- C_0\exp(-c_0 x) .
\end{equation}
 \end{lemma}
 \begin{proof}
 Recall the partial order $\succeq_{\h}$ from \eqref{eq:PartialHeight} for a simple exclusion process on $\mathfrak{A}$. Verifying the marginal transition rates, we observe that the process $(\chi^{\star}_t)_{t \geq 0}$ has the law of an ASEP on $\Z$ with censoring, where the rightmost empty site in $\chi^{\star}_s$ is replaced by a first class particle whenever a second class particle of type $1,2,3$ exits from $(\chi_t)_{t\geq 0}$ from site $N$ at time $s$. An edge $e$ is in the censoring scheme $\mathcal{C}$ for $(\chi^{\star}_t)_{t \geq 0}$ at time $t$ if and only if it was merged in $\chi_t$ in the deletion step, or if one of its endpoints is occupied by a particle which is not present in $\chi_t$, and thus was only added in the construction when extending the configuration to $\Z$. To see that $\mathcal{C}$ is a genuine censoring scheme, observe that $(\hat{\chi}_t)_{t \geq 0}$ with
 \begin{equation}
\hat{\chi}_t(x) := \begin{cases} 1 &\text{ if } \chi_t(x)=1 \\
2 &\text{ if } \chi_t(x) \in \{ 2_1, 2_2, 2_3, 2_4,2_5  \} \\
0 &\text{ if } \chi_t(x) \in \{ 2_0, 2_{-1} , 0\}
 \end{cases}
 \end{equation} for all $x\in \lbr N \rbr$ and $t\geq 0$ is again a continuous-time Markov chain, noting in particular that an update of a $\{2_0,2_1\}$ pair to $\{2_{-1},2_2\}$ in the definition of $(\chi_t)_{t \geq 0}$ is consistent with the projection $(\hat{\chi}_t)_{t \geq 0}$. Recall the configurations $\vartheta_{m}$ from \eqref{def:GroundStates}. Let $(\eta^{-y}_{t})_{t \geq 0}$ and $(\eta^{0}_{t})_{t \geq 0}$ be two ASEPs on $\mathfrak{A}_{-y}$ and $\mathfrak{A}_{0}$, started from $\vartheta_0$ and $\vartheta_{-y}$, and using the same censoring scheme $\mathcal{C}$ as $(\chi^{\star}_t)_{t \geq 0}$, respectively.
Observe that since $\chi^{\star}_0=\vartheta_0$, and $\chi_0$ contains at most $y$ many type $1,2,3$ particles by our assumptions, the basic coupling $\mathbf{P}$ ensures that
\begin{equation}\label{eq:DominationASEP}
\mathbf{P}( \eta^{-y}_{t} \succeq_{\h} \chi^{\star}_t \succeq_{\h} \eta^{0}_t \, \forall t \geq 0) = 1 .
\end{equation} The result now follows from Lemma \ref{lem:BlockingMeasureMaximum} applied to the processes $(\eta^{-y}_{t})_{t \geq 0}$ on $\mathfrak{A}_{-y}$ and $(\eta^{0}_{t})_{t \geq 0}$ on $\mathfrak{A}_0$ together with \eqref{eq:DominationASEP}.
 \end{proof}

\begin{remark}\label{rem:Diminished}
In the same way as in Definition \ref{def:PartiallyDiminished}, we can define a second diminished process $(\bar{\chi}^{\star}_t)_{t \geq 0}$, swapping the roles of type $4,5$ and type $-1,0$ second class particles. More precisely, let $(w_t)_{t \geq 0}$ be a family of $\{0,1\}$-valued vectors, denoting the second class particles of types $-1$ to~$3$ which have left the segment at the site $N$ by time~$t$. For all $t\geq 0$, we assign a configuration $\chi^{\star}_t=\chi^{\star}_t(v_t) \in \{0,1\}^\Z$ by first deleting all vertices in $\chi_t$ which are empty, contain a first class particle or a type $4$ or type $5$ second class particle.
We then concatenate the vector $w_t$ at the right-hand side of the diminished segment and turn all second class particles to first class particles if they are of type $1,2$ or $3$ and into empty sites if they are of type $-1$ or $0$. Finally, we extend the segment to a configuration $\bar{\chi}_t^{\star} \in \{0,1\}^\Z$ in the same way as in Definition~\ref{def:PartiallyDiminished} to ensure that $(\bar{\chi}^{\star}_t)_{t \geq 0}$ can be interpreted as an ASEP on $\Z$ with censoring,  where whenever a type $1,2$ or $3$ second class particles leaves at site $N$, the left-most particle is replaced by an empty site.
It is straightforward to verify that a version of Lemma~\ref{lem:DiminishedMultiASEP} continues to hold for the diminished process $(\bar{\chi}^{\star}_t)_{t \geq 0}$.
\end{remark}

\subsection{Exit time of second class particles}\label{sec:ExitSecondClassHigh}

We will now utilize the processes $(\chi_t)_{t \geq 0}$ and $(\chi^{\star}_t)_{t \geq 0}$ to establish Lemma \ref{lem:IterateMixingTimesHigh} and Lemma  \ref{lem:IterateMixingTimesHighCritical} on the open ASEP in the weakly high density phase.
The processes $(\eta^{1}_t)_{t \geq 0}$, $(\eta^{2}_t)_{t \geq 0}$ and $(\eta^{3}_t)_{t \geq 0}$ are defined with respect to parameters $(q,\alpha,\beta,\gamma,\delta)$ such that the corresponding boundary densities $\rho^{(j)}_{\Lup}$ and $\rho^{(j)}_{\Rup}$ for $j \in \lbr 3 \rbr$ satisfy
\begin{equation*}
\rho_{\Lup}^{(j)}=\rho_{\Rup}^{(j)} = \rho_n = \frac{1}{2} + 2^{-n}
\end{equation*} with some $n \in \lbr \frac{1}{2} \log_2(N)\rbr$. We denote by $\succeq_{\c}$ the componentwise ordering and assume that 
\begin{equation}\label{eq:Condi12}
\eta_0^{1}  \preceq_{\c} \eta_0^{\textup{up}} \sim \textup{Ber}_N\Big( \frac{1}{2} + 2^{-(n-1)} \Big) \quad \text{ and } \quad
\eta_0^{2} \succeq_{\c} \eta_0^{\textup{low}} \sim \textup{Ber}_N\Big( \frac{1}{2} - 2^{-(n-1)} \Big) ,
\end{equation}
as well as that $\eta^{1}_0 \succeq_{\c} \eta^{2}_0$. For the process $(\eta^{4}_t)_{t \geq 0}$, we will set in the following $\alpha^{\prime}=\alpha$ and $\gamma^{\prime}=\gamma$, and hence $\rho_{\Lup}^{(4)}=\rho_{\Lup}^{(1)}$, while we choose $\beta^{\prime}>\beta$ and $\delta^{\prime}=\delta$ such that $\rho_{\Rup}^{(4)}$ satisfies
\begin{equation*}
\rho_{\Rup}^{(4)} = \frac{1}{2} + 2^{-(n+1)} < \rho_{\Rup}^{(1)} .
\end{equation*}
Note that $(\eta^{4}_t)_{t \geq 0}$ is by construction an open ASEP in the fan region of the high density phase, and that $(\chi_t)_{t \geq 0}$ contains no second class particles of types $0$ or $-1$. Observe that due to Lemma~\ref{lem:Bernoulli}, we can choose the initial configurations of $(\eta^{3}_t)_{t \geq 0}$ and $(\eta^{4}_t)_{t \geq 0}$ such that
\begin{equation}\label{eq:ConstructionZeta}
\mathbf{P}\left( \eta^{3}_t  \succeq_\c  \eta^{4}_t  \, \forall t \geq 0\right) = 1  .
\end{equation}
In particular, all second class particles in $(\zeta_t)_{t \geq 0}$ are of type $\Bup$.
We start with the case where $\kappa<\frac{1}{2}$ for $q$ from \eqref{def:TriplePointScaling}, and study the time $\texit$
\begin{equation}\label{def:ExitTimes}
\texit := \inf\{ t \geq 0 \colon \xi_t(x) \in \{0,1\} \, \forall x\in \lbr N \rbr \} 
\end{equation}
when all second class particles left the disagreement process $(\xi_t)_{t \geq 0}$ for $(\eta^{1}_t)_{t \geq 0}$ and $(\eta^{2}_t)_{t \geq 0}$.
\begin{lemma}\label{lem:ExittimesSecondClass} Let $q$ from \eqref{def:TriplePointScaling} with $\kappa<\frac{1}{2}$ and $\psi>0$.
Let $n \in \lbr \frac{1}{2} \log_2(N) \rbr$. Then there exist constants $c_0,C_0>0$ such that the exit time of second class particles for the disagreement process $(\xi_t)_{t \geq 0}$ satisfies
\begin{equation}
\mathbf{P}\left(\texit \geq  2^{n+6} N(1-q)^{-1}\right) \leq C_0\exp(-c_0 \min(2^{-n}N^{1/2},N^{\expont})^{1/2})
\end{equation} for all $N$ sufficiently large, where we recall $\expont$ from \eqref{def:KappaEpsilon}.
\end{lemma}
\begin{proof}
We claim that under the basic coupling $\mathbf{P}$, there exist constants  $c_1,C_1>0$ such that for any $t \geq 0$
\begin{equation}\label{eq:Particles1}
\mathbf{P}\left( \sum_{x \in \lbr N \rbr} \mathds{1}_{\{ \eta^{1}_t(x) \neq  \eta^{2}_t(x)\}} \geq 4 N 2^{-n} \right) \leq C_1 \exp\big(-c_1 2^{-n}N^{1/2}\big) 
\end{equation} for all $N$ large enough. To see this, note that  second class particles can only exit the segment in $(\xi_t)_{t \geq 0}$, but not enter. Assumption \eqref{eq:Condi12} for time $0$ together with a standard Chernoff bound gives the claim.
Similarly, using Lemma \ref{lem:Bernoulli} for the marginals of the stationary process $(\zeta_t)_{t \geq 0}$, there exist constants $c_2,C_2>0$ such that for any $t \geq 0$
\begin{equation}\label{eq:Particles2}
\mathbf{P}\left( \sum_{x \in \lbr N \rbr} \mathds{1}_{\{ \eta^{3}_t(x) \neq  \eta^{4}_t(x)\}} \geq 2 N 2^{-n} \right) \leq C_2 \exp\big(-c_2 2^{-n}N^{1/2} \big)
\end{equation} 
for all $N$ large enough. For all $t\geq 0$, let $M_4(t)$ denote the number of second class particles in the process $(\zeta_t)_{t \geq 0}$, which have left at site $1$ by time $t$. Note that this agrees with the number of type $4$ second class particles in the process $(\chi_t)_{t \geq 0}$, which have exited the segment at site $1$ until time $t$. Recall $\expont$ from \eqref{def:KappaEpsilon}. From Lemma~\ref{lem:CurrentHighLowFinal} (applied iteratively for times $[(i-1)2^{n+1}N(1-q)^{-1},i2^{n+1}N(1-q)^{-1}]$ with $i\in \lbr 2^5 \rbr$) and \eqref{eq:Particles2} for times $t=0$ and $t=2^{n+6}N(1-q)^{-1}$, we get that
\begin{equation}\label{eq:LowerDeviationApplied}
\mathbf{P}\left( M_{4}(2^{n+6}N(1-q)^{-1}) \geq  8 N 2^{-n} \right) \geq 1- C_3\exp\big(-c_3\min(2^{-n}N^{1/2},N^{\expont})^{1/2}\big)
\end{equation} for some constants $c_3,C_3>0$, and all $N$ large enough. Recall from Lemma \ref{lem:DiminishedMultiASEP} that with probability at least $1-C_4 \exp(-c_4 N^{\expont})$ for some constants $C_4,c_4>0$, every type $1,2,3$ second class particle in $(\chi_t)_{t \geq 0}$ has until time $N^{3}$ at most $N^{\kappa+\expont}\leq N^{1/2}$ many second class particles of type $4$ or type $5$ to its left, counting also all second class particles which have exited the segment.
Since on the event in \eqref{eq:Particles1}, the process $(\chi_t)_{t \geq 0}$ can create a total of at most $4N2^{-n}$ type $5$ second class particles by coalescence of type $3$ and type $4$ second class particles, together with \eqref{eq:Particles2} and  \eqref{eq:LowerDeviationApplied}, this yields the desired bound on the exit time of type $1,2,3$ second class particles in $(\chi_t)_{t \geq 0}$.
\end{proof}

Similarly, we have the following bound when $\kappa=\frac{1}{2}$. We will only provide the necessary modifications in the proof of Lemma \ref{lem:ExittimesSecondClass}. Write in the following $\overrightarrow{1}$ and $\overrightarrow{0}$ for the all full and all empty configurations.

\begin{lemma}\label{lem:ExittimesSecondClassCritical}
Let $q$ from \eqref{def:TriplePointScaling} with $\kappa=\frac{1}{2}$ and $\psi>0$. Assume that $n$ and $n-1$ satisfy \eqref{eq:nassumption}. Then there exist constants $C_0,C_1>0$ such that the exit time of second class particles for the disagreement process $(\xi_t)_{t \geq 0}$ satisfies
\begin{equation}\label{eq:FirstCritical}
\mathbf{P}\left(\texit \geq C_1 2^{n} N^{3/2}\right) \leq C_0 N^{-9}
\end{equation} for all $N$ sufficiently large. Moreover, when $n = \lceil \log_2(c_1^{-1}\log(N)) \rceil $ for the constant $c_1>0$ from Lemma \ref{lem:ModDeviationMaxSecondClass}, and $\eta_0^{1} = \overrightarrow{1}$ as well as $\eta_0^{2} = \overrightarrow{0}$, then we get that for some constants $C_2,C_3>0$ and all $N$ large enough
\begin{equation}\label{eq:SecondCritical}
\mathbf{P}\left(\texit \geq C_3 \log(N)N^{3/2} \log(N) \right) \leq C_2 N^{-9} .
\end{equation}
\end{lemma}
\begin{proof}
The first part \eqref{eq:FirstCritical} follows from the same arguments as Lemma~\ref{lem:ExittimesSecondClass} using Lemma~\ref{lem:CurrentCriticalKappaHigh} under assumption \eqref{eq:nassumption} instead of Lemma~\ref{lem:CurrentHighLowFinal} for a lower bound on number of second class particles which have left in $(\zeta_t)_{t \geq 0}$ the segment until time $C_1 2^{n} N^{3/2}$ at site $1$. Here, we choose $C_1>0$ sufficiently large so that the right-hand side in \eqref{eq:LowerDeviationApplied} is at least $1-N^{-9}$ and the constant $c_2$ in \eqref{eq:nassumption} sufficiently small such that $2^{n} \geq C_1 \log(N)$ holds.
For the second statement \eqref{eq:SecondCritical}, for $n = \lceil \log_2(c_1^{-1}\log(N)) \rceil $, we iterate the lower bound in \eqref{eq:LowerDeviationApplied} $2^{n}$ many times to ensure that at least $8N$ many second class particles of type $4$ have exited at site $1$ in $(\chi_t)_{t \geq 0}$ with probability at least $1-N^{-8}$. Noting that the process $(\chi_t)_{t \geq 0}$ contains at most $N$ many second class particles of types $1,2,3$, and hence at most $N$ many second class particles of type $5$ are created, we conclude.
\end{proof}

Using Lemma \ref{lem:ExittimesSecondClass} and Lemma \ref{lem:ExittimesSecondClassCritical}, we obtain the desired bounds on the coupling time, and hence the mixing time of the open ASEP in the weakly high density phase.

\begin{proof}[Proof of Lemma \ref{lem:IterateMixingTimesHigh} and Lemma \ref{lem:IterateMixingTimesHighCritical}]
Consider open ASEPs $(\tilde{\eta}^{(j)}_t)_{t \geq 0}$ for $j \in \lbr 4 \rbr$ with respect to the same parameters $\gamma,\delta,q$, and
\begin{equation*}
\alpha^{(3)} \geq \alpha^{(1)} =  \alpha^{(2)}  \geq \alpha^{(4)} \quad \text{ and } \quad  \beta^{(4)} \geq \beta^{(1)} = \beta^{(2)} \geq \beta^{(3)}
\end{equation*} such that the respective effective densities satisfy
\begin{equation*}
\rho_{\Lup}^{(1)}=\rho_{\Rup}^{(1)}=\frac{1}{2}+2^{-n}, \quad \rho_{\Lup}^{(3)}=\rho_{\Rup}^{(3)}=\frac{1}{2}+2^{-(n-1)}, \quad
\rho_{\Lup}^{(4)}=\rho_{\Rup}^{(4)}=\frac{1}{2}-2^{-(n-1)} .
\end{equation*}  We let $\tilde{\eta}^{(1)}_0=\tilde{\eta}^{(3)}_0=\overrightarrow{1}$ and  $\tilde{\eta}^{(2)}_0=\tilde{\eta}^{(4)}_0=\overrightarrow{0}$ almost surely, where we recall that $\overrightarrow{1}$ and $\overrightarrow{0}$ denote the all full and all empty configuration. Note that under the basic coupling of the open ASEP for different boundary conditions -- see for example Lemma 2.1 in \cite{GNS:MixingOpen} -- we get that
\begin{equation*}
\mathbf{P}\left( \tilde{\eta}^{(3)}_t \succeq_{\c} \tilde{\eta}^{(1)}_t  \succeq_{\c}  \tilde{\eta}^{(2)}_t  \succeq_{\c}  \tilde{\eta}^{(4)}_t \, \forall t \geq 0 \right) = 1 .
\end{equation*}
Recalling Remark \ref{rem:MixingTimesWeaklyLow}, as well as $\eta_0^{\textup{up}}$ and $\eta_0^{\textup{low}}$ from \eqref{eq:Condi12}, there exists a coupling $\tilde{\mathbf{P}}$ such that at time $T=t^{N,n-1}_{\cou}(\varepsilon)$,  we get that for all $N$ large enough
\begin{equation*}
\tilde{\mathbf{P}}\left( \left\{ \tilde{\eta}^{(3)}_{T} = \eta_0^{\textup{up}}\right\} \cap \left\{ \tilde{\eta}^{(4)}_{T} =  \eta_0^{\textup{low}} \right\}\right) \geq 1 - 2 \varepsilon .
\end{equation*}  When $\kappa<\frac{1}{2}$, we apply Lemma \ref{lem:ExittimesSecondClass} for $\eta_0^1=\tilde{\eta}^{(1)}_{T}$ and $\eta_0^2=\tilde{\eta}^{(2)}_{T}$ to conclude. Note that this includes the special case $n=1$, which yields \eqref{eq:MixingBaseCase}. For $\kappa=\frac{1}{2}$, we apply Lemma \ref{lem:ExittimesSecondClass} for $\eta_0^1=\tilde{\eta}^{(1)}_{T}$ and $\eta_0^2=\tilde{\eta}^{(2)}_{T}$ whenever $n-1$ and $n$ satisfy \eqref{eq:Condi12}, and \eqref{eq:SecondCritical} otherwise, to obtain the statement~\eqref{eq:MixingBaseCaseCritical}.
\end{proof}

\section{Upper bounds on the mixing time at the triple point}\label{sec:MixingTriple}

Using a similar strategy as in Section \ref{sec:MixingHighLow} for the mixing time of the open ASEP in the weakly high and low density phase, we  now show the upper bounds on the mixing times in Theorem \ref{thm:MixingTimesTriple} and Theorem \ref{thm:MixingTimesKPZ}. We start by recalling the setup of Section~\ref{sec:MultiSpeciesOpenASEP}, adapted for the open ASEP at  the triple point. \\

Recall the process $(\chi_t)_{t \geq 0}$ from Definition \ref{def:PartiallyOrdered} with respect to open ASEPs $(\eta^{j}_t)_{t \geq 0}$ for $j \in \lbr 4 \rbr$, which we now specify. For $q$ from \eqref{def:TriplePointScaling} with some $\kappa \in [0,\frac{1}{2}]$ and $\psi>0$, assume that the boundary parameters $\alpha,\beta,\gamma,\delta$ satisfy the conditions \eqref{def:LiggettTypeCondition} and \eqref{eq:ScalingCondition} with respect to some constants $\tilde{A},\tilde{C} \in \R$. Recall $D$ and $B$ from \eqref{def:d} and \eqref{def:b}, respectively, where
\begin{equation}\label{eq:BDRelation}
B=-q e^{-\tilde B N^{-1/2} + o(N^{-1/2})} \quad \text{ and } \quad D=-q e^{-\tilde D N^{-1/2} + o(N^{-1/2})}
\end{equation}
for some constants $\tilde{B},\tilde{D} \in \R$ when $\kappa<\frac{1}{2}$ and $\tilde{B}=\tilde{D} > - \psi$ for $\kappa=\frac{1}{2}$. Note that for $\gamma,\delta>0$
\begin{equation}\label{eq:ScalingRelationABCD}
 CD = - \frac{\alpha}{\gamma} \quad \text{ and } \quad  AB = - \frac{\beta}{\delta} .
\end{equation}
The processes $(\eta^{1}_t)_{t \geq 0}$, $(\eta^{2}_t)_{t \geq 0}$ and $(\eta^{3}_t)_{t \geq 0}$ are defined with respect to parameters $(q,\alpha,\beta,\gamma,\delta)$.
Moreover, we set
\begin{equation}\label{eq:CondiNew}
\eta_0^{1} \preceq_{\c} \eta_0^{\textup{up}} \sim \textup{Ber}_N\Big( \frac{1}{2} + 2^{-k} \Big)  \quad \text{ and } \quad
\eta_0^{2} \succeq_{\c} \eta_0^{\textup{low}} \sim \textup{Ber}_N\Big( \frac{1}{2} - 2^{-\ell} \Big)
\end{equation}
for some $k,\ell \in \N$ specified later on,
and assert that $\eta^{1}_0(x) \geq \eta^{2}_0(x)$ for all $x\in \lbr N \rbr$.
For the process $(\eta^{4}_t)_{t \geq 0}$, we will set in the following $\alpha^{\prime}>\alpha$ and $\beta^{\prime}>\beta$, while  $\gamma^{\prime} \leq \gamma$ and $\delta^{\prime}\leq \delta$. In order to choose $\alpha^{\prime},\beta^{\prime},\gamma^{\prime},\delta^{\prime}$, recall the function $F$ from \eqref{def:F} for $\kappa<\frac{1}{2}$, and the function $\tilde{F}$ from \eqref{def:Ftilde} for $\kappa=\frac{1}{2}$. We will choose  $\alpha^{\prime}$ and $\beta^{\prime}$ sufficiently large such that
\begin{equation}\label{eq:ScalingConditionNew}
A' = A(\beta^{\prime},\delta^{\prime},q)= \exp(-\tilde{A}^{\prime}N^{-1/2}) \quad \text{and} \quad C' = C(\alpha^{\prime},\gamma^{\prime},q)=\exp(-\tilde{C}^{\prime}N^{-1/2})
\end{equation} with respect to some constants $\tilde{A}^{\prime},\tilde{C}^{\prime}$ satisfying
\begin{equation}\label{eq:DifferenceCondition}
\begin{split}
F(\tilde{A}^{\prime},\tilde{C}^{\prime}) - F(\tilde{A},\tilde{C}) \geq c_{\kappa,\tilde{A},\tilde{C}}>0 & \ \text{ if } \kappa<\frac{1}{2} \\
\tilde{F}(\tilde{A}^{\prime},\tilde{C}^{\prime}) - \tilde{F}(\tilde{A},\tilde{C}) \geq c_{\kappa,\tilde{A},\tilde{C}}>0 & \ \text{ if } \kappa=\frac{1}{2}
\end{split}
\end{equation} for some constant $c_{\kappa,\tilde{A},\tilde{C}}>0$, depending only on $\kappa \in [0,\frac{1}{2}]$, $\tilde{A}$ and $\tilde{C}$. At the same time, we require that \eqref{def:LiggettsCondition} holds with respect to $\alpha^{\prime},\beta^{\prime},\gamma^{\prime},\delta^{\prime},q$. Note that we can always find such $\alpha^{\prime}$ and $\beta^{\prime}$ and a strictly positive constant $c_{\kappa,\tilde{A},\tilde{C}}$ due to Lemma~\ref{lemMonotonicity} for $\kappa<\frac{1}{2}$ and Lemma~\ref{lemMonotonicityBis} for $\kappa=\frac{1}{2}$, and suitable $\gamma^{\prime},\delta^{\prime}$ due to \eqref{eq:ScalingRelationABCD} for $\kappa=\frac{1}{2}$.
The processes $(\eta^{3}_t)_{t \geq 0}$ and $(\eta^{4}_t)_{t \geq 0}$ are both chosen to be stationary, satisfying
\begin{equation}\label{eq:ConstructionZetaNew}
\mathbf{P}\left( \eta^{3}_t  \succeq_\c  \eta^{4}_t  \, \forall  t \geq 0\right) = 1
\end{equation} with respect to the basic coupling $\mathbf{P}$.

\subsection{Exit time of second class particles}\label{sec:ExitSecondClass}

In the following, we adapt the strategy from  Section~\ref{sec:ExitSecondClassHigh} for mixing times in the weakly high and low density phase. Recall from \eqref{def:ExitTimes} that $\texit$ denotes the exit time of the second class particles in the disagreement process $(\xi_t)_{t \geq 0}$ between $(\eta^{1}_t)_{t \geq 0}$ and $(\eta^{2}_t)_{t \geq 0}$ with initial conditions  \eqref{eq:CondiNew}. We have the following result on the exit time of second class particles at the triple point when $\kappa<\frac{1}{2}$.

\begin{lemma}\label{lem:ExitMax}
Let $q$ from \eqref{def:TriplePointScaling} with $\kappa<\frac{1}{2}$ and $\psi>0$. Assume that the parameters $\alpha,\beta,\gamma,\delta$ satisfy the conditions \eqref{def:LiggettTypeCondition} and \eqref{eq:ScalingCondition} for some constants $\tilde{A},\tilde{C}$, and that there exist constants $C_{\textup{up}},C_{\textup{low}}>0$ such that $k,\ell$ from \eqref{eq:CondiNew} satisfy
\begin{equation}
\frac{C_{\textup{up}}}{\sqrt{N}} \geq 2^{-k}  \geq -2^{-\ell} \geq - \frac{C_{\textup{low}}}{\sqrt{N}}
\end{equation} for all $N$ large enough. Then there exist constants $C_0,c_0>0$, depending only $\tilde{A},\tilde{C},C_{\textup{up}},C_{\textup{low}}$, such that we have
\begin{equation}
\mathbf{P}\left(\texit \leq  C_0 N^{3/2}(1-q)^{-1}\right) \geq c_0
\end{equation} for all $N$ sufficiently large.
%
%
\end{lemma}

\begin{proof}
We follow a similar strategy as for Lemma \ref{lem:ExittimesSecondClass}.
Note that under the basic coupling $\mathbf{P}$, there exist constants  $c_1,C_1>0$ such that for any $t \geq 0$ and $m\in \N$
\begin{equation}\label{eq:Particles1New}
\mathbf{P}\left( \sum_{x \in \lbr N \rbr} \mathds{1}_{ \{ \eta^{1}_t(x) \neq  \eta^{2}_t(x) \}} \geq \sqrt{m} (C_{\textup{up}}+C_{\textup{low}})\sqrt{N} \right) \leq C_1 \exp(- c_1 \sqrt{m})
\end{equation} for all $N$ large enough. Similarly, using Lemma \ref{lem:Bernoulli} for the marginals of the stationary process $(\zeta_t)_{t \geq 0}$, and assumption \eqref{eq:ConstructionZetaNew}, there exist constants $c_2,C_2,C_3>0$ such that for any $t \geq 0$ and $m\in \N$
\begin{equation}\label{eq:Particles2New}
\mathbf{P}\left( \sum_{x \in \lbr N \rbr} \mathds{1}_{ \{ \eta^{3}_t(x) \neq  \eta^{4}_t(x) \} } \geq \sqrt{m} C_3 \sqrt{N} \right) \leq C_2 \exp(- c_2 \sqrt{m})
\end{equation} for all $N$ large enough. For the disagreement process $(\zeta_t)_{t \geq 0}$ between $(\eta^{3}_t)_{t \geq 0}$ and $(\eta^{4}_t)_{t \geq 0}$, let $\mathcal{J}^{\Aup}_{t}$ denote the number of type $\Aup$ second class particles which have exited the segment at site $1$ by time $t\geq 0$. Similarly, let $\mathcal{J}^{\Bup}_{t}$ denote the number of type $\Bup$ second class particles which have exited the segment at site $N$ by time $t\geq 0$.
Set in the following
\begin{equation*}
T=T(m)= m N^{3/2}(1-q)^{-1} .
\end{equation*}
Let $(\mathcal{J}^{3}_{t})_{t \geq 0}$ and $(\mathcal{J}^{4}_{t})_{t \geq 0}$ denote the current in $(\eta^{3}_t)_{t \geq 0}$ and $(\eta^{4}_t)_{t \geq 0}$ at site $1$, respectively. Since in the disagreement process $(\zeta_t)_{t \geq 0}$ type $\Aup$ particles are only created at site $N$ while type $\Bup$ particles are only created at site $1$, we observe that
\begin{equation}\label{eq:Helper1}
  \mathcal{J}^{4}_{t} - \mathcal{J}^{3}_{t} = \mathcal{J}^{\Aup}_{t} + \mathcal{J}^{\Bup}_{t} + \sum_{x \in \lbr N \rbr} \mathds{1}_{\{ \zeta_t(x) = \Bup \}}  - \mathds{1}_{\{ \zeta_0(x) = \Bup \}}
\end{equation} almost surely for all $t\geq 0$.
From Lemma \ref{lem:CurrentMaxPhase}, with our choice of the boundary parameters according to \eqref{eq:DifferenceCondition}, and \eqref{eq:Particles2New} at times $t=0$ and $t=T$, we see that there exist constants $c_4,c_5>0$ such that for all $m\in \N$
\begin{equation}\label{eq:Helper2}
\mathbf{P}\left( \mathcal{J}^{\Aup}_{T} + \mathcal{J}^{\Bup}_{T} \geq  m c_4 \sqrt{N} \right) \geq \frac{c_5}{m^2} .
\end{equation}
Let $M_0(s)$ denote the number of type $0$ second class particles which have left in $(\chi_t)_{t \geq 0}$ until time $s\geq 0$ at site $N$. Similarly, let $M_4(s)$ denote the number of type $4$ second class particles which have left in $(\chi_t)_{t \geq 0}$ until time $s\geq 0$ at site $1$. We define the events
\begin{align*}
\mathcal{A}^{m}_1 := \left\{ M_4(T) \geq \frac{c_4}{4} m \sqrt{N} \right\} \quad \text{ and } \quad \mathcal{A}^{m}_2 &:= \left\{ M_0(T) \geq \frac{c_4}{4} m \sqrt{N} \right\} , 
\end{align*}
and recall that only type $4$ second class particles are created at site $N$ (paired with a type $\Aup$ second class particle in $(\zeta_t)_{t \geq 0}$) while only type $0$ second class particles are created at site $1$ (paired with a type $\Bup$ second class particle in $(\zeta_t)_{t \geq 0}$) in the process $(\chi_t)_{t \geq 0}$.
Since the events in \eqref{eq:Particles1New} and \eqref{eq:Particles2New} ensure that there are at most $\sqrt{m}(C_3+C_{\textup{up}}+C_{\textup{low}})\sqrt{N}$ many second class particles of types $\Aup$ and $\Bup$ in the segment at times $0$ and $T$, we get from \eqref{eq:Helper1} and \eqref{eq:Helper2} that there exists some $m_0 \in \N$ and $c_6>0$ such that
\begin{equation}\label{eq:MaxBoundAm}
\begin{split}
\max\big(\mathbf{P}(\mathcal{A}^{m}_1),\mathbf{P}(\mathcal{A}^{m}_2)\big) &= \max\Big(\mathbf{P}\big(M_4(T) \geq \frac{c_4}{4} m \sqrt{N}\big),\mathbf{P}\big(M_0(T) \geq \frac{c_4}{4} m \sqrt{N}\big) \Big) \\
&\geq \frac{1}{2}\mathbf{P}\left( \mathcal{J}^{\Aup}_{T} + \mathcal{J}^{\Bup}_{T} \geq  m c_4 \sqrt{N} \right)  - 2 C_2 \exp(-c_2 \sqrt{m}) \geq  \frac{c_6}{4m^2}
\end{split}
\end{equation} for all $m \geq m_0$, and all $N$ large enough. The same arguments as in Lemma \ref{lem:ExittimesSecondClass}, using the diminished process $(\chi^{\star}_t)_{t \geq 0}$, guarantee that with probability at least $1-C_7\exp(-c_7 N^{\expont})$ for some constants $c_7,C_7>0$ and $\expont$ from \eqref{def:KappaEpsilon}, there are at most $N^{\kappa+\expont}\leq N^{1/2}$ many second class particles of type $4$ to the left of any type $1,2,3$ second class particle at time $T$. Similarly, using the diminished process $(\bar{\chi}^{\star}_t)_{t \geq 0}$ from Remark~\ref{rem:Diminished}, we see that with probability at least $1-C_8\exp(-c_8 N^{\expont})$ for some constants $c_8,C_8>0$, there are at most $N^{\kappa+\expont}$ many second class particles of type $0$ to the right of any type $1,2,3$ second class particle at time~$T$. Hence, choosing $m = \max(8 c^{-1}_4,m_0)$, either of the events $\mathcal{A}^{m}_1$ and $\mathcal{A}^{m}_2$ implies that $\texit$ has occurred by time $T$ with probability tending to $1$ as $N \rightarrow \infty$, allowing us to conclude by the bounds in \eqref{eq:MaxBoundAm}.
\end{proof}

We have the following result when $\kappa=\frac{1}{2}$. Since we apply the same arguments as in the proof of Lemma~\ref{lem:ExittimesSecondClassCritical}, we only describe the necessary adjustments of the proof.

\begin{lemma}\label{lem:ExitMaxCritical}
Let $q$ from \eqref{def:TriplePointScaling} with $\kappa=\frac{1}{2}$ and $\psi>0$. Assume that the parameters $\alpha,\beta,\gamma,\delta$ satisfy the conditions \eqref{def:LiggettTypeCondition} and \eqref{eq:ScalingCondition} for some constants $\tilde{A},\tilde{C}$, and that there exist constants $C_{\textup{up}},C_{\textup{low}}>0$ such that $k,\ell$ from \eqref{eq:CondiNew} satisfy
\begin{equation}
\frac{C_{\textup{up}}\log(N)}{\sqrt{N}} \geq 2^{-k}  \geq -2^{-\ell} \geq - \frac{C_{\textup{low}}\log(N)}{\sqrt{N}}
\end{equation} for all $N$ large enough. Then there exist constants $C_0,c_0>0$, depending only $\tilde{A},\tilde{C},C_{\textup{up}},C_{\textup{low}}$, such that we have for all $N$ sufficiently large
\begin{equation}
\mathbf{P}\left(\texit \leq  C_0 N^2 \log(N) \right) \geq \frac{c_0}{\log^{2}(N)} .
\end{equation}

\end{lemma}
\begin{proof}
Note that by Lemma \ref{lem:Bernoulli} and a standard moderate deviation bound, there exists a constant $C_1>0$ such that
\begin{equation}\label{eq:Particles3New}
\mathbf{P}\Big( \sum_{x \in \lbr N \rbr} \mathds{1}_{ \{ \eta^{1}_t(x) \neq  \eta^{2}_t(x) \}} + \mathds{1}_{ \{ \eta^{3}_t(x) \neq  \eta^{4}_t(x) \} }  \geq C_1 \log(N)\sqrt{N} \Big) \leq N^{-9}
\end{equation} for all $N$ large enough. Set $T=mN^{2}\log(N)$ for some constant $m\in \N$ chosen later on.
 Recall that $M_4(s)$ and $M_0(s)$ denote the number of type $4$ and type $0$ second class particles, which have exited in $(\chi_t)_{t \geq 0}$ by time $s$ at site $1$ and site $N$, respectively. The same arguments as in Lemma \ref{lem:ExitMax} with a lower bound on the current of second class particles in $(\zeta_t)_{t \geq 0}$ by Lemma \ref{lem:CurrentCriticalKappaMax} yield that
\begin{equation}\label{eq:LowerM04}
\max\Big( \mathbf{P} \big(M_4(T) \geq  m c_2 \log(N)\sqrt{N} \big),  \mathbf{P} \big(M_0(T) \geq  m c_2 \log(N)\sqrt{N} \big) \Big)  \geq \frac{c_3}{\log^{2}(N)m^2}
\end{equation} for some constants $c_2,c_3>0$, for all $m$ fixed, and all $N$ large enough. Moreover, by Lemma~\ref{lem:DiminishedMultiASEP} and Remark~\ref{rem:Diminished}, there exists a constant $C_4>0$ such that with probability at least $1-N^{-8}$ at most $C_4 \log(N)\sqrt{N}$ many type $4$ second class particles are to the left of any type $1,2,3$ second class particle in $\chi_T$, while at most $C_4 \log(N)\sqrt{N}$ many type $0$ second class particles are to the right of any type $1,2,3$ second class particle in $\chi_T$. Choosing now the constant $m$ in \eqref{eq:LowerM04} sufficiently large, we conclude.
\end{proof}

\subsection{Proof of the upper bound on the mixing time}

We have now all tools to establish an upper bound on the mixing time of the open ASEP at the triple point. We start with the case where $q$  satisfies \eqref{def:TriplePointScaling} with $\kappa<\frac{1}{2}$.

\begin{proof}[Proof of the upper bound in Theorem \ref{thm:MixingTimesTriple}]

Let $C_{\textup{up}},C_{\textup{low}}>0$ be two constants such that
\begin{equation}\label{eq:OrderingDensitiesFinal}
\frac{1}{2}+\frac{C_{\textup{up}}}{\sqrt{N}} \geq \max\left( \frac{A}{A+1},\frac{1}{C+1}\right)  \geq \min\left( \frac{A}{A+1},\frac{1}{C+1}\right) \geq \frac{1}{2}-\frac{C_{\textup{low}}}{\sqrt{N}}
\end{equation} for all $N$ large enough, assuming without loss of generality that $\frac{C_{\textup{up}}}{\sqrt{N}}$ and $\frac{C_{\textup{low}}}{\sqrt{N}}$ are powers of $2$ (this is where we make the choice for $k$ and $\ell$ from \eqref{eq:CondiNew}). Recall the coupling time $t_{\cou}(\varepsilon)$ from \eqref{def:MixingTimeCoupling} for the open ASEP with respect to boundary parameters $\alpha,\beta,\gamma,\delta$ and consider four open ASEPs $(\tilde{\eta}^{(j)}_t)_{t \geq 0}$ for $j \in \lbr 4 \rbr$. We use the same parameters $\gamma,\delta,q$ as well as
\begin{equation*}
\alpha^{(3)} \geq \alpha^{(1)} =  \alpha^{(2)}  \geq \alpha^{(4)} \quad \text{ and } \quad  \beta^{(4)} \geq \beta^{(1)} = \beta^{(2)} \geq \beta^{(3)}
\end{equation*} such that the respective effective densities satisfy
\begin{equation*}
\rho_{\Lup}^{(1)}=\frac{1}{C+1}, \quad \rho_{\Rup}^{(1)}=\frac{A}{A+1}, \quad \rho_{\Lup}^{(3)}=\rho_{\Rup}^{(3)}=\frac{1}{2}+\frac{C_{\textup{up}}}{\sqrt{N}}, \quad
\rho_{\Lup}^{(4)}=\rho_{\Rup}^{(4)}=\frac{1}{2}-\frac{C_{\textup{low}}}{\sqrt{N}}.
\end{equation*}
We let $\tilde{\eta}^{(1)}_0=\tilde{\eta}^{(3)}_0=\overrightarrow{1}$ and  $\tilde{\eta}^{(2)}_0=\tilde{\eta}^{(4)}_0=\overrightarrow{0}$ almost surely, where we recall that $\overrightarrow{1}$ and $\overrightarrow{0}$ denote the all full and all empty configuration, respectively. Under the basic coupling, we get
\begin{equation}\label{eq:CouplingMaxFinal}
\mathbf{P}\left( \tilde{\eta}^{(3)}_t \succeq_{\c} \tilde{\eta}^{(1)}_t  \succeq_{\c}  \tilde{\eta}^{(2)}_t  \succeq_{\c}  \tilde{\eta}^{(4)}_t \, \forall t \geq 0 \right) = 1 .
\end{equation}
Consider two (random) configurations
\begin{equation*}
\eta_0^{\textup{up}} \sim \textup{Ber}_N\Big(\frac{1}{2}+\frac{C_{\textup{up}}}{\sqrt{N}}\Big) \quad \text{ and }\quad \eta_0^{\textup{low}}\sim \textup{Ber}_N\Big(\frac{1}{2}-\frac{C_{\textup{low}}}{\sqrt{N}}\Big) .
\end{equation*}
 Then by Proposition \ref{pro:MixingTimesWeaklyHighLow}, with $2^{-n}=C_{\textup{up}}/\sqrt{N}$, increasing the constants $C_{\textup{up}}$ and $C_{\textup{low}}$ from \eqref{eq:OrderingDensitiesFinal} if necessary, for every $\varepsilon>0$, there exists some constant $C_0>0$ and a coupling $\tilde{\mathbf{P}}$ such that for all $N$ large enough
\begin{equation}\label{eq:COuplingAid}
\tilde{\mathbf{P}}\left(  \tilde{\eta}^{(3)}_{C_0N^{3/2}(1-q)^{-1}} = \eta_0^{\textup{up}} \text{ and }  \tilde{\eta}^{(4)}_{C_0N^{3/2}(1-q)^{-1}} =  \eta_0^{\textup{low}} \text{ with } \eta_0^{\textup{up}} \succeq_{\c} \eta_0^{\textup{low}} \right) \geq 1 - \varepsilon /2 .
\end{equation} Let $c_0>0$ be taken from Lemma \ref{lem:ExitMax} and set $\varepsilon=c_0/2$. Then by \eqref{eq:COuplingAid} together with Lemma~\ref{lem:ExitMax} and \eqref{eq:CouplingMaxFinal}, there exists a constant $C_2>0$ and a coupling $\bar{\mathbf{P}}$ such that
\begin{equation}\label{eq:CoalFinite}
\bar{\mathbf{P}}\left( \tilde{\eta}^{(1)}_{C_2N^{3/2}(1-q)^{-1}} = \tilde{\eta}^{(2)}_{C_2N^{3/2}(1-q)^{-1}} \right) \geq \varepsilon
\end{equation} for all $N$ large enough. Since $\varepsilon>0$ does not depend on $N$, iterating \eqref{eq:CoalFinite} order $\varepsilon^{-1}$ many times and using the coupling representation of the total-variation distance gives the desired result.
\end{proof}

Using a similar argument, we obtain an upper bound on mixing time of the open ASEP when $q$ satisfies \eqref{def:TriplePointScaling} with $\kappa=\frac{1}{2}$.

\begin{proof}[Proof of the upper bound in Theorem~\ref{thm:MixingTimesKPZ}]

We use the same setup as in the proof of the upper bound of Theorem~\ref{thm:MixingTimesTriple}, but assert that
\begin{equation*}
\rho_{\Lup}^{(3)}=\rho_{\Rup}^{(3)}=\frac{1}{2}+\frac{C_{\textup{up}}\log(N)}{\sqrt{N}} = \frac{1}{2} + 2^{-k}, \quad
\rho_{\Lup}^{(4)}=\rho_{\Rup}^{(4)}=\frac{1}{2}-\frac{C_{\textup{low}}\log(N)}{\sqrt{N}} = \frac{1}{2} - 2^{-\ell}.
\end{equation*} for some constants $C_{\textup{up}},C_{\textup{low}}>0$ and $k,\ell \in \N$, where we recall $k$ and $\ell$ from \eqref{eq:CondiNew}.
Then by Proposition \ref{pro:MixingTimesWeaklyHighLowCritical}, increasing the constants $C_{\textup{up}},C_{\textup{low}}$ if necessary, there exists a constant $C_0>0$ such that
\begin{equation*}
\mathbf{P}\left(  \tilde{\eta}^{(3)}_{C_0N^{2}\log(N)^{-1}} = \eta_0^{\textup{up}} \text{ and }  \tilde{\eta}^{(4)}_{C_0N^{2}\log(N)^{-1}} =  \eta_0^{\textup{low}} \text{ with } \eta_0^{\textup{up}} \succeq_{\c} \eta_0^{\textup{low}} \right) \geq 1 - 2N^{-8}
\end{equation*} for all $N$ large enough. Now  Lemma~\ref{lem:ExitMaxCritical} yields that for some constant $C_1>0$
\begin{equation}\label{eq:AlmostPositive}
\liminf_{N \rightarrow \infty}  \log^2(N)  \mathbf{P}\left( \tilde{\eta}^{(1)}_{C_1N^{2}\log(N)} =   \tilde{\eta}^{(2)}_{C_1N^{2}\log(N)} \right)> 0 .
\end{equation} Iterating  \eqref{eq:AlmostPositive} order $\log^{2}(N)$ many times gives the desired result.
\end{proof}

\begin{remark}\label{rem:MaxCurrent}
{In view of Remark~\ref{rem:ExtendedMaxCurrent}, we believe that the above arguments can be used to show that for $\kappa=0$ and any $A,C\leq 1$, the open ASEP under the basic coupling contains after a time of order $N^{3/2}$ at most order $N^{1/2}$ many second class particles. 
 In the proof of Theorem~\ref{thm:MixingTimesTriple}, we crucially rely on the fact that by decreasing the parameters $A,C$ from \eqref{eq:ScalingCondition} by order $N^{-1/2}$, this changes the expected current at order $N^{-1}$; see Proposition \ref{pro:CurrentMaxCurrent1}. This allows us bound the exit time of the remaining order $N^{1/2}$ many second class particles. Since we do not expect this to hold for general $A,C<1$, a different argument will be required to bound the exit time of the remaining order $N^{1/2}$ many second class particles. }
\end{remark}

\begin{remark}\label{rem:GeneralKappa}
Observe that we apply the same line of arguments for both regimes $\kappa<\frac{1}{2}$ and $\kappa=\frac{1}{2}$. It is therefore natural to conjecture that the upper bound of order $N^{3/2+\kappa}$ on the mixing time can be extended to the case $\kappa=\frac{1}{2}$.
However, this requires an improved bound on the variance of the current of the open ASEP of order $N$ until time $T \asymp N^{2}$, which could for example be achieved by an improved bound on the speed of second class particles in Proposition \ref{pro:SpeedOfDisagreement} when $\kappa=\frac{1}{2}$. We leave this to future work.
\end{remark}

\section{Lower bound on the mixing times}\label{sec:MixingTimesLowerBounds}

In this section, we provide lower bounds on the mixing times of the open ASEP at the triple point.
We start with an auxiliary result on the speed of second class particles for an ASEP on $\Z$.
Fix some $z>0$ (specified later on), and let $\bar{\pi}^{z}_N$ denote the product measure { on $\{1,2,\infty\}^{\Z}$ with marginals
\begin{equation}\label{eq:InitialLower1}
\begin{split}
    \bar{\pi}^{z}_N(\eta(x)=1) = \frac{1}{2} , \quad 
    \bar{\pi}^{z}_N(\eta(x)=2) = \frac{z}{\sqrt{N}} , \quad 
    \bar{\pi}^{z}_N(\eta(x)=\infty) = \frac{1}{2} - \frac{z}{\sqrt{N}} , 
\end{split}
\end{equation} for all $x\in [\frac{3}{8}N,\frac{5}{8}N]$ as well as
\begin{equation}\label{eq:InitialLower2}
\bar{\pi}^{z}_N(\eta(x)=1) = \bar{\pi}^{z}_N(\eta(x)=\infty) = \frac{1}{2}
\end{equation} for all $x\notin [\frac{3}{8}N,\frac{5}{8}N]$. }With a slight abuse of notation, we will also interpret $\bar{\pi}^{z}_N$ as a probability measure on the space $\{0,1,2\}^{N}$ for some $N \in \N$ {(identifying types $0$ and $\infty$)}. We have the following result on the location of second class particles in an ASEP on $\Z$ started from~$\bar{\pi}^{z}_N$.

\begin{lemma}\label{lem:ExtraSecondClass} Let $q$ satisfy \eqref{def:TriplePointScaling} with some $\kappa<\frac{1}{2}$. Consider an ASEP $(\eta_t^{\Z})_{t \geq 0}$ on $\Z$ with initial distribution $\bar{\pi}^{z}_N$. Then for every $\varepsilon,z>0$, there exists some constant $c_1>0$ such that for all $N$ large enough
\begin{equation}\label{eq:Extra1}
\P\left( \eta_t^{\Z}(x) \in \{ 0,1 \} \, \forall t \in [0,c_1(1-q)^{-1}N^{3/2}] \, , x \notin \big[\tfrac{1}{4}N,\tfrac{3}{4}N\big] \right) \geq 1- \varepsilon .
\end{equation}
Similarly, if $q$ satisfies \eqref{def:TriplePointScaling} with $\kappa=\frac{1}{2}$, then for every $\varepsilon,z>0$, there exists some constant $c_2>0$ such that for all $N$ large enough
\begin{equation}\label{eq:Extra2}
\P\left( \eta_t^{\Z}(x) \in \{ 0,1 \} \, \forall t \in [0,c_2 N^{2}\log^{-1}(N)] \, , x \notin \big[\tfrac{1}{4}N,\tfrac{3}{4}N\big] \right) \geq 1- \varepsilon .
\end{equation}
\end{lemma}
\begin{proof}
Let $\kappa<\frac{1}{2}$, and recall Lemma \ref{lem:MaxSecondClassOrder} on moderate deviations for a collection of second class particles. The claim in \eqref{eq:Extra1} follows now by the same arguments as for $\rho=1/2$ in Lemma \ref{lem:MaxSecondClassOrder}, shifting the lattice by $N/2$, and using the fact that $N^{\kappa+\expont+\frac{1}{2}} \leq N/8$ for all $N$ large enough. The bound in \eqref{eq:Extra2} for the case $\kappa=\frac{1}{2}$ is similar using Lemma \ref{lem:MaxSecondClassOrderCritical} instead of Lemma \ref{lem:MaxSecondClassOrder} for a moderate deviation estimate.
\end{proof}


In order to show the desired lower bounds on the mixing time, we require a simple observation on the number of particles under the stationary distribution $\mu_N$ of the open ASEP.

\begin{lemma}\label{lem:StationaryParticleDensity}
Let $q$ satisfy \eqref{def:TriplePointScaling} for some $\kappa \in [0,\frac{1}{2}]$ and assume that the boundary parameters satisfy the assumption \eqref{eq:ScalingCondition}. Then for every $\varepsilon>0$, there exists some constant $C_0>0$ such that for all $i\in \lbr 4 \rbr$
\begin{equation}
\mu_N\Bigg( \sum_{(i-1)N/4< x \leq i N/4 } \eta(x) \notin  \left[\tfrac12 N - C_0\sqrt{N},\tfrac12 N + C_0\sqrt{N} \right]\Bigg) \leq \varepsilon
\end{equation} with $N \in \N$ large enough.
\end{lemma}
\begin{proof}
This statement is immediate from Lemma \ref{lem:Bernoulli} which says that $\mu_N$ is stochastically dominated from below by a Bernoulli-$(\frac{1}{2}-m N^{-1/2})$-product measure and from above by a Bernoulli-$(\frac{1}{2}+m N^{-1/2})$-product measure with some constant $m>0$, together with a standard moderate deviation bound.
\end{proof}

We have now all tools to show the desired lower bounds on the mixing times.

\begin{proof}[Proof of the lower bound in Theorem \ref{thm:MixingTimesTriple} and Theorem \ref{thm:MixingTimesKPZ}]

We let $(\eta_t)_{t \geq 0}$ denote an open ASEP on $\{0,1\}^{N}$ with respect to boundary parameters $\alpha,\beta,\gamma,\delta$ and $q$,  which satisfy \eqref{eq:ScalingCondition}.
We define $(\eta^{(1)}_t)_{t \geq 0}$ to be an open ASEP on $\{0,1\}^{N}$ and $(\eta^{(2)}_t)_{t \geq 0}$ to be a disagreement processes on $\{0,1,2\}^{N}$ with the same parameters $\gamma,\delta,q$ as $(\eta_t)_{t \geq 0}$, but where $\alpha^{\prime},\beta^{\prime}$ are chosen such that the stationary distribution $\bar{\mu}_N$ of both processes satisfies
\begin{equation*}
\bar{\mu}_N \sim \textup{Ber}_N\Big(\frac{1}{2}\Big) .
\end{equation*} We assert that $\eta_0=\eta^{(1)}_0 \sim \bar{\mu}_N$ while $\eta^{(2)}_0 \sim \bar{\pi}^{z}_N$ for some $z>0$ specified later on. Moreover, we assume that $\eta^{(2)}_0 \succeq_{\c} \eta^{(1)}_0$,  extending the component-wise ordering $\succeq_{\textup{c}}$ with respect to the ordering $1 \succeq 2 \succeq 0$ on $\{0,1,2\}^{N}$.
 Let $\kappa<\frac{1}{2}$ and $\varepsilon>0$. Let $(\eta^{\Z}_t)_{t \geq 0}$ be an ASEP on $\Z$ with $\eta_0^{(2)}(x)=\eta_0^{\Z}(x)$ for all $x \in \lbr N \rbr$, and a Bernoulli-$\frac{1}{2}$-product measure for all $x \notin \lbr N \rbr$. Note that when projecting all second class particles to first class particles, the law of $(\eta^{\Z}_t)_{t \geq 0}$ is stochastically dominated from above by a Bernoulli-$\big(\frac{1}{2}+\frac{z}{\sqrt{N}}\big)$-product measure. Thus, combining Proposition \ref{pro:SpeedOfDisagreement} and Lemma~\ref{lem:SpeedOfPropagationCoupling}, there exists a constant $c_1>0$, depending only on $\varepsilon$ and $z$, such that under the basic coupling between $(\eta^{(2)}_t)_{t \geq 0}$ and $(\eta^{\Z}_t)_{t \geq 0}$, we get that for all $N$ large enough
\begin{equation}\label{eq:LowPre1}
\mathbf{P}\Big( \eta_t^{(2)}(x)=\eta_t^{\Z}(x) \, \forall t \in [0,c_1(1-q)^{-1}N^{3/2}] , x \in \left[\tfrac{1}{4}N,\tfrac{3}{4}N\right] \Big) \geq 1- \varepsilon .
\end{equation}
From \eqref{eq:LowPre1} and Lemma~\ref{lem:ExtraSecondClass} to bound the motion of second class particles within $(\eta^{\Z}_t)_{t \geq 0}$, we see that for every $\varepsilon,z>0$, there exists some constant $c_2>0$ such that for all $N$ large enough
\begin{equation}\label{eq:Low1}
\mathbf{P}\left( \eta_t^{(1)}(x)=\eta_t^{(2)}(x) \, \forall t \in [0,c_2(1-q)^{-1}N^{3/2}] ,  x \notin \left[\tfrac{1}{4}N,\tfrac{3}{4}N\right] \right) \geq 1-\varepsilon .
\end{equation}
When $\kappa=\frac{1}{2}$, a similar argument yields that for every $\varepsilon,z>0$, there exists some constant $c_3>0$ such that for all $N$ large enough
\begin{equation}\label{eq:Low2}
\mathbf{P}\left( \eta_t^{(1)}(x)=\eta_t^{(2)}(x) \, \forall t \in [0,c_3 N^{2}\log^{-1}(N)] , x \notin \left[\tfrac{1}{4}N,\tfrac{3}{4}N\right] \right) \geq 1-\varepsilon .
\end{equation}
Using again Proposition~\ref{pro:SpeedOfDisagreement} and Lemma~\ref{lem:SpeedOfPropagationCoupling} when $\kappa<\frac{1}{2}$, we see that for every $\varepsilon,z>0$, there exists some constant $c_4>0$ such that for all $N$ large enough
\begin{equation}\label{eq:Low3}
\mathbf{P}\left( \eta_t^{(1)}(x)=\eta_t(x) \, \forall t \in [0,c_4(1-q)^{-1}N^{3/2}] , x \in \left[\tfrac{1}{4}N,\tfrac{3}{4}N\right] \right) \geq 1- \varepsilon .
\end{equation}
Similarly, when $\kappa=\frac{1}{2}$, we get that for every $\varepsilon,z>0$, there exists some constant $c_5>0$ such that for all $N$ large enough
\begin{equation}\label{eq:Low4}
\mathbf{P}\left( \eta_t^{(1)}(x)=\eta_t(x) \, \forall t \in [0,c_5 N^{2}\log^{-1}(N)] , x \in \left[\tfrac{1}{4}N,\tfrac{3}{4}N\right] \right) \geq 1- \varepsilon .
\end{equation}
Combining \eqref{eq:Low1} and \eqref{eq:Low3} for $\kappa<\frac{1}{2}$, with probability at least $1-2\varepsilon$, no second class particle in $(\eta_t)_{t \geq 0}$ started from $\bar{\pi}_N$ has left the segment by time $\min(c_2,c_4)N^{3/2}(1-q)^{-1}$.
Similarly, combining \eqref{eq:Low2} and \eqref{eq:Low4} for $\kappa=\frac{1}{2}$, with probability at least $1-2\varepsilon$, no second class particle in $(\eta_t)_{t \geq 0}$ started from $\bar{\pi}_N$ has left the segment by time $\min(c_3,c_5)N^{2}\log^{-1}(N)$.
In view of Lemma \ref{lem:StationaryParticleDensity}, for any $\varepsilon>0$, first taking $z>0$ sufficiently large and then the constants $c_2,c_3,c_4,c_5>0$ sufficiently small, this yields the desired lower bounds on the mixing time in Theorem \ref{thm:MixingTimesTriple} and Theorem \ref{thm:MixingTimesKPZ}.
\end{proof}

\subsection*{Acknowledgment}
We are grateful to Amol Aggarwal and Ivan Corwin for valuable comments on moderate deviations for the current of the open ASEP.
Moreover, we thank Márton Balázs and Guillaume Barraquand for helpful discussions. This work was partly funded by the Deutsche Forschungsgemeinschaft (DFG, German Research Foundation) the CRC 1720 – 539309657 and under Germany’s Excellence Strategy - EXC 2047/1 – 390685813. D.S. is partially funded by the Packard Foundation via Amol Aggarwal's Packard Fellowships for Science and Engineering and by the Simons Foundation via Ivan Corwin's Investigator Award.

\bibliographystyle{imsart-number} 
\bibliography{ASEPTriple}   


\end{document}